\documentclass[1 [leqno,11pt]{amsart}
\usepackage{amssymb, amsmath,amsmath,latexsym,amssymb,amsfonts,amsbsy, amsthm}
\def\inte#1{
\displaystyle\mathop{#1\kern0pt}^\circ }

\newcommand{\ef}{ \hfill $ \blacksquare $ \vskip 3mm}
\newcommand{\beqo}{\begin{equation*}}
\newcommand{\eeqo}{\end{equation*}}
\newcommand{\beno}{\begin{eqnarray*}}
\newcommand{\eeno}{\end{eqnarray*}}

\numberwithin{equation}{section}

\let\pa=\partial
\let\al=\alpha

\let\d=\delta
\let\e=\varepsilon

\let\r=\rho

\let\f=\frac

\let\p=\psi

\let\D=\Delta

\let\Om=\Omega
\let\wt=\widetilde

\let\tri=\triangle
\let\ep=\varepsilon
\def\cA{{\mathcal A}}
\def\cB{{\mathcal B}}
\def\cC{{\mathcal C}}

\def\cM{{\mathcal M}}

\def\cR{{\mathcal R}}
\def\cS{{\mathcal S}}

\def\grad{\nabla}

\def\virgp{\raise 2pt\hbox{,}}
\def\cdotpv{\raise 2pt\hbox{;}}

\def\eqdefa{\buildrel\hbox{\footnotesize def}\over =}

\def\C{\mathop{\mathbb C\kern 0pt}\nolimits}
\def\DD{\mathop{\mathbb D\kern 0pt}\nolimits}
\def\EE{\mathop{{\mathbb E \kern 0pt}}\nolimits}
\def\K{\mathop{\mathbb K\kern 0pt}\nolimits}
\def\N{\mathop{\mathbb N\kern 0pt}\nolimits}
\def\Q{\mathop{\mathbb Q\kern 0pt}\nolimits}
\def\R{\mathop{\mathbb R\kern 0pt}\nolimits}
\def\SS{\mathop{\mathbb S\kern 0pt}\nolimits}
\def\ZZ{\mathop{\mathbb Z\kern 0pt}\nolimits}
\def\TT{\mathop{\mathbb T\kern 0pt}\nolimits}
\def\P{\mathop{\mathbb P\kern 0pt}\nolimits}

\newcommand{\la}{\lambda}

\newcommand{\Z}{{\ZZ}}

\def\dv{\mbox{div}}

\def\dive{\mathop{\rm div}\nolimits}

\def\Supp{\mathop{\rm Supp}\nolimits\ }

\def\no{\noindent}
\def\na{\nabla}
\def\p{\partial}
\def\th{\theta}

\newcommand{\beq}{\begin{equation}}
\newcommand{\eeq}{\end{equation}}
\newcommand{\ben}{\begin{eqnarray}}
\newcommand{\een}{\end{eqnarray}}
\newcommand{\andf}{\quad\hbox{and}\quad}

\newtheorem{defi}{Definition}[section]
\newtheorem{thm}{Theorem}[section]
\newtheorem{lem}{Lemma}[section]
\newtheorem{rmk}{Remark}[section]

\newtheorem{prop}{Proposition}[section]
\renewcommand{\theequation}{\thesection.\arabic{equation}}
\begin{document}
\title[Wellposedness of incompressible inhomogeneous NS equations]
{ Global wellposdeness to incompressible inhomogeneous
 fluid system with bounded density and  non-Lipschitz velocity }
\author[J. Huang]{Jingchi Huang}\address[J. HUANG]
{Academy of Mathematics $\&$ Systems Science, Chinese Academy of
Sciences, Beijing 100190, P. R. CHINA} \email{jchuang@amss.ac.cn}
\author[M. PAICU]{Marius Paicu}
\address [M. PAICU]
{Universit\'e  Bordeaux 1\\
 Institut de Math\'ematiques de Bordeaux\\
F-33405 Talence Cedex, France}
\email{marius.paicu@math.u-bordeaux1.fr}
\author[P. ZHANG]{Ping Zhang}%
\address[P. ZHANG]
 {Academy of
Mathematics $\&$ Systems Science and  Hua Loo-Keng Key Laboratory of
Mathematics, The Chinese Academy of Sciences\\
Beijing 100190, CHINA } \email{zp@amss.ac.cn}
\date{12/Oct/2012}
\maketitle
\begin{abstract}
In this paper, we first prove the global existence of weak solutions
to the d-dimensional  incompressible inhomogeneous  Navier-Stokes
equations with initial data $a_0\in L^\infty(\R^d),$
$u_0=(u_0^h,u_0^d)\in \dot{B}^{-1+\f{d}p}_{p,r}(\R^d),$ which
satisfies
$\bigl(\mu\|a_0\|_{L^\infty}+\|u_0^h\|_{\dot{B}^{-1+\frac{d}p}_{p,r}}\bigr)\exp\bigl(
C_r{\mu^{-2r}}\|u_0^d\|_{\dot{B}^{-1+\frac{d}p}_{p,r}}^{2r}\bigr)\leq
c_0\mu$ for some positive constants $c_0, C_r$ and $1< p<d,$
$1<r<\infty.$ The regularity of the initial velocity is critical to
the scaling of this system and is general enough to generate
non-Lipschitz velocity field. Furthermore, with additional
regularity assumption on the initial velocity or on the initial
density, we can also prove the uniqueness of such solution. We
should mention that the classical maximal $L^p(L^q)$ regularity
theorem for the heat kernel plays an essential role in this context.
\end{abstract}


\vskip 0.2cm


\setcounter{equation}{0}
\section{Introduction}

In this paper, we consider the  global wellposedness to the
following d-dimensional  incompressible inhomogeneous Navier-Stokes
equations with  the regularity of the initial velocity being almost
critical and the initial density being a bounded positive function,
which satisfies some nonlinear smallness condition,
\begin{equation}
 \left\{\begin{array}{l}
\displaystyle \pa_t \rho + \dv (\rho u)=0,\qquad (t,x)\in\R^+\times\R^d, \\
\displaystyle \pa_t (\rho u) + \dv (\rho u\otimes u) -\mu\D u+\grad\Pi=0, \\
\displaystyle \dv\, u = 0, \\
\displaystyle \rho|_{t=0}=\rho_0,\quad \rho u|_{t=0}=\r_0u_0,
\end{array}\right. \label{1.1}
\end{equation}
where $\rho, u=(u^h, u^d)$ stand for the density and  velocity of
the fluid respectively, $\Pi$  is a scalar pressure function,
 and $\mu$ the
viscosity coefficient. Such system describes a fluid which is
obtained by mixing two immiscible fluids that are incompressible and
that have different densities. It may also describe a fluid
containing a melted substance. We remark that our hypothesis on the
density is of physical interest, which corresponds to the case of a
mixture of immiscible fluids with different and bounded densities.

In particular, we shall focus on the global wellposedness of
\eqref{1.1} with small homogeneity for the initial density function
in $L^\infty(\R^d)$ and small horizontal components of the velocity
compared with its vertical component. This approach was already
applied by Paicu and Zhang \cite{PZ1, PZ2} for 3-D anisotropic
Navier-Stokes equations and for inhomogeneous Navier-Stokes system
in the framework of Besov spaces.
 The main novelty of the present paper is to consider the initial density function in $L^\infty(\R^d),$
  which is close enough to some positive constant. Then in order to handle  the
nonlinear terms appearing in \eqref{1.1}, we need to use the maximal
regularity effect for the classical heat equation. We should mention
that the initial data have scaling invariant regularities
 and the global weak
solutions obtained here, under a nonlinear-type smallness condition,
also belong to the critical spaces. Moreover, the regularity of the
velocity field obtained in this paper is general enough to include
the case of non-Lipschitz vector fields.

When  $\rho_0$ is bounded away from 0, Kazhikov \cite{KA} proved
that: the inhomogeneous Navier-Stokes equations \eqref{1.1} has at
least one global weak solutions in the energy space. In addition, he
also proved the global existence of strong solutions to this system
for small data in three space dimensions and all data in two
dimensions. However, the uniqueness of both type weak solutions has
not be solved.  Lady\v zenskaja and Solonnikov  \cite{LS} first
addressed the question of unique resolvability of (\ref{1.1}). More
precisely, they  considered the system \eqref{1.1} in a bounded
domain $\Om$ with homogeneous Dirichlet boundary condition for $u.$
Under the assumption that $u_0\in W^{2-\frac2p,p}(\Om)$ $(p>d)$ is
divergence free and vanishes on  $\p\Om$ and that $\r_0\in C^1(\Om)$
is bounded away from zero, then they \cite{LS} proved
\begin{itemize}
\item Global well-posedness in dimension $d=2;$
\item Local well-posedness in dimension $d=3.$ If in addition $u_0$ is small in $W^{2-\frac2p,p}(\Om),$
then global well-posedness holds true.
\end{itemize}
Similar results were obtained by Danchin \cite{danchin2} in $\R^d$
with initial data in the almost critical Sobolev spaces.  Abidi, Gui
and Zhang \cite{A-G-Z} investigated the large time decay and
stability to any given global smooth solutions of \eqref{1.1}, which
in particular implies the global wellposedness of 3-D inhomogeneous
Navier-Stokes equations with axi-symmetric initial data provided
that there is no swirl part for the initial velocity field  and the
initial density is close enough to a positive constant. In general,
when  the viscosity coefficient, $\mu(\rho),$ depends on $\rho,$
Lions \cite{LP} proved the global existence of weak solutions to
\eqref{1.1}  in any space dimensions.

In the case when the density function $\rho$ is away from zero, we
denote by $a\eqdefa\frac{1}{\rho}-1$, then the system (\ref{1.1})
can be
 equivalently reformulated as
\begin{equation}\label{INS}
 \quad\left\{\begin{array}{l}
\displaystyle \pa_t a + u \cdot \grad a=0,\qquad (t,x)\in \R^+\times\R^d,\\
\displaystyle \pa_t u + u \cdot \grad u+ (1+a)(\grad\Pi-\mu\D u)=0, \\
\displaystyle \dv\, u = 0, \\
\displaystyle (a, u)|_{t=0}=(a_0, u_{0}).
\end{array}\right.
\end{equation}
Notice that just as the classical Navier-Stokes system, the
inhomogeneous Navier-Stokes system (\ref{INS}) also  has a scaling.
More precisely,  if $(a, u)$ solves (\ref{INS}) with initial data
$(a_0, u_0)$, then for $\forall \, \ell>0$,
\begin{equation}\label{1.2}
(a, u)_{\ell} \eqdefa (a(\ell^2\cdot, \ell\cdot), \ell u(\ell^2
\cdot, \ell\cdot))\quad\mbox{and}\quad (a_0,u_0)_\ell\eqdefa
(a_0(\ell\cdot),\ell u_0(\ell\cdot))
\end{equation}
 $(a, u)_{\ell}$ is also a solution of (\ref{INS}) with initial data $(a_0,u_0)_\ell$.

 In \cite{danchin}, Danchin studied in general space dimension $d$
the unique solvability of the system (\ref{INS}) in scaling
invariant homogeneous Besov spaces, which generalized the celebrated
results by Fujita and Kato \cite{fujitakato} devoted to the
classical Navier-Stokes system. In particular, the norm of $(a, u)
\in \dot{B}_{2,\infty}^{\frac{d}{2}}(\R^d)\cap
L^{\infty}(\R^d)\times \dot{B}_{2,1}^{\frac{d}{2}-1}(\R^d)$ is
scaling invariant  under the change of scale of \eqref{1.2}. In this
case, Danchin proved that if the initial data $(a_0,u_0) \in
\dot{B}_{2,\infty}^{\frac{d}{2}}(\R^d)\cap L^{\infty}(\R^d)\times
\dot{B}_{2, 1}^{\frac{d}{2}-1}(\R^d)$ with $a_0$ sufficiently small
in $ \dot{B}_{2, \infty}^{\frac{d}{2}}(\R^d)\cap L^{\infty}(\R^d)$,
then the system (\ref{INS}) has a unique local-in-time solution.
Abidi \cite{abidi} proved that if $1<p<2d,$ $
0<\underline{\mu}<\tilde{\mu}(a),$ $u_0\in
\dot{B}_{p,1}^{\frac{d}p-1}(\R^d)$ and $a_0\in
\dot{B}_{p,1}^{\frac{d}p}(\R^d),$ then (\ref{INS}) has a global
solution provided that $\|a_0\|_{\dot{B}_{p,1}^{\frac{d}p}}
+\|u_0\|_{\dot{B}_{p,1}^{\frac{d}p-1}}\leq c_0$ for some $c_0$
sufficiently small. Furthermore, such a  solution is unique if
$1<p\leq d.$ This result generalized the corresponding results in
\cite{danchin, danchin2} and was improved by Abidi and Paicu in
\cite{AP} when $\tilde{\mu}(a)$ is a positive constant by using
different Lebesgue indices for the density and for the velocity.
More precisely, for $a_0\in\dot{B}^{\frac{d}{q}}_{q,1}(\R^d)$ and
$u_0\in \dot{B}^{-1+\frac{d}{p}}_{p,1}(\R^d)$ with $|\frac
1{p}-\frac{1}{q}|<\frac 1d$ and $\frac{1}{p}+\frac{1}{q}>\frac 1d,$
they obtained the existence of solutions to \eqref{INS}  and under a
more restrictive condition: $\frac{1}{p}+\frac{1}{q}\geq
\frac{2}{d},$ they proved the uniqueness of this solution. In
particular, with a well prepared regularity for the density
function, this result implies the global existence of solutions to
\eqref{INS} for any $1<p<\infty$ and the uniqueness of such solution
when $1<p< 2d$.
 Very recently, Danchin and Mucha \cite{dm} filled the gap for
 the uniqueness result in \cite{abidi}  with $p\in (d,2d)$
through Langrage approach, and Abidi, Gui and Zhang relaxed the
smalness condition for $a_0$ in \cite{AGZ2, AGZ3}.

 On the other
hand, when the initial density $\r_0\in L^\infty(\Om)$ with positive
lower bound and $u_0\in H^2(\Om),$ Danchin and Mucha \cite{dm2}
proved the local wellposedness of \eqref{1.1}. They also proved the
global wellposedness result provided that the fluctuation of the
initial density is sufficiently small, and initial velocity is small
in $B^{2-\f2q}_{q,p}(\Om)$ for $1<p<\infty, d<q<\infty$ in 3-D and
any velocity in $B^1_{4,2}(\Om)\cap L^2(\Om)$ in 2-D. Motivated by
Proposition \ref{prop1.1} below concerning the alternative
definition of Besov spaces (see Definition \ref{def1.1}) with
negative indices and \cite{Kato}, where Kato solved the local (resp.
global) wellposedness of 3-D classical Navier-Stokes system through
elementary $L^p$ approach, we shall investigate the global existence
of weak solutions to \eqref{INS} with initial data $a_0\in
L^\infty(\R^d)$ and $u_0\in \dot{B}^{-1+\f{d}p}_{p,r}(\R^d)$ for
$p\in (1,d)$ and $r\in (1,\infty),$ which satisfies the nonlinear
smallness condition \eqref{small1}. Furthermore, if we assume, in
addition, $u_0\in  \dot{B}^{-1+\f{d}p+\e}_{p,r}(\R^d)$ for some
sufficiently small  $\e>0$, we can also prove the uniqueness of such
solution.

\begin{defi}\label{defi1.1} {\sl We call $(a,u, \na\Pi)$ a global weak solution of \eqref{INS}
if
\begin{itemize}
\item for any test function $\phi\in
C^\infty_c([0,\infty)\times\R^d),$ there holds
\beq\label{def1.1a}\begin{split}
\int_0^\infty\int_{\R^d}&a(\p_t\phi+u\cdot\na\phi)\,dx\,dt+\int_{\R^d}\phi(0,x)a_0(x)\,dx=0,\\
& \int_0^\infty\int_{\R^d}\dive u\phi\,dx\,dt=0,\end{split} \eeq

\item for any vector valued function $\Phi=(\Phi^1,\cdots,\Phi^d)\in
C_c^\infty([0,\infty)\times\R^d),$  one has \beq\label{def1.1b}
\int_0^\infty\int_{\R^d}\Bigl\{u\cdot\p_t\Phi-(u\cdot\na u) \cdot
\Phi +(1+a)(\mu\D
u-\na\Pi)\cdot\Phi\Bigr\}\,dx\,dt+\int_{\R^d}u_0\cdot\Phi(0,x)\,dx=0.
\eeq
\end{itemize}}
\end{defi}

We denote the vector field by $u=(u^h,u^d)$ where
$u^h=(u^1,u^2,...,u^{d-1})$. Our first main result in this paper is
as follows:

\begin{thm}\label{thm1.1}
{\sl Let  $p\in (1,d)$ and $r\in(1,\infty).$ Let $a_0\in
L^\infty(\R^d)$ and $ u_0\in \dot{B}^{-1+\f{d}p}_{p,r}(\R^d).$ Then
there exist positive constants $c_0, C_r$ so that if \beq
\label{small1} \eta \eqdefa
\bigl(\mu\|a_0\|_{L^\infty}+\|u_0^h\|_{\dot{B}^{-1+\f{d}p}_{p,r}}\bigr)\exp\Bigl\{C_r
\mu^{-2r}\|u_0^d\|_{\dot{B}^{-1+\f dp}_{p,r}}^{2r}\Bigr\}\leq
c_0\mu, \eeq \eqref{INS} has a global weak solution $(a,u)$ in the
sense of Definition \ref{defi1.1}, which satisfies

\no (1)\ when $p\in (1, \f{dr}{3r-2}],$
 \beq\label{thm1a}
\begin{split}
&\mu^{\f1r}\|\D u^h\|_{L^r(\R^+;L^{\f{dr}{3r-2}})}+
\mu^{\f1{2r}}\|\na u^h\|_{L^{2r}(\R^+;L^{\f{dr}{2r-1}})}
\leq C\eta,\\
&\mu^{\f1r}\|\D u^d\|_{L^r(\R^+;L^{\f{dr}{3r-2}})}
+\mu^{\f1{2r}}\|\na u^d\|_{L^{2r}(\R^+;L^{\f{dr}{2r-1}})}\leq
C\|u_{0}^d\|_{\dot{B}^{-1+\f
dp}_{p,r}}+ c\mu,\\
 &\mu^{\f1r}\|\na\Pi\|_{L^r;(\R^+;L^{\f{dr}{3r-2}})}\leq C\eta\bigl(\|u_{0}^d\|_{\dot{B}^{-1+\f
dp}_{p,r}}+ c\mu\bigr);\end{split} \eeq

\no (2)\ when $p\in (\f{dr}{3r-2}, d),$ \beq\label{thm1b}
\begin{split}\mu^{\f12(3-\f{d}{p_1})}&\|t^{\alpha_1} \D
 u^h\|_{L^{2r}(\R^+; L^{p_1})} +
\mu^{1-\f{d}{2p_2}}\bigl(\|t^{\beta_1}\na u^h\|_{L^{2r}(\R^+;
L^{p_2})}+
 \|t^{\beta_2}\na u^h\|_{L^{\infty}(\R^+; L^{p_2})}\bigr)\\
&+\mu^{\f12(1-\f{d}{p_3})}\bigl(\|t^{\gamma_1}u^h\|_{L^{\infty}(\R^+;L^{p_3})}
+
\|t^{\gamma_2}u^h\|_{L^{2r}(\R^+; L^{p_3})}\bigr)\leq C\eta,\\
\mu^{\f12(3-\f{d}{p_1})}&\|t^{\alpha_1} \D u^d\|_{L^{2r}(\R^+;
L^{p_1})}
 + \mu^{1-\f{d}{2p_2}}\bigl(\|t^{\beta_1}\na u^d\|_{L^{2r}_t(L^{p_2})}+\|t^{\beta_2}\na u^d\|_{L^{\infty}(\R^+; L^{p_2})}\bigr)\\
 &+
\mu^{\f12(1-\f{d}{p_3})}\bigl(\|t^{\gamma_1}u^d\|_{L^{\infty}(\R^+;
L^{p_3})} +\|t^{\gamma_2}u^d\|_{L^{2r}(\R^+; L^{p_3})}\bigr)
 \leq 2C\| u_{0}^d\|_{\dot{B}^{-1+\f{d}p}_{p,r}} +c\mu,
\end{split}
\eeq and \beq\label{thm1c}
\begin{split}
&\mu^{\f12(3-\f{d}{p_1})}\|t^{\alpha_2} \D u\|_{L^{r}(\R^+;
L^{p_1})} \leq C\bigl(\|
u_{0}\|_{\dot{B}^{-1+\f{d}p}_{p,r}}+\f{\eta}\mu(\|
u_{0}^d\|_{\dot{B}^{-1+\f{d}p}_{p,r}}+c\mu+\eta)\bigr),\\
&\mu^{\f12(3-\f{d}{p_1})}\bigl(\|t^{\alpha_1} \na\Pi\|_{L^{2r}(\R^+;
L^{p_1})}+\|t^{\alpha_2} \na\Pi\|_{L^{r}(\R^+; L^{p_1})}\bigr)\leq
C\eta\bigl(\| u_{0}^d\|_{\dot{B}^{-1+\f{d}p}_{p,r}}+c\mu\bigr),
\end{split}\eeq
 for some small enough constant $c,$ where $p_1, p_2, p_3$
satisfy $\max (p, \f{dr}{2r-1})<p_1<d$,
 and $\f{dr}{r-1}<p_3<\infty$ so that $\f1{p_2}+\f1{p_3}=\f1{p_1},$  the indices
 $\alpha_1$, $\alpha_2,$
$\beta_1$, $\beta_2$, $\gamma_1$, $\gamma_2$ are determined by
 \beq\label{indexad}\begin{split}
 & \alpha_1=\f12(3-\f{d}{p_1})-\f1{2r},\quad
\beta_1=\f12(2-\f{d}{p_2})-\f1{2r} \andf
\gamma_1=\f12(1-\f{d}{p_3}),\\
& \alpha_2=\f12(3-\f{d}{p_1})-\f1{r},\quad \
\beta_2=\f12(2-\f{d}{p_2}) \andf
\gamma_2=\f12(1-\f{d}{p_3})-\f1{2r}. \end{split} \eeq
 Furthermore, if we assume, in addition, that $u_0\in
\dot{B}^{-1+\f{d}p+\e}_{p,r}(\R^d)$ for $0<\ep<\min\{ \f1r, 1-\f1r,
\f{d}p-1\}$, then such a global solution is unique.
 }
\end{thm}

\begin{rmk}\label{rmk1.0}
The main idea to prove Theorem \ref{thm1.1} is to use the maximal
$L^p(L^q)$ regularizing effect for heat kernel (see Lemma
\ref{lem1}). In fact, similar to the classical Navier-Stokes
equations (\cite{Kato}), we first reformulate \eqref{INS} as
\beq\label{rmk1.0a} u=e^{\mu
t\D}u_0+\int_0^te^{\mu(t-s)\D}\bigl\{-u\cdot\na u+\mu a\D
u-(1+a)\na\Pi\bigr\}\,ds. \eeq Then we can prove appropriate
approximate solutions to \eqref{rmk1.0a} satisfies the uniform
estimate (\ref{thm1a}-\ref{thm1c}). With these estimates, the
existence part of Theorem \ref{thm1.1} follows by a compactness
argument.

We remark that given initial data $u_0\in
\dot{B}^{-1+\f{d}p}_{p,r}(\R^d),$ the maximal regularity we can
expect for $u$ is $\wt{L}^1_t(\dot{B}^{1+\f{d}p}_{p,r}).$ With this
regularity for $u$ and $a\in L^\infty(\R^+\times\R^d),$ we do not
know how to define the product $a\D u$ in the sense of distribution
if $p<d.$ This explains in some sense why we can only prove Theorem
\ref{thm1.1} for $p\in (1,d).$
\end{rmk}

\begin{rmk}\label{rmk1.1}
The smallness condition \eqref{small1} is motivated by the one in
\cite{PZ2} (see also \cite{GZ2, PZ1, Zhangt} for the related works
on 3-D incompressible anisotropic Navier-Stokes system), where we
prove that: for $1<q\leq p<6$ with $\frac1q-\frac1p\leq \frac13,$
given any data $a_0\in \dot{B}^{\f3q}_{q,1}(\R^3)$ and
$u_0=(u_0^h,u_0^3)\in \dot{B}^{-1+\frac3p}_{p,1}(\R^3)$ verifying
\beq \label{1.4} \eta\eqdefa
\bigl(\mu\|a_0\|_{\dot{B}_{q,1}^{\f3q}}+\|u_0^h\|_{\dot{B}^{-1+\frac3p}_{p,1}}\bigr)\exp\Bigl\{
C_0\|u_0^3\|_{\dot{B}^{-1+\frac3p}_{p,1}}^2\ \Big/\mu^2\Bigr\}\leq
c_0\mu, \eeq for some positive constants $c_0$ and $C_0,$
\eqref{INS} has a unique global solution $a\in C([0,\infty);
\dot{B}^{\f3q}_{q,1}(\R^3))$ and $u\in C([0,\infty);
\dot{B}^{-1+\frac3p}_{p,1}(\R^3))\cap
L^1(\R^+;\dot{B}^{1+\frac3p}_{p,1}(\R^3)).$ Similar wellposedness
result (\cite{HPZ2}) holds with $\|a_0\|_{\dot{B}_{q,1}^{\f3q}}$ in
\eqref{1.4} being replaced by
$\|a_0\|_{\cM(\dot{B}_{p,1}^{-1+\f3p})},$ the norm to the multiplier
space of $\dot{B}_{p,1}^{-1+\f3p}(\R^3).$
  We emphasize that
our proof  in \cite{PZ2,HPZ2} uses in a fundamental way the
algebraical structure of \eqref{INS}, namely, $\dive u=0,$ which
will also be one of the key ingredients in the proof of Theorem
\ref{thm1.1} and Theorem \ref{thm1.2} below.
\end{rmk}

\begin{rmk}\label{rmk1.2}
We should also mention the recent interesting result by Danchin and
Mucha \cite{dm2} that with more regularity assumption on the initial
velocity field, namely, $m\leq\r_0<M$ and $u_0\in H^2(\R^d)$ for
$d=2,3,$ they can prove the local wellposedness of \eqref{1.1} for
large data and global wellposedness for small data. We emphasize
that here we work our initial velocity field in the critical space
$\dot{B}^{-1+\f{d}p}_{p,r}(\R^d)$ for $p\in (1,d)$ and $r\in
(1,\infty)$ and also the fact that
  Theorem \ref{thm1.1} remains to be valid in the case of
  bounded smooth domain with Dirichlet boundary conditions for the velocity
  field.  Moreover, our
uniqueness result in Theorem \ref{thm1.1} is strongly inspired by
the Lagrangian approach in \cite{dm2}, but with an almost critical
regularity for the velocity,  the proof here will be much more
complicated.  In fact, the small extra regularity compared to the
scaling  \eqref{1.2}, namely, $u_0\in B^{-1+\frac
dp+\e}_{p,r}(\R^d)$ for some small $\e>0,$ is useful to obtain that
$\Delta u\in L^1_{loc}(L^{d+\eta})$ for some $\eta>0,$ which
combined with $\Delta u\in L^1_{loc}(L^{p_1})$ with $p_1<d$ (see
\eqref{thm1b}) implies that $u\in L^1_{loc}(Lip)$ and this allows us
to reformulate \eqref{INS} in the Lagrangian coordinates. One may
check Theorem \ref{thm1.1bis} One may check Theorem \ref{thm1.1bis}
below for more information about this solution.
\end{rmk}

A different approach to recover the uniqueness of the solution is to impose
 more regularity on the density function. Indeed, if the density is such that $a_0\in   B^{\f{d}q+\e}_{q,\infty}(\R^d)$, for some
 small positive $\ep$, we can also
prove the global wellposedness of \eqref{INS} under the nonlinear
smallness condition \eqref{small2}:

\begin{thm}\label{thm1.2}
{\sl Let $r\in(1,\infty),$ $1<q\leq p<2d$ with $\frac1q-\frac1p\leq
\frac1d$  and $\e\in (0,\f{2d}p-1)$ be any positive real number. Let
$a_0\in L^\infty(\R^d)\cap   B^{\f{d}q+\e}_{q,\infty}(\R^d)$ and $
u_0\in {\dot B}^{-1+\f{d}p-\e}_{p,r}(\R^d)\cap {\dot
B}^{-1+\f{d}p}_{p,r}(\R^d).$ There exist positive constants $c_0,
C_{r,\e}$ so that if \beq \label{small2} \delta \eqdefa
\bigl(\mu\|a_0\|_{L^\infty\cap B^{\f{d}q+\e}_{q,\infty}
}+\|u_0^h\|_{{\dot B}^{-1+\f{d}p-\e}_{p,r}\cap {\dot
B}^{-1+\f{d}p}_{p,r}}\bigr)\exp\Bigl\{C_{r,\e}
\mu^{-2r}\|u_0^d\|_{{\dot B}^{-1+\f{d}p-\e}_{p,r}\cap {\dot
B}^{-1+\f dp}_{p,r}}^{2r}\Bigr\}\leq c_0\mu, \eeq \eqref{INS} has a
global solution $(a,u)$ so that \beno
\begin{split}
& a\in  C([0,\infty); L^\infty(\R^d)\cap
 B^{\f{d}q+\f{\e}2}_{q,\infty}(\R^d)),\\
& u \in C([0,\infty); \dot B^{-1+\f{d}p-\e}_{p,r}(\R^d)\cap \dot
B^{-1+\f{d}p}_{p,r}(\R^d))\cap \wt{L}^1(\R^+; \dot
B^{1+\f{d}p-\e}_{p,r}(\R^d)\cap\dot B^{1+\f{d}p}_{p,r}(\R^d)),
\end{split}
\eeno and there holds \beq
\begin{split}
&\|u^h\|_{\wt{L}^\infty(\R^+;\dot{B}^{-1+\f{d}p-\ep}_{p,r})}+
\|u^h\|_{\wt{L}^\infty(\R^+;\dot{B}^{-1+\f{d}p}_{p,r})}+\mu\bigl(\|a\|_{\wt{L}^\infty(\R^+;
 B^{\f{d}q+\f\e2 }_{q,\infty})}\\
&\qquad+\|u^h\|_{\wt{L}^1(\R^+;\dot{B}^{1+\f{d}p-\e}_{p,r})}+\|u^h\|_{\wt{L}^1(\R^+;\dot{B}^{1+\f{d}p}_{p,r})}\bigr)\leq C\delta,\\
&\|u^d\|_{\wt{L}^\infty(\R^+;\dot{B}^{-1+\f{d}p-\e}_{p,r})}+\|u^d\|_{\wt{L}^\infty(\R^+;\dot{B}^{-1+\f{d}p}_{p,r})}+\mu
\bigl(\|u^d\|_{\wt{L}^1(\R^+;\dot{B}^{1+\f{d}p-\e}_{p,r})}\\
&\qquad+\|u^d\|_{\wt{L}^1(\R^+;\dot{B}^{1+\f{d}p}_{p,r})}\bigr)\leq
2\|u_0^d\|_{\dot{B}^{-1+\f{d}p-\e}_{p,r}\cap
\dot{B}^{-1+\f{d}p}_{p,r}}+c_2\mu,
\end{split}
\label{5.20} \eeq for some $c_2$ sufficiently small, and the norm
$\|\cdot\|_{X\cap Y}=\|\cdot\|_{X}+\|\cdot\|_{Y}.$
 Furthermore, this solution is unique if $\f1q+\f1p\geq\f2d.$ }
\end{thm}

\begin{rmk}\label{rmk1.3} We point out that
Haspot \cite{Haspot} proved the local well-posedness of \eqref{INS}
under similar conditions of Theorem \ref{thm1.2}. Our novelty here
is the global existence of solutions to \eqref{INS} under the
smallness condition  \eqref{small2}. We should also mention that: to
overcome the difficulty that one can not use Gronwall's inequality
in the framework of Chemin-Lerner spaces, motivated by \cite{PZ1,
PZ2}, we introduced the weighted Chemin-Lerner type Besov norms  in
Definition \ref{defa.3}, which will be one of key ingredients used
in the proof of Theorem \ref{thm1.2}.
\end{rmk}

\begin{rmk}
We remark that in the previous works on the global wellposedness of
\eqref{INS}, the third index, $r,$ of the Besov spaces, to which the
initial data belong,  always equals to $1$. In this case, the
regularizing effect of heat equation allows the velocity field to be
in $L^1(\R^+,Lip(\R^d)),$ which is very useful to solve the
transport equation without losing any derivative of the initial
data. In both Theorem \ref{thm1.1} and Theorem \ref{thm1.2}, the
regularity of the velocity field is general enough to include
non-Lipschitz vector-fields.
\end{rmk}

\no{\bf The organization of the paper.} In the second section, we
present the proof to the existence part of Theorem \ref{thm1.1} in
the case when $1<p\leq\f{dr}{3r-2}$. In Section 3, we shall present
the proof  to the existence part of Theorem \ref{thm1.1}  for the
remaining case: $\f{dr}{3r-2}< p<d$.
 In Section 4, we shall present the proof to the uniqueness part of
Theorem \ref{thm1.1} with an extra regularity on the initial
velocity. In Section 5, we prove Theorem \ref{thm1.2} which gives
the uniqueness in the case where we have an additional regularity on
the density.
 Finally in the appendix, we collect some basic
facts on  Littlewood-Paley theory and Besov spaces, which have been
used throughout this paper.

Let us complete this section by the notations we shall use in this
context:

\no{\bf Notation.} For $X$ a Banach space and $I$ an interval of
$\R,$ we denote by ${C}(I;\,X)$ the set of continuous functions on
$I$ with values in $X,$ and by  $L^q(I;\,X)$ stands for the set of
measurable functions on $I$ with values in $X,$ such that
$t\longmapsto\|f(t)\|_{X}$ belongs to $L^q(I).$  For $a\lesssim b$,
we mean that there is a uniform constant $C,$ which may be different
on different lines, such that $a\leq Cb$. We shall denote by
$(c_{j,r})_{j\in\Z}$ to be a generic element of $\ell^r(\Z)$ so that
$c_{j,r}\geq 0$ and $\sum_{j\in\Z}c_{j,r}^r=1.$

\setcounter{equation}{0}
\section{Proof  to the existence part of Theorem \ref{thm1.1} for
$1<p\leq\f{dr}{3r-2}$}\label{sect2}

 One of the key ingredients used in the proof to the existence part of Theorem \ref{thm1.1} lies in Proposition
 \ref{prop1.1} below:

\begin{prop}[Theorem 2.34 of \cite{BCD}] \label{prop1.1}
{\sl Let $s$ be a negative real number and $(p,r)\in [1,\infty]^2.$
A constant $C$ exists such that \beno
C^{-1}\mu^{\f{s}2}\|f\|_{\dot{B}^{s}_{p,r}}\leq
\bigl\|\|t^{-\f{s}2}e^{t\D}f\|_{L^p}\bigr\|_{L^r(\R^+;\f{dt}{t})}\leq
C\mu^{\f{s}2}\|f\|_{\dot{B}^{s}_{p,r}}. \eeno }
\end{prop}

 In particular, for $r\geq 1,$ we deduce  from Proposition
 \ref{prop1.1} that $f\in \dot{B}^{-\f2r}_{p, r}(\R^d)$ is equivalent to
$e^{\mu t\D}f\in L^r(\R^+;L^p(\R^d))$.

Notice that given $u_0\in \dot{B}^{-1+\f{d}p}_{p,r}(\R^d)$ with
$1<p\leq \f{dr}{3r-2}$ and $r\in (1,\infty),$ $\tri u_0 \in
\dot{B}^{-3+\f dp}_{p,r}(\R^d),$ we can always find some $q_1\geq p$
and $r_1\geq r$ such that
$$-3+\f dp \geq -3+\f{d}{q_1}=-\f2{r_1}\geq -\f2r.$$ Choosing
$r_1=r$ in the above inequality leads to $q_1=\f{dr}{3r-2}$. Then
$\tri u_0 \in \dot{B}^{-\f2r}_{\f{dr}{3r-2}, r}(\R^d),$ we infer
that $\tri e^{\mu t\D}u_0\in L^r(\R^+; L^{\f{dr}{3r-2}}(\R^d))$.
Similarly, we can choose some $q_2\geq p$ and $r_2\geq r$ with
$-2+\f{d}{q_2}=-\f2{r_2},$ so that  $\na u_0 \in \dot{B}^{-2+\f
dp}_{p,r}(\R^d)\hookrightarrow \dot{B}^{-2+\f{d}{q_2}}_{q_2,
r_2}(\R^d).$ Choosing $r_2=2r$ gives rise to
$q_2=\f{dr}{2r-1}>\f{dr}{3r-2}\geq p$ and $\na e^{\mu t\D}u_0 \in
L^{2r}(R^+;L^{\f{dr}{2r-1}}(\R^d))$. And Sobolev embedding ensures
that $e^{t\D}u_0\in L^{2r}(R^+;L^{\f{dr}{r-1}}(\R^d))$ in this case.

The other key ingredient used in the proof of Theorem \ref{thm1.1}
is the following lemma (see \cite{Lem} for instance), which is
called maximal $L^p(L^q)$ regularity for the heat kernel.

\begin{lem}\label{lem1} (Lemma 7.3 of \cite{Lem})
{\sl The operator $\cA$ defined by $f(t,x)\mapsto\int_0^t \tri
e^{\mu(t-s)\tri}f\,ds$ is bounded from  $L^p((0,T); L^q(R^d))$ to
$L^p((0,T); L^q(R^d))$ for every $T\in (0,\infty]$ and $1 <
p,q<\infty$. Moreover, there holds \beno \|\cA
f\|_{L^p_T(L^q)}\leq\f{C}\mu\|f\|_{L^p_T(L^q)}.\eeno }
\end{lem}

\begin{lem}\label{lem2a}
{\sl Let  $1 < r<\infty.$ The operator $\cB$ defined by
$f(t,x)\mapsto \int_0^t \na e^{\mu(t-s)\tri}f\,ds$ is bounded from
$L^{2r}((0,T); L^{\f{dr}{2r-1}}(R^d))$ to $ L^r((0,T);
L^{\f{dr}{3r-2}}(R^d))$ for every $T\in (0,\infty],$ and there holds
\beq\label{2.1wrt} \bigl\| \int_0^t\na
e^{\mu(t-s)\tri}f\,ds\bigr\|_{L^{2r}_T(L^{\f{dr}{2r-1}})}\leq
\f{C}{\mu^{\f{2r-1}{2r}}}\|f\|_{ L^r_T(L^{\f{dr}{3r-2}})}. \eeq
 }
\end{lem}
\begin{proof} Notice that
\beq\label{2.1a}
\begin{split}
 \na
e^{\mu(t-s)\tri}f(s,x)=&\f{\sqrt{\pi}}{(4\pi\mu(t-s))^{\f{d+1}2}}\int_{\R^d}\f{(x-y)}{2\sqrt{\mu(t-s)}}
\exp\Bigl\{-\f{|x-y|^2}{4\mu(t-s)}\Bigr\}f(s,y)\,dy\\
\eqdefa& \f{\sqrt{\pi}}{(4\pi\mu(t-s))^{\f{d+1}2}}
K(\f{\cdot}{\sqrt{4\mu(t-s)}})\ast f(s,x). \end{split}\eeq Applying
Young's inequality in the space variables yields \beno\begin{split}
 \|\na
e^{\mu(t-s)\tri}&( 1_{[0,T]}(s)f)(s,\cdot)\|_{L^{\f{dr}{2r-1}}}\\
\leq&
C(\mu(t-s))^{-\f{d+1}2}\|K(\f{\cdot}{\sqrt{4\pi\mu(t-s)}})\|_{L^{\f{dr}{dr-r
+1}}} \|1_{[0,T]}(s) f(s)\|_{L^{\f{dr}{3r-2}}}\\
\leq& C(\mu(t-s))^{-\f{2r-1}{2r}}\|1_{[0,T]}(s)
f(s)\|_{L^{\f{dr}{3r-2}}},
\end{split} \eeno
where $1_{[0,t]}(s)$ denotes the characteristic function on $[0,t],$
from which and Hardy-Littlewood-Sobolev inequality, we conclude the
proof of \eqref{2.1wrt}.
\end{proof}

In what follows, we shall seek a solution $(a,u)$ of \eqref{1.1} in
the following space: \beq \label{2.4} X\eqdefa\bigl\{ (a,u):\ a\in
L^\infty(\R^+\times\R^d),\ \na u \in L^{2r}(\R^+;
L^{\f{dr}{2r-1}}(\R^d)),\ \tri u \in L^{r}(\R^+;
L^{\f{dr}{3r-2}}(\R^d))\,\bigr\}.\eeq

 We first mollify the initial data $(a_0,u_0),$ and then construct the approximate
 solutions
 $(a_n,u_n)$ via
\begin{equation}\label{schema}
\begin{cases}
\partial_ta_{n}+u_{n}\cdot\nabla a_{n}=0,\\
\partial_t u_{n}+u_{n}\cdot\nabla u_{n}-\mu\Delta u_{n}+\nabla \Pi_{n}=a_{n}(\mu\Delta u_{n}-\nabla \Pi_{n})\\
\text{div\;} u_{n}=0\\
(a_{n}, u_{n})|_{t=0}=(S_{N+n}a_0, S_{N+n}u_0),
\end{cases}
\end{equation}
where $N$ is a large enough positive integer, and $S_{n+N}a_0$
denotes the partial sum of $a_0$ (see the Appendix for its
definition).

We have the following proposition concerning the uniform bounds of
$(a_n,u_n)$.

\begin{prop}\label{uniformbound}
{\sl Under the assumptions of Theorem \ref{thm1.1}, \eqref{schema}
has a unique global smooth solution $(a_n,u_n,\na\Pi_n)$ which
satisfies
 \beq\label{2.2ad}
\begin{split}
& \mu^{\f1r}\|\D
u_{n}^h\|_{L^r(\R^+;L^{\f{dr}{3r-2}})}\mu^{\f1{2r}}\|\na u_{n}^h\|_{L^{2r}(\R^+;L^{\f{dr}{2r-1}})}\leq C\eta,\\
&\mu^{\f1r}\|\D
u_n^{d}\|_{L^r(\R^+;L^{\f{dr}{3r-2}})}+\mu^{\f1{2r}}\|\na
+u_n^{d}\|_{L^{2r}(\R^+;L^{\f{dr}{2r-1}})} \leq
C\|u_{0}^d\|_{\dot{B}^{-1+\f dp}_{p,r}}+ c\mu,\end{split}\eeq and
\beq\label{2.2ap}
\mu^{\f1r}\|\na\Pi_n\|_{L^r(\R^+;L^{\f{dr}{3r-2}})} \leq
C\eta\bigl(\|u_{0}^d\|_{\dot{B}^{-1+\f dp}_{p,r}}+ c\mu\bigr) \eeq
for some small enough constant $c$ and $\eta$ given by
\eqref{small1}. }
\end{prop}

\begin{proof} For $N$ large enough, it is easy to prove that
\eqref{schema} has a unique local smooth solution $(a_n, u_n,
\na\Pi_n)$ on $[0,T_n^\ast)$ for some positive time $T_n^\ast.$
Without loss of generality, we may assume that $T_n^\ast$ is the
lifespan to this solution. It is easy to observe that \beq
\label{2.2aq} \|a_n\|_{L^\infty((0,T_n^\ast)\times\R^d)}\leq
\|a_0\|_{L^\infty}. \eeq

Next, for $\la_1, \la_2>0,$ we denote \beq
\label{2.3al}\begin{split} &f_{1,n}(t)\eqdefa \|\na
u^d_n(t)\|_{L^{\f{dr}{2r-1}}}^{2r},\quad f_{2,n}(t)\eqdefa
\|\D u_n^{d}(t)\|_{L^{\f{dr}{3r-2}}}^r,\\
&u_{\la,n}(t,x)\eqdefa u_n(t,x)\exp\Bigl\{-
\int_0^t\bigl(\la_1f_{1,n}(t')+\la_2f_{2,n}(t')\bigr)\,dt'\Bigr\},
\end{split} \eeq and similar notation for
$\Pi_{\la,n}(t,x).$ To deal with the pressure function $\Pi_n$ in
\eqref{schema}, we get, by taking divergence to the momentum
equation of \eqref{schema}, that \beno -\tri
\Pi_{\la,n}=\dive(a_{n}\na \Pi_{\la,n})-\mu \dive(a_{n}\tri
u_{\la,n})+\dive(u_{n}\cdot\na u_{\la,n}), \eeno from which,
\eqref{2.2aq},  and  the fact $\dive u_{n}=0$, so that \beno
\begin{split}
\dive(u_{n}\cdot\na u_{\la,n})=&\dive_h
(u_{n}^h\cdot\na_hu_{\la,n}^h)+\dive_h
(u_n^{d}\pa_du_{\la,n}^h)\\
&+\pa_d (u_{\la,n}^h\cdot\na_hu_n^{d}) -\pa_d
(u_{n}^h\dive_hu_{\la,n}^h),\end{split} \eeno we deduce that
\beq\label{2.3a}
\begin{split}
\|\na \Pi_{\la,n}(t)\|_{L^{\f{dr}{3r-2}}}\leq
&C\Bigl\{\|a_{0}\|_{L^\infty}\|\na
\Pi_{\la,n}(t)\|_{L^{\f{dr}{3r-2}}} +\mu\|a_{0}\|_{L^\infty}\|\tri
u_{\la,n}\|_{L^{\f{dr}{3r-2}}}\\
&+\bigl(\|u_{n}^h\|_{L^{\f{dr}{r-1}}}+\|u_n^{d}\|_{L^{\f{dr}{r-1}}}\bigr)\|\na
u_{\la,n}^h\|_{L^{\f{dr}{2r-1}}}\\
&+\|u_{\la,n}^h\|_{L^{\f{dr}{r-1}}}\|\na
u_n^{d}\|_{L^{\f{dr}{2r-1}}}\Bigr\}. \end{split} \eeq
 In
particular, if $\eta$ in \eqref{small1} is so small that
$C\|a_0\|_{L^\infty}\leq \f12$, we infer from \eqref{2.3a} that
\beq\label{pest}
\begin{split}
\|\na \Pi_{\la,n}(t)\|_{L^{\f{dr}{3r-2}}}\leq &
C\Bigl\{\mu\|a_0\|_{L^\infty}\|\tri
u_{\la,n}\|_{L^{\f{dr}{3r-2}}}+\|u_{n}^h\|_{L^{\f{dr}{r-1}}}\|\na
u_{\la,n}^h\|_{L^{\f{dr}{2r-1}}}\\
&+\|u_n^{d}\|_{L^{\f{dr}{r-1}}}\|\na
u_{\la,n}^h\|_{L^{\f{dr}{2r-1}}}+\|u_{\la,n}^h\|_{L^{\f{dr}{r-1}}}\|\na
u_n^{d}\|_{L^{\f{dr}{2r-1}}}\Bigr\}.
\end{split}
\eeq Notice on the other hand that we can also equivalently
reformulate the momentum equation of \eqref{schema} as
\beq\label{pesta} u_{n}^i= u_{n,L}^i+ \int_0^t
e^{\mu(t-s)\Delta}\bigl\{-u_{n}\nabla u_{n}^i+\mu a_{n}\Delta
u_{n}^i-(1+a_{n})\pa_i \Pi_{n}\bigr\}\,ds\quad\mbox{for}\quad
i=1,2,3,\eeq and $u_{n,L}\eqdefa e^{\mu t\D}S_{n+N}u_{0},$ from
which and \eqref{2.3al}, we write \beq \label{2.6}
\begin{split}
u_{\la,n}^h=& u_{n,L}^h\exp\Bigl\{-
\int_0^t(\la_1f_{1,n}(t')+\la_2f_{2,n}(t'))\,dt'\Bigr\}\\
&+ \int_0^t e^{\mu(t-s)\Delta}\exp\Bigl\{-
\int_s^t(\la_1f_{1,n}(t')+\la_2f_{2,n}(t'))\,dt'\Bigr\}\\
&\qquad\times\bigl\{-u_{n}\nabla u_{\la,n}^h+\mu a_{n}\Delta
u_{\la,n}^h-(1+a_{n})\na_h \Pi_{\la,n}\bigr\}\,ds.
\end{split}\eeq Applying Lemma \ref{lem1}, Lemma \ref{lem2a}, \eqref{2.2aq},  and \eqref{2.3a} leads to
\beq\label{2.6a}
\begin{split} \mu^{1-\f1{2r}}&\|\na u_{\la,n}^h\|_{L^{2r}_t(L^{\f{dr}{2r-1}})}+\mu\|\D u_{\la,n}^h\|_{L^r_t(L^{\f{dr}{3r-2}})}\\
\leq &\mu^{1-\f1{2r}}\|\na u_{n,L}^h\|_{L^{2r}_t(L^{\f{dr}{2r-1}})}+
\mu\|\D u_{L}^h\|_{L^r_t(L^{\f{dr}{3r-2}})}+C\Bigl\{
\mu\|a_0\|_{L^\infty}\|\D
u_{\la,n}^h\|_{L^r_t(L^{\f{dr}{3r-2}})}\\
&+\mu\|a_0\|_{L^\infty}\|\D
u_{\la,n}^{d}\|_{L^r_t(L^{\f{dr}{3r-2}})}+\|u_{n}^h\|_{L^{2r}_t(L^{\f{dr}{r-1}})}\|\na
u_{\la,n}^h\|_{L^{2r}_t(L^{\f{dr}{2r-1}})}\\
&+\Bigl(\int_{0}^te^{-
r\int_s^t(\la_1f_{1,n}(t')+\la_2f_{2,n}(t'))\,dt'}\bigl(\|u_n^{d}(s)\|_{L^{\f{dr}{r-1}}}^r\|\na
u_{\la,n}^h(s)\|_{L^{\f{dr}{2r-1}}}^r\\
&\qquad+\|u_{\la,n}^h\|_{L^{\f{dr}{r-1}}}^r\|\na
u_n^{d}(s)\|_{L^{\f{dr}{2r-1}}}^r\bigr)\,ds\Bigr)^{\f1r}\Bigr\}
\quad\mbox{for}\quad t\leq T_n^\ast.
\end{split}\eeq  Then as
$C\|a_0\|_{L^\infty}\leq \f12,$ and \beno
\begin{split} &\|u_{n}\|_{L^{2r}_t(L^{\f{dr}{r-1}})}\leq
C\|\na u_{n}\|_{L^{2r}_t(L^{\f{dr}{2r-1}})},\quad \|\D
u_{\la,n}^d\|_{L^r_t(L^{\f{dr}{3r-2}})}\leq\f1{(\la_2 r)^{\f1r}},\\
&\Bigl(\int_{0}^te^{-\la_1
r\int_s^tf_{1,n}(t')\,dt'}\|u_{n-1}^d(s)\|_{L^{\f{dr}{r-1}}}^r\|\na
u_{\la,n}^h(s)\|_{L^{\f{dr}{2r-1}}}^r\,ds\Bigr)^{\f1r}\leq
\f1{(\la_1 r)^{\f1{2r}}}\|\na
u_{\la,n}^h(s)\|_{L^{2r}_t(L^{\f{dr}{2r-1}})},
\end{split}
\eeno so that we infer from \eqref{2.6a} that \beno \begin{split}
\mu^{1-\f1{2r}}&\|\na
u_{\la,n}^h\|_{L^{2r}_t(L^{\f{dr}{2r-1}})}+\mu\|\D
u_{\la,n}^h\|_{L^r_t(L^{\f{dr}{3r-2}})}\\
\leq & C
\mu^{1-\f1r}\|u_0^h\|_{\dot{B}^{-1+\f{d}p}_{p,r}}+C\Bigl\{\f{\mu}{(\la_2
r)^{\f1r}}\|a_0\|_{L^\infty}+\|u_{n}^h\|_{L^{2r}_t(L^{\f{dr}{r-1}})}\|\na
u_{\la,n}^h\|_{L^{2r}_t(L^{\f{dr}{2r-1}})}\\
&+\f1{(\la_1 r)^{\f1{2r}}}\|\na
u_{\la,n}^h(s)\|_{L^{2r}_t(L^{\f{dr}{2r-1}})}\Bigr\}.
\end{split}\eeno
Taking $\la_1=\f{(4C)^{2r}}{\mu^{2r-1}r}$ and
$\la_2=\f{C^r}{r\mu^{r-1}}$ in the above inequality, we obtain
\beq\label{2.7} \begin{split}\f34&\mu^{1-\f1{2r}}\|\na
u_{\la,n}^h\|_{L^{2r}_t(L^{\f{dr}{2r-1}})}+ \mu\|\D
u_{\la,n}^h\|_{L^r_t(L^{\f{dr}{3r-2}})}\\
\leq & C
\mu^{1-\f1r}\|u_0^h\|_{\dot{B}^{-1+\f{d}p}_{p,r}}+\mu^{2-\f1r}\|a_0\|_{L^\infty}+\|u_{n}^h\|_{L^{2r}_t(L^{\f{dr}{r-1}})}\|\na
u_{\la,n}^h\|_{L^{2r}_t(L^{\f{dr}{2r-1}})}.
\end{split}\eeq
Let  $c_1$ be a small enough positive constant, which will be
determined later on, we denote \beq \label{2.8} \bar{T}_n\eqdefa
\sup\bigl\{\ T\leq T_n^\ast;\ \mu^{\f1{2r}}\|\na
u_{n}^h\|_{L^{2r}_T(L^{\f{dr}{r-1}})}+ \mu^{\f1r}\|\D
u_{n}^h\|_{L^r_T(L^{\f{dr}{3r-2}})}\leq c_1\mu\ \bigr\}. \eeq Then
it follows from \eqref{2.3al} and \eqref{2.7} that for
$t\leq\bar{T}_n$ \beq\label{2.9}
\begin{split} &\f12\mu^{\f1{2r}}\|\na
u_{n}^h\|_{L^{2r}_t(L^{\f{dr}{2r-1}})}+\mu^{\f1r}\|\D
u_{n}^h\|_{L^r_t(L^{\f{dr}{3r-2}})}\\
&\leq C
\bigl(\|u_0^h\|_{\dot{B}^{-1+\f{d}p}_{p,r}}+\mu\|a_0\|_{L^\infty}\bigr)\\
&\qquad\times\exp\Bigl\{-C_r\int_0^t\bigl( \mu^{1-2r}\|\na
u^d_n(s)\|_{L^{\f{dr}{2r-1}}}^{2r}+\mu^{1-r}\|\D
u^d_n(s)\|_{L^{\f{dr}{3r-2}}}^r\bigr)\,ds\Bigr\}.
\end{split}
\eeq

On the other hand, it follows from a similar derivation of
\eqref{2.6a} that \beno
\begin{split}\mu^{1-\f1{2r}}\|\na& u_n^{d}\|_{L^{2r}_t(L^{\f{dr}{2r-1}})}+ \mu\|\D u_n^{d}\|_{L^r_t(L^{\f{dr}{3r-2}})}\\
\leq & \mu^{1-\f1{2r}}\|\na
u_{n,L}^d\|_{L^{2r}_t(L^{\f{dr}{2r-1}})}+ \mu\|\D
u_{n,L}^d\|_{L^r(L^{\f{dr}{3r-2}})}+C\Bigl\{
\mu\|a_0\|_{L^\infty}\|\D u_{n}\|_{L^r_t(L^{\f{dr}{3r-2}})}\\
&+\bigl(\|u_{n}^h\|_{L^{2r}_t(L^{\f{dr}{r-1}})}+\|u_n^{d}\|_{L^{2r}_t(L^{\f{dr}{r-1}})}\bigr)\|\na
u_{n}^h\|_{L^{2r}_t(L^{\f{dr}{2r-1}})}\\
&+\|u_{n}^h\|_{L^{2r}_t(L^{\f{dr}{r-1}})}\|\na
u_n^{d}\|_{L^{2r}_t(L^{\f{dr}{2r-1}})}\Bigr\}\quad\mbox{for}\quad
t\leq T_n^\ast,
\end{split}\eeno
from which and \eqref{2.8}, we deduce that \beq\label{2.10}
\begin{split}
\mu^{\f1r}\|\D u_n^{d}&\|_{L^r_t(L^{\f{dr}{3r-2}})}+\mu^{\f1{2r}}\|\na u_n^{d}\|_{L^{2r}_t(L^{\f{dr}{2r-1}})}\\
\leq & 2C\|u_0^d\|_{\dot{B}^{-1+\f
dp}_{p,r}}+2Cc_1\mu(1+\|a_0\|_{L^\infty})\quad\mbox{for}\quad t\leq
\bar{T}_n.
\end{split}\eeq
Substituting \eqref{2.10} into \eqref{2.9} leads to \beq\label{2.11}
\begin{split} &\mu^{\f1r}\|\D
u_{n}^h\|_{L^r_t(L^{\f{dr}{3r-2}})}+\f12\mu^{\f1{2r}}\|\na
u_{n}^h\|_{L^{2r}_t(L^{\f{dr}{2r-1}})}\\
&\leq
C\bigl(\|u_0^h\|_{\dot{B}^{-1+\f{d}p}_{p,r}}+\mu\|a_0\|_{L^\infty}\bigr)\exp\Bigl\{C_r
\mu^{-2r}\|u_0^d\|_{\dot{B}^{-1+\f
dp}_{p,r}}^{2r}\Bigr\}\leq\f{c_1}2\mu\quad\mbox{for}\quad
t\leq\bar{T}_n,
\end{split}
\eeq as long as $C_r$ is sufficiently large and $c_0$ is small
enough in \eqref{small1}. This shows that $\bar{T}_n=T_n^\ast.$ Then
thanks to \eqref{2.10}, \eqref{2.11} and Theorem 1.3 in \cite{Kim},
we conclude that $T_n^\ast=\infty$ and there holds \eqref{2.2ad}. By
virtue of \eqref{2.2ad} and \eqref{2.3a}, we infer \beno
\begin{split}
\|\na\Pi_{n}\|_{L^r_t(L^{\f{dr}{3r-2}})}\leq &C\Bigl\{
\mu\|a_0\|_{L^\infty}\|\D
u_{n}\|_{L^r_t(L^{\f{dr}{3r-2}})}+\bigl(\|\na
u_{n}^h\|_{L^{2r}_t(L^{\f{dr}{2r-1}})}\\
&\qquad+\|\na u_n^{d}\|_{L^{2r}_t(L^{\f{dr}{2r-1}})}\bigr)\|\na
u_{n}^h\|_{L^{2r}_t(L^{\f{dr}{2r-1}})}\Bigr\}\\
\leq
&C\mu^{-\f1r}\eta\bigl(\|u_0^d\|_{\dot{B}^{-1+\f{d}p}_{p,r}}+\eta+c\mu\bigr),
\end{split}
\eeno for some small constant $c,$ which leads to \eqref{2.2ap}.
This completes the proof of the proposition.
\end{proof}

We now present the proof to the existence part of Theorem \ref{thm1.1} in the case when $1<p\leq \f{dr}{3r-2}.$\\

\no{\bf Proof to the existence part of Theorem \ref{thm1.1} for
$1<p\leq \f{dr}{3r-2}.$ }\ Indeed thanks to \eqref{schema} and
\eqref{2.2ad},  \eqref{2.2ap}, we infer that $\{\p_t u_n\}$ is
uniformly bounded in $L^r(\R^+;L^{\f{dr}{3r-2}}(\R^d)),$ from which,
\eqref{2.2ad}, \eqref{2.2ap},  and Ascoli-Arzela Theorem, we
conclude that there exists a subsequence of $\{a_n,u_n,\na\Pi_n\},$
which we still denote by $\{a_n,u_n, \na\Pi_n\}$ and some $(a,u, \na
\Pi)$ with $a\in L^\infty(\R^+\times\R^d),$ $\na u\in L^{2r}(\R^+;
L^{\f{dr}{2r-1}}(\R^d))$ and $\D u, \na\Pi\in
L^r(\R^+;L^{\f{dr}{3r-2}}(\R^d)),$ such that \beq \label{2.12}
\begin{split}
&a_n \rightharpoonup  a\quad \mbox{weak}\ \ast \ \mbox{in}\
L^\infty_{loc}(\R^+\times\R^d),\\
&  u_n \rightarrow u \quad \mbox{strongly}\  \ \mbox{in}\
L^{2r}_{loc}(\R^+; L^{\f{dr}{r-1}}_{loc}(\R^d)),\\
&  \na u_n \rightarrow \na u \quad \mbox{strongly}\  \ \mbox{in}\
L^{2r}_{loc}(\R^+; L^{\f{dr}{2r-1}}_{loc}(\R^d)),\\
& \D u_n \rightharpoonup \D u\quad\mbox{and}\quad\na\Pi_n
\rightharpoonup \na\Pi \ \ \mbox{weakly}\quad \mbox{in}\ L^r(\R^+;
L^{\f{dr}{3r-2}}(\R^d)).
\end{split}
\eeq Obviously, $(a_n,u_n, \na\Pi_n)$ satisfies \beq\label{2.16}
\begin{split}
&\int_0^\infty\int_{\R^d}a_n(\p_t\phi+u_n\cdot\na\phi)\,dx\,dt+\int_{\R^d}S_{n+N}a_0(x)\phi(0,x)\,dx=0,\\
&\qquad\qquad
\int_0^\infty\int_{\R^d}\dive u_n\phi\,dx\,dt=0\quad\mbox{and}\\
& \int_0^\infty\int_{\R^d}\Bigl\{u_n\cdot\p_t\Phi-(u_n\cdot\na
u_n)\cdot\Phi +(1+a_n)\bigl(\mu\D
u_n-\na\Pi_n\bigr)\cdot\Phi\Bigr\}\,dx\,dt\\
&\qquad\quad+\int_{\R^d}S_{n+N}u_0\cdot\Phi(0,x)\,dx=0,
\end{split}
\eeq for all the test function $\phi, \Phi$ given by Definition
\ref{defi1.1}.

Therefore thanks to \eqref{2.12}, to prove that $(a,u,\na\Pi)$
obtained in \eqref{2.12} is indeed a global weak solution of
\eqref{INS} in the sense of  Definition \ref{defi1.1}, we only need
to show that $\f1{1+a_n}\to \f1{1+a}$ almost everywhere in
$\R^+\times\R^d$ as $n\to \infty.$ Toward this, we shall follow the
compactness argument in  \cite{LP} to prove that $\{a_n\}$ strongly
converges to $a$ in $L^m_{loc}(\R^+\times\R^d)$ for any $m<\infty.$
In fact, it is easy to observe from the transport equation of
\eqref{schema} that \beno \p_t a_n^2+\dive (u_n a_n^2)=0, \eeno from
which, we deduce that \beq \p_t \overline{a^2}+\dive
(u\overline{a^2})=0, \label{2.13} \eeq where we denote
$\overline{a^2}$ to be the weak $\ast$ limit of $\{a_n^2\}.$

While thanks to \eqref{2.12} and \eqref{2.16}, there holds \beno
\p_ta+\dive(u a)=0 \eeno in the sense of distributions. Moreover, as
$\na u\in L^r(\R^+; L^{\f{dr}{2r-1}}(\R^d)),$ we infer by a
mollifying argument as that in \cite{DP} that \beq\label{2.14}
\p_ta^2+\dive(ua^2)=0. \eeq Subtracting \eqref{2.14} from
\eqref{2.13}, we obtain \beq\label{2.17} \p_t
(\overline{a^2}-a^2)+\dive (u(\overline{a^2}-a^2))=0. \eeq Notice
that $\{S_{n+N}a_0\}$ converges to $a_0$ almost everywhere in
$\R^d,$ which implies $(\overline{a^2}-a^2)(0,x)=0$ for a. e.
$x\in\R^d.$ Whence it follows from the uniqueness theorem for the
 transport equation in \cite{DP} that \beno
(\overline{a^2}-a^2)(t,x)=0\quad\mbox{for}\quad a. \ e.\ \
(t,x)\in\R^+\times\R^d, \eeno which together with
$\|a_n\|_{L^\infty}\leq\|a_0\|_{L^\infty}$  implies that
\beq\label{2.15} a_n\ \rightarrow\ a\quad \mbox{strongly in}\quad
L^m_{loc}(\R^+\times\R^d)\quad\mbox{for any}\ m<\infty. \eeq Thanks
to \eqref{2.12} and \eqref{2.15}, we can take $n\to\infty$ in
\eqref{2.16}  to verify that $(a,u, \na\Pi)$ obtained  in
\eqref{2.12} satisfies \eqref{def1.1a} and \eqref{def1.1b}.
Moreover, thanks to \eqref{2.2ad} and \eqref{2.2ap}, there holds
\eqref{thm1a}. This completes the proof to the existence part of
Theorem \ref{thm1.1} for $p\in (1, \f{dr}{3r-2}].$\ef

\setcounter{equation}{0}
\section{Proof  to the existence part of Theorem \ref{thm1.1} for $\f{dr}{3r-2}<p<d$}

In this case, as $-3+\f{d}{p}<-\f2r$,  it is impossible to find some
$p_1\geq p$ and $r_1\geq r$ so that $-3+\f{d}{p_1}=-\f2{r_1}$.
However, notice that for all $p_1\geq p$ and $r_1\geq r$, $\D u_0\in
\dot{B}^{-3+\f{d}{p_1}}_{p_1,r_1}(\R^d),$  then we deduce from
Proposition \ref{prop1.1} that $t^{\f12(3-\f{d}{p_1})-\f1{r_1}}\tri
e^{t\tri}u_0\in L^{r_1}(\R^+; L^{p_1}(\R^d)).$ Similarly, we choose
$p_2>\f{d}2$ and $p_3>d$ with $\f1{p_2}+\f1{p_3}=\f1{p_1}$, there
holds $t^{\f12(2-\f{d}{p_2})-\f1{r_1}}\na e^{t\tri}u_0\in
L^{r_1}(\R^+; L^{p_2}(\R^d))$ and $t^{\f12(1-\f{d}{p_3})}
e^{t\tri}u_0\in L^\infty(\R^+; L^{p_3}(\R^d)).$ With these time
weights before $e^{t\D}u_0,$ $\na e^{t\D}u_0$ and $\D e^{t\D}u_0,$
to prove the existence part of Theorem \ref{thm1.1} for
$\f{dr}{3r-2}<p<d$,
 we need to use the
following time-weighted version of maximal $L^p(L^q)$ regularizing
effect for the heat kernel:

\begin{lem}\label{lem3}
{\sl Let $1<p,q<\infty$ and $\alpha$ a non-negative real number
satisfying $\alpha +\f1p<1$. Let $\cA$ be the operator defined by
Lemma \ref{lem1}. Then if $t^\alpha f\in L^p((0,T);L^q(\R^d))$ for
some $T\in (0,\infty]$, $t^\alpha \cA f\in L^p((0,T);L^q(\R^d))$,
and there holds \beq\label{3.1kp}\|t^\alpha \cA f\|_{L^p_T(L^q)}\leq
\f{C}{\mu}\|t^\alpha f\|_{L^p_T(L^q)}.\eeq }
\end{lem}

\begin{proof}
We first split the operator $\cA$ as \beq\label{3.1a}\cA f =
\Bigl(\int_0^{\f{t}{2}}+\int_{\f{t}{2}}^t\Bigr)\tri
e^{\mu(t-s)\tri}f\,ds\eqdefa\cA_1f+\cA_2f.\eeq Note that when $s\in
[\f{t}2,t],$ $t^\al$ is comparable to $s^\al,$ so that it follows
from the proof of Lemma 7.3 in \cite{Lem} that
$$\|t^\alpha \cA_2 f\|_{L^p_T(L^q)}\leq \f{C}{\mu}\|t^\alpha f\|_{L^p_T(L^q)}.$$
To handle $\cA_1f$ in \eqref{3.1a}, we write
$$t^\alpha\cA_1 f(t) = \int_0^{\f{t}{2}} \f{t^\alpha}{\mu(t-s)s^{\alpha}}\mu(t-s)\tri e^{\mu(t-s)\tri}(s^\alpha f)\,ds.$$
As $\mu t\tri e^{\mu t\tri}$ is a bounded operator from $L^q(\R^d)$
to $L^q(\R^d)$, we have
$$\|t^\alpha\cA_1 f(t)\|_{L^q(\R^d)}\leq \f{C}{\mu} \int_0^{\f{t}{2}} \f{t^\alpha}{(t-s)s^{\alpha}}\|s^\alpha f(s)\|_{L^q(\R^d)}\,ds.$$
Let $F_\al(s)\eqdefa\|s^\alpha f(s)\|_{L^q(\R^d)}$ and using the
change of variable that $s=t\tau,$  we obtain
$$\|t^\alpha\cA_1 f(t)\|_{L^q(\R^d)}\leq \f{C}{\mu} \int_0^{\f12}\f1{(1-\tau)\tau^\alpha}F_\al(t\tau)\,d\tau,$$
from which and Minkowski inequality, we infer
$$\|t^\alpha\cA_1 f\|_{L^p_T(L^q)}\leq \f{C}{\mu} \int_0^{\f12}\f1{(1-\tau)\tau^{\alpha+\f1p}}\,d\tau \|F_\al\|_{L^p_T},$$
which along with the fact: $\alpha +\f1p<1$,  concludes the proof of
\eqref{3.1kp}.
\end{proof}

In order to heal with the estimate of $u$ and $\na u,$ we need the
following lemmas, which will be used in the proof of  both the
existence part and uniqueness part of Theorem \ref{thm1.1}.

\begin{lem}\label{lem4}
{\sl Let the operator $\cB$ be defined by $f(t,x)\mapsto \int_0^t
\na e^{\mu(t-s)\tri}f\,ds$ and  $\cC$  by $f(t,x)\mapsto \int_0^t
e^{\mu(t-s)\tri}f\,ds$. Let $\e\geq 0$ be a small enough number,
$r_1>1,$ and $q_1$, $q_2,$  $q_3$ satisfy $ \f{d{r_1}}{2{r_1}-2}
<q_1<\f{d}{1-\ep}$, $d<q_3<\infty$ and $\f1{q_2}+\f1{q_3}=\f1{q_1}$.
We denote \beq\label{index}
\wt{\alpha}^\ep=\f12(3-\f{d}{q_1}-\ep)-\f1{r_1},\quad
\wt{\beta}_1^\ep=\f12(2-\f{d}{q_2}-\ep)-\f1{r_1} \andf
\wt{\gamma}_1^\ep=\f12(1-\f{d}{q_3}-\ep).\eeq  Then if
$t^{\wt{\alpha}^\ep} f\in L^{r_1}((0,T); L^{q_1}(\R^d))$ for some
$T\in (0,\infty]$, $t^{\wt{\beta}_1^\ep} \cB f\in L^{r_1}((0,T);
L^{q_2}(\R^d))$ and $t^{\wt{\gamma}_1^\ep} \cC f\in
L^\infty((0,T);L^{q_3}(\R^d))$, and there holds \beq\label{2.2kj}
\|t^{\wt{\beta}_1^\ep} \cB f\|_{L^{r_1}_T(L^{q_2})}\leq
\f{C_\e}{\mu^{\f12+\f{d}{2q_3}}}\|t^{\wt{\alpha}^\ep}
f\|_{L^{r_1}_T(L^{q_1})},\eeq  and \beq\label{2.3kn}
\|t^{\wt{\gamma}_1^\ep} \cC f\|_{L^\infty_T(L^{q_3})}\leq
\f{C_\e}{\mu^{\f{d}{2q_2}}}\|t^{\wt{\alpha}^\ep}
f\|_{L^{r_1}_T(L^{q_1})}. \eeq }
\end{lem}

\begin{proof} We first get, by applying  Young's inequality to
\eqref{2.1a}, that \beq\label{2.2b}\begin{split}
 \|\na
e^{\mu(t-s)\tri}f(s,\cdot)\|_{L^{q_2}} \leq&
C(\mu(t-s))^{-\f{d+1}2}\|K(\f{\cdot}{\sqrt{4\mu(t-s)}})\|_{L^{\f{q_3}{q_3-1}}} \|f(s)\|_{L^{q_1}}\\
\leq& C(\mu(t-s))^{-(\f12+\f{d}{2q_3})}\|f(s)\|_{L^{q_1}}.
\end{split} \eeq
Then let $F^\e(s)\eqdefa\|s^{\wt{\alpha}^\ep} f(s)\|_{L^{q_1}}$, we
have\beno t^{\wt{\beta}_1^\ep}\| \cB f(t)\|_{L^{q_2}} \leq
C\mu^{-(\f12+\f{d}{2q_3})}
 t^{\wt{\beta}_1^\ep}\int_0^t (t-s)^{-(\f12+\f{d}{2q_3})} s^{-\wt{\alpha}^\ep}F^\e(s)\,ds.\eeno
Using change of variable with $s=t\tau$ and  the fact that
$\wt{\beta}_1^\ep-(\f12+\f{d}{2q_3})-\wt{\alpha}^\ep+1=0,$  we
obtain \beno t^{\wt{\beta}_1^\ep}\| \cB f(t)\|_{L^{q_2}} \leq
C\mu^{-(\f12+\f{d}{2q_3})} \int_0^1(1-\tau)^{-(\f12+\f{d}{2q_3})}
\tau^{-\wt{\alpha}^\ep}F^\e(t\tau)\,d\tau.\eeno Applying Minkowski
inequality leads to  \beno \begin{split} \|t^{\wt{\beta}_1^\ep} \cB
f\|_{L^{r_1}_T(L^{q_2})}\leq& C\mu^{-(\f12+\f{d}{2q_3})} \int_0^1
(1-\tau)^{-(\f12+\f{d}{2q_3})}
\tau^{-\wt{\alpha}^\ep-\f1{r_1}}\,d\tau\|F^\e\|_{L^{r_1}_T}\\
\leq& C\mu^{-(\f12+\f{d}{2q_3})} \int_0^1
(1-\tau)^{-(\f12+\f{d}{2q_3})}\tau^{-\f12(3-\f{d}{q_1}-\ep)}\,d\tau\|t^{\wt{\alpha}^\ep}
f\|_{L^{r_1}_T(L^{q_1})}. \end{split}\eeno Note that the assumption:
$q_3>d$, which together with $q_1<\f{d}{1-\ep}$ implies that
$0<\f12+\f{d}{2q_3}, \f12(3-\f{d}{q_1}-\ep)<1.$ This proves
\eqref{2.2kj}.

To deal with $\cC f$, we write
$$e^{\mu(t-s)\tri}f= (\mu(t-s))^{-\f{d}{2q_2}}e^{-\mu(t-s)\tri}
(-\mu(t-s)\tri)^{\f{d}{2q_2}}(- \tri)^{-\f{d}{2q_2}}f.$$ Observing
that for any $\delta>0$, $e^\tri (-\tri)^{\delta}$ is a bounded
linear operator from $L^{q_3}(\R^d)$ to $L^{q_3}(\R^d)$, and
$(-\tri)^{-\delta} f= |\cdot|^{-(d-2\delta)}\ast f$ for
$0<2\delta<d$, so that applying Hardy-Littlewood-Sobolev inequality
gives rise to \beq\label{2.2c}
\begin{split} \|e^{\mu(t-s)\tri}f(s)\|_{L^{q_3}}\leq&
C(\mu(t-s))^{-\f{d}{2q_2}}
\|(-\tri)^{-\f{d}{2q_2}}f(s)\|_{L^{q_3}}\\
\leq&C(\mu(t-s))^{-\f{d}{2q_2}}\|f(s)\|_{L^{q_1}}. \end{split}\eeq
Let $r_1'$ be the conjugate index of $r_1$, we get, by applying
H\"older's inequality, that \beq\label{2.2d} \begin{split} \|\cC
f(t)\|_{L^{q_3}}\leq& C\mu^{-\f{d}{2q_2}}\int_0^t
(t-s)^{-\f{d}{2q_2}}s^{-\wt{\alpha}^\ep}F^\e(s)\,ds\\
\leq& C\mu^{-\f{d}{2q_2}}\big( \int_0^t
(t-s)^{-\f{d}{2q_2}r_1'}s^{-\wt{\alpha}^\ep
r_1'}\,ds\big)^{\f1{r_1'}}\|t^{\wt{\alpha}^\ep}
f\|_{L^{r_1}_T(L^{q_1})}.\end{split}\eeq While using once again the
change of variable that $s=t\tau$, we obtain  \beno \begin{split}
\int_0^t (t-s)^{-\f{d}{2q_2}r_1'}s^{-\wt{\alpha}^\ep r_1'}\,ds
=&t^{1-\f{d}{2q_2}r_1'-\wt{\alpha}^\ep r_1'}\int_0^1
(1-\tau)^{-\f{d}{2q_2}r_1'}\tau^{-\wt{\alpha}^\ep r_1'}\,d\tau\\
=&t^{-\wt{\gamma}_1^\ep r_1'}\int_0^1
(1-\tau)^{-\f{d}{2q_2}r_1'}\tau^{-\wt{\alpha}^\ep r_1'}\,d\tau.
\end{split} \eeno Recalling the  assumptions that $q_2>q_1>\f{dr_1}{2r_1-2}$ and $q_1<\f{d}{1-\ep}$,
which imply $\f{d}{2p_2}r_1'<1$ and $\wt{\alpha}^\ep r_1'<1$. This
together with \eqref{2.2d} ensures that \beno \|\cC f(t)\|_{L^{q_3}}
\leq C_\e\mu^{-\f{d}{2q_2}}
t^{-\wt{\gamma}_1^\ep}\|t^{\wt{\alpha}^\ep}
f\|_{L^{r_1}_T(L^{q_1})}. \eeno This concludes the proof of
\eqref{2.3kn} and Lemma \ref{lem4}.
\end{proof}

In order to deal with the term $u^d\pa_du$ in \eqref{pesta}, we need
also the following lemma.

\begin{lem}\label{lem5}
{\sl Let $\e,$ $r_1$, $q_1$, $\wt{\alpha}^\ep$ and the operators
$\cB$, $\cC$ be given by Lemma \ref{lem4}. Let $r_1>2$ and $q_2$,
$q_3$ satisfy $\f1{q_2}+\f1{q_3}=\f1{q_1}$,
$\f{dr_1}{r_1-2}<q_3<\infty,$ we denote \beq\label{index1}
\wt{\beta}_2^\ep=\f12(2-\f{d}{q_2}-\ep) \andf
\wt{\gamma}_2^\ep=\f12(1-\f{d}{q_3}-\ep)-\f1{r_1}.\eeq Then if
$t^{\wt{\alpha}^\ep} f\in L^{r_1}((0,T);L^{q_1}(\R^d))$ for some
$T\in (0,\infty]$, $t^{\wt{\beta}_2^\ep} \cB f\in
L^\infty((0,T);L^{q_2}(\R^d)),$ $t^{\wt{\gamma}_2^\ep} \cC f\in
L^{r_1}((0,T);L^{q_3}(\R^d)),$ and there holds \beq\label{2.4}
\|t^{\wt{\beta}_2^\ep} \cB f\|_{L^{\infty}_T(L^{q_2})}\leq
\f{C_\e}{\mu^{\f12+\f{d}{2q_3}}}\|t^{\wt{\alpha}^\ep}
f\|_{L^{r_1}_T(L^{q_1})},\eeq and \beq\label{2.5}
\|t^{\wt{\gamma}_2^\ep} \cC f\|_{L^{r_1}_T(L^{q_3})}\leq
\f{C_\e}{\mu^{\f{d}{2q_2}}}\|t^{\wt{\alpha}^\ep}
f\|_{L^{r_1}_T(L^{q_1})}. \eeq }
\end{lem}
\begin{proof}
We first get, by a similar derivation of \eqref{2.2b}, that \beno
\begin{split}
\|\cB f(t)\|_{L^{q_2}} \leq& C\mu^{-(\f12+\f{d}{2q_3})} \int_0^t
(t-s)^{-(\f12+\f{d}{2q_3})}
s^{-\wt{\alpha}^\ep}F^\e(s)\,ds\\
\leq&C\mu^{-(\f12+\f{d}{2q_3})} \Bigl( \int_0^t
(t-s)^{-(\f12+\f{d}{2q_3})r_1'} s^{-\wt{\alpha}^\ep
r_1'}\,ds\Bigr)^{\f1{r_1'}}\|t^{\wt{\alpha}^\ep}
f\|_{L^{r_1}_T(L^{q_1})}.
\end{split}\eeno Whereas using the change of variable: $s=t\tau,$ leads to \beno
\begin{split} \int_0^t (t-s)^{-(\f12+\f{d}{2q_3})r_1'}s^{-\wt{\alpha}^\ep r_1'}\,ds
=&t^{1-(\f12+\f{d}{2q_3})r_1'-\wt{\alpha}^\ep r_1'}\int_0^1
(1-\tau)^{-(\f12+\f{d}{2q_3})r_1'}\tau^{-\wt{\alpha}^\ep r_1'}\,d\tau\\
=&t^{-\wt{\beta}_2^\ep r_1'}\int_0^1
(1-\tau)^{-(\f12+\f{d}{2q_3})r_1'}\tau^{-\wt{\alpha}^\ep
r_1'}\,d\tau.
\end{split} \eeno By virtue of the assumptions: $q_3> \f{dr_1}{r_1-2}$ and
$q_1<\f{d}{1-\ep},$ we have $0<(\f12+\f{d}{2q_3})r_1',\ \alpha^\ep
r_1'<1$. As a consequence, we obtain \beno \|\cB
f(t)\|_{L^{q_2}}\leq C_\e\mu^{-(\f12+\f{d}{2q_3})}
t^{-\wt{\beta}_2^\ep}\|t^{\wt{\alpha}^\ep} f\|_{L^{r_1}_T(L^{q_1})},
\eeno which yields \eqref{2.4}.

On the other hand, it follows the same line of \eqref{2.2c} that
\beno\|e^{\mu(t-s)\tri}f(s)\|_{L^{q_3}} \leq
C(\mu(t-s))^{-\f{d}{2q_2}}\|f(s)\|_{L^{q_1}}, \eeno which implies
\beno t^{\wt{\gamma}_2^\ep}\|\cC f(t)\|_{L^{q_3}} \leq
C\mu^{-\f{d}{2q_2}}
t^{\wt{\gamma}^\ep_2}\int_0^t(t-s)^{-\f{d}{2q_2}}
s^{-\wt{\alpha}^\ep}F^\e(s)\,ds. \eeno Using changing of variable
with $s=t\tau$ and the fact:
$\wt{\gamma}_2^\ep-\f{d}{2q_2}-\wt{\alpha}^\ep+1=0,$ we obtain \beno
t^{\wt{\gamma}_2^\ep}\|\cC f(t)\|_{L^{q_3}} \leq C\mu^{-\f{d}{2q_2}}
\int_0^1(1-\tau)^{-\f{d}{2q_2}} \tau^{-\wt{\alpha}^\ep}
F^\e(t\tau)\,d\tau,\eeno from which, we deduce \beno
\begin{split} \|t^{\wt{\gamma}_2^\ep} \cC
f\|_{L^{r_1}_T(L^{q_3})}\leq& C\mu^{-\f{d}{2q_2}} \int_0^1
(1-\tau)^{-\f{d}{2q_2}}
\tau^{-\wt{\alpha}^\ep-\f1{r_1}}\,d\tau\|F^\e\|_{L^{r_1}_T}\\
\leq& C\mu^{-\f{d}{2q_2}} \int_0^1
(1-\tau)^{-\f{d}{2q_2}}\tau^{-\f12(3-\f{d}{q_1}-\ep)}\,d\tau\|t^{\wt{\alpha}^\ep}
f\|_{L^{r_1}_T(L^{q_1})}, \end{split} \eeno which along with the
facts: $q_2>q_1>\f{dr_1}{2r_1-2}>\f{d}2$,  ensures the integral
above is finite. This gives \eqref{2.5} and we complete the proof of
the lemma.
\end{proof}

In what follows, we take $r_1=2r>2$ and $p_1, p_2$ and $p_3$
satisfying $\max (p, \f{dr}{2r-1})<p_1<d$,
 and $\f{dr}{r-1}<p_3<\infty$ so that $\f1{p_2}+\f1{p_3}=\f1{p_1}.$
 We shall seek a solution $(a,u,\na\Pi)$
of \eqref{INS} in the following functional space: \beq
\label{3.6}\begin{split} X=\Bigl\{\ u: \, \,&a\in
L^\infty(\R^+\times\R^d),\quad t^{\gamma_1} u\in L^\infty(\R^+;
L^{p_3}(\R^d)), \\
&t^{\gamma_2} u\in L^{2r}(\R^+; L^{p_3}(\R^d)),\quad t^{\beta_1}\na
u \in L^{2r}(\R^+;L^{p_2}(\R^d)),\, \, \\
 &t^{\beta_2}\na u \in
L^{\infty}(\R^+;L^{p_2}(\R^d))£¬\quad t^{\alpha_1}\tri u \in
L^{2r}(\R^+; L^{p_1}(\R^d))\,\Bigr\},\end{split}\eeq with the
indices  $\alpha_1$, $\beta_1$, $\beta_2$, $\gamma_1$, $\gamma_2$
being determined by \eqref{indexad}.

 We construct the approximate solution
 $(a_n,u_n, \na\Pi_n)$  through \eqref{schema}. Similar to Proposition \ref{uniformbound}, we  have the following proposition concerning
the uniform time-weighted bounds of $(a_n,u_n)$.

\begin{prop}\label{uniformbound1}
{\sl Under the assumptions of Theorem \ref{thm1.1}, \eqref{schema}
has a unique global smooth solution $(a_n,u_n,\na\Pi_n)$ which
satisfies
 \beq\label{3.7}
\begin{split} \mu^{\f12(3-\f{d}{p_1})}&\|t^{\alpha_1} \D u_{n}^h\|_{L^{2r}(\R^+; L^{p_1})}
+ \mu^{1-\f{d}{2p_2}}\bigl(\|t^{\beta_1}\na u_{n}^h\|_{L^{2r}(\R^+;
L^{p_2})}+
 \|t^{\beta_2}\na u_{n}^h\|_{L^{\infty}(\R^+; L^{p_2})}\bigr)\\
&+
\mu^{\f12(1-\f{d}{p_3})}\bigl(\|t^{\gamma_1}u_{n}^h\|_{L^{\infty}(\R^+;L^{p_3})}
+
\|t^{\gamma_2}u_{n}^h\|_{L^{2r}(\R^+;L^{p_3})}\bigr)\leq C\eta,\\
\mu^{\f12(3-\f{d}{p_1})}&\|t^{\alpha_1} \D u_{n}^d\|_{L^{2r}(\R^+;
L^{p_1})}
 + \mu^{1-\f{d}{2p_2}}\bigl(\|t^{\beta_1}\na u_{n}^d\|_{L^{2r}(\R^+; L^{p_2})}+\|t^{\beta_2}\na u_{n}^d\|_{L^{\infty}(\R^+; L^{p_2})}\bigr)
 \\
 &+\mu^{\f12(1-\f{d}{p_3})}\bigl(\|t^{\gamma_1}u_{n}^d\|_{L^{\infty}(\R^+;L^{p_3})}
+\|t^{\gamma_2}u_{n}^d\|_{L^{2r}(\R^+;L^{p_3})}\bigr)\leq
C\|u_{0}^d\|_{\dot{B}^{-1+\f dp}_{p,r}}+ c\mu,\end{split}\eeq and
\beq\label{3.17}\begin{split}
&\mu^{\f12(3-\f{d}{p_1})}\|t^{\alpha_2} \D u_n\|_{L^{r}(\R^+;
L^{p_1})}\leq C\bigl(\|
u_{0}\|_{\dot{B}^{-1+\f{d}p}_{p,r}}+\f{\eta}\mu(\|
u_{0}^d\|_{\dot{B}^{-1+\f{d}p}_{p,r}}+c\mu+\eta)\bigr),\\
&\mu^{\f12(3-\f{d}{p_1})}\bigl(\|t^{\alpha_1}
\na\Pi_n\|_{L^{2r}(\R^+; L^{p_1})}+\|t^{\alpha_2}
\na\Pi_n\|_{L^{r}(\R^+; L^{p_1})}\bigr)\leq
C\eta\bigl(\|u_{0}^d\|_{\dot{B}^{-1+\f dp}_{p,r}}+ c\mu\bigr)
\end{split}
\eeq for some small enough constant $c,$ $\eta$ being given by
\eqref{small1}  and $\al_2$  by \eqref{indexad}.}
\end{prop}

\begin{proof} For $N$ large enough, it is easy to prove that
\eqref{schema} has a unique local smooth solution $(a_n, u_n)$ on
$[0,T_n^\ast).$ Without loss of generality, we may assume that
$T_n^\ast$ is the lifespan to this solution. For $\la_1, \la_2,
\la_3>0,$  we denote \beq \label{3.8}\begin{split}
&f_{1,n}(t)\eqdefa \|t^{\beta_1}\na u^d_n(t)\|_{L^{p_2}}^{2r} ,\quad
f_{2,n}(t)\eqdefa \|t^{\gamma_2}u^d_{n}(t)\|_{L^{p_3}}^{2r},\quad
f_{3,n}(t)\eqdefa
\| t^{\alpha_1}\D u_{n}^d(t)\|_{L^{p_1}}^{2r},\\
&u_{\la,n}(t,x)\eqdefa u_n(t,x)\exp\Bigl\{-
\int_0^t\bigl(\la_1f_{1,n}(t')+\la_2f_{2,n}(t')+\la_3f_{3,n}(t')\bigr)\,dt'\Bigr\},
\end{split} \eeq and similar notation for
$\Pi_{\la,n}(t,x).$ Then  we deduce by a similar derivation of
\eqref{2.3a} that \beno
\begin{split}
\|\na \Pi_{\la,n}(t)\|_{L^{p_1}}\leq
&C\Bigl\{\|a_{0}\|_{L^\infty}\|\na \Pi_{\la,n}(t)\|_{L^{p_1}}
+\mu\|a_{0}\|_{L^\infty}\|\tri
u_{\la,n}\|_{L^{p_1}}+\|u_{n}^h\|_{L^{p_3}}\|\na
u_{\la,n}^h\|_{L^{p_2}}\\
&+\|u_n^{d}\|_{L^{p_3}}\|\na
u_{\la,n}^h\|_{L^{p_2}}+\|u_{\la,n}^h\|_{L^{p_3}}\|\na
u_n^{d}\|_{L^{p_2}}\Bigr\}.
\end{split} \eeno
 In
particular, if $\eta$ in \eqref{small1} is so small that
$C\|a_0\|_{L^\infty}\leq \f12$, we obtain \beq\label{pest1}
\begin{split}
\|\na \Pi_{\la,n}(t)\|_{L^{p_1}}\leq &
C\Bigl\{\mu\|a_0\|_{L^\infty}\|\tri
u_{\la,n}\|_{L^{p_1}}+\|u_{n}^h\|_{L^{p_3}}\|\na
u_{\la,n}^h\|_{L^{p_2}}\\
&+\|u_n^{d}\|_{L^{p_3}}\|\na
u_{\la,n}^h\|_{L^{p_2}}+\|u_{\la,n}^h\|_{L^{p_3}}\|\na
u_n^{d}\|_{L^{p_2}}\Bigr\}.
\end{split}
\eeq While it follows form \eqref{pesta} and \eqref{3.8} that \beq
\label{3.9}
\begin{split}
u_{\la,n}^h=& u_{n,L}^h\exp\Bigl\{-
\int_0^t(\la_1f_{1,n}(t')+\la_2f_{2,n}(t')+\la_3f_{3,n}(t'))\,dt'\Bigr\}\\
&+ \int_0^t e^{\mu(t-s)\Delta}\exp\Bigl\{-
\int_s^t(\la_1f_{1,n}(t')+\la_2f_{2,n}(t')+\la_3f_{3,n}(t'))\,dt'\Bigr\}\\
&\qquad\times\bigl(-u_{n}\nabla u_{\la,n}^h+\mu a_{n}\Delta
u_{\la,n}^h-(1+a_{n})\na_h \Pi_{\la,n}\bigr)\,ds.
\end{split}\eeq
Applying Lemma \ref{lem4} for $\e=0$ and $r_1=2r$ gives \beno
\begin{split} \mu^{\f{d}{2p_2}}&\|t^{\gamma_1}u_{\la,n}^h\|_{L^{\infty}_t(L^{p_3})}
 + \mu^{\f12+\f{d}{2p_3}}\|t^{\beta_1}\na
 u_{\la,n}^h\|_{L^{2r}_t(L^{p_2})}\\
 \leq & \mu^{\f{d}{2p_2}}\|t^{\gamma_1}u_{n,L}^h\|_{L^{\infty}_t(L^{p_3})}
 + \mu^{\f12+\f{d}{2p_3}}\|t^{\beta_1}\na
 u_{n,L}^h\|_{L^{2r}_t(L^{p_2})}\\
 &+\bigl\|\exp\bigl\{-
\int_s^t(\la_1f_{1,n}(t')+\la_2f_{2,n}(t')+\la_3f_{3,n}(t'))\,dt'\bigr\}\\
&\qquad\times s^{\alpha_1}\bigl(-u_{n}\nabla u_{\la,n}^h+\mu
a_{n}\Delta u_{\la,n}^h-(1+a_{n})\na_h
\Pi_{\la,n}\bigr)\bigr\|_{L^{2r}_t(L^{p_1})},
 \end{split}
 \eeno for indices $\beta_1,\gamma_1$ given by \eqref{indexad}.

Similarly, it follows from Lemma \ref{lem3} and Lemma \ref{lem5}
for $\e=0$ and $r_1=2r$ that \beno
\begin{split}
 \mu^{\f{d}{2p_2}}&\|t^{\gamma_2}u_{\la,n}^h\|_{L^{2r}_t(L^{p_3})}
 +\mu^{\f12+\f{d}{2p_3}}\|t^{\beta_2}\na u_{\la,n}^h\|_{L^{\infty}_t(L^{p_2})}+\mu\|t^{\alpha_1} \D u_{\la,n}^h\|_{L^{2r}_t(L^{p_1})}\\
\leq &\mu^{\f{d}{2p_2}}\|t^{\gamma_2}u_{n,L}^h\|_{L^{2r}_t(L^{p_3})}
 +\mu^{\f12+\f{d}{2p_3}}\|t^{\beta_2}\na u_{n,L}^h\|_{L^{\infty}_t(L^{p_2})}+\mu\|t^{\alpha_1} \D u_{n,L}^h\|_{L^{2r}_t(L^{p_1})}\\
 &+\bigl\|\exp\bigl\{-
\int_s^t(\la_1f_{1,n}(t')+\la_2f_{2,n}(t')+\la_3f_{3,n}(t'))\,dt'\bigr\}\\
&\qquad\times s^{\alpha_1}\bigl(-u_{n}\nabla u_{\la,n}^h+\mu
a_{n}\Delta u_{\la,n}^h-(1+a_{n})\na_h
\Pi_{\la,n}\bigr)\bigr\|_{L^{2r}_t(L^{p_1})},
 \end{split}
 \eeno
 for indices $\al_1, \beta_2,\gamma_2$ given by \eqref{indexad}.

As a consequence, we obtain, by using \eqref{2.2aq}, that
 \beq\label{3.10}
\begin{split} \mu^{\f{d}{2p_2}}&\bigl(\|t^{\gamma_1}u_{\la,n}^h\|_{L^{\infty}_t(L^{p_3})} +\|t^{\gamma_2}u_{\la,n}^h\|_{L^{2r}_t(L^{p_3})}\bigr)
 + \mu^{\f12+\f{d}{2p_3}}\bigl(\|t^{\beta_1}\na u_{\la,n}^h\|_{L^{2r}_t(L^{p_2})}\\
 &+\|t^{\beta_2}\na u_{\la,n}^h\|_{L^{\infty}_t(L^{p_2})}\bigr)
 +\mu\|t^{\alpha_1} \D u_{\la,n}^h\|_{L^{2r}_t(L^{p_1})}\\
\leq &
\mu^{-\f12(1-\f{d}{p_1})}\|u_0^h\|_{\dot{B}^{-1+\f{d}{p_1}}_{p_1,2r}}+C\Bigl\{
\mu\|a_0\|_{L^\infty}\bigl(\|t^{\alpha_1} \D u_{\la,n}^h\|_{L^{2r}_t(L^{p_1})}\\
&+\|t^{\alpha_1} \D
u_{\la,n}^d\|_{L^{2r}_t(L^{p_1})}\bigr)+\|t^{\gamma_1}u_{n}^h\|_{L^{\infty}_t(L^{p_3})}\|t^{\beta_1}\na
u_{\la,n}^h\|_{L^{2r}_t(L^{p_2})}\\
&+\Bigl(\int_{0}^te^{-
2r\int_s^t(\la_1f_{1,n}(t')+\la_2f_{2,n}(t'))\,dt'}\bigl(\|s^{\gamma_2}u_{n}^d(s)\|_{L^{p_3}}^{2r}\|s^{\beta_2}\na
u_{\la,n}^h(s)\|_{L^{p_2}}^{2r}\\
&\qquad+\|s^{\gamma_1}u_{\la,n}^h\|_{L^{p_3}}^{2r}\|s^{\beta_1}\na
u_{n}^d(s)\|_{L^{p_2}}^{2r}\bigr)\,ds\Bigr)^{\f1{2r}}\Bigr\}.
\end{split}\eeq  Then as
$C\|a_0\|_{L^\infty}\leq \f12,$ and \beno
\begin{split} &\|s^{\alpha_1}\D u_{\la,n}^d\|_{L^{2r}_t(L^{p_1})}\leq\f1{(2\la_3 r)^{\f1{2r}}},\\
&\Bigl(\int_{0}^te^{-2\la_1
r\int_s^tf_{1,n}(t')\,dt'}\|s^{\gamma_1}u_{\la,n}^h\|_{L^{p_3}}^{2r}\|t^{\beta_1}\na
u_{n}^d(s)\|_{L^{p_2}}^{2r} \,ds\Bigr)^{\f1{2r}}\leq \f1{(2\la_1
r)^{\f1{2r}}}\|s^{\gamma_1} u_{\la,n}^h\|_{L^{\infty}_t(L^{p_3})},\\
&\Bigl(\int_{0}^te^{-2\la_2 r\int_s^tf_{2,n}(t')\,dt'}
 \|s^{\gamma_2}u_{n}^d(s)\|_{L^{p_3}}^{2r}\|t^{\beta_2}\na
u_{\la,n}^h(s)\|_{L^{p_2}}^{2r}\,ds\Bigr)^{\f1{2r}}\leq \f1{(2\la_2
r)^{\f1{2r}}} \|s^{\beta_2}\na
u_{\la,n}^h\|_{L^{\infty}_t(L^{p_2})},
\end{split}
\eeno  we infer from \eqref{3.10} that \beno \begin{split}
\mu^{\f{d}{2p_2}}&\bigl(\|t^{\gamma_1}u_{\la,n}^h\|_{L^{\infty}_t(L^{p_3})}
+\|t^{\gamma_2}u_{\la,n}^h\|_{L^{2r}_t(L^{p_3})}\bigr)
+ \mu^{\f12+\f{d}{2p_3}}\bigl(\|t^{\beta_2}\na u_{\la,n}^h\|_{L^{\infty}_t(L^{p_2})} \\
 &+\|t^{\beta_1}\na u_{\la,n}^h\|_{L^{2r}_t(L^{p_2})}\bigr)
 +\mu\|t^{\alpha_1} \D u_{\la,n}^h\|_{L^{2r}_t(L^{p_1})}\\
\leq &
\mu^{-\f12(1-\f{d}{p_1})}\|u_0\|_{\dot{B}^{-1+\f{d}{p_1}}_{p_1,r}}+C\Bigl\{\f{\mu}{(2\la_3
r)^{\f1{2r}}}\|a_0\|_{L^\infty}+\|t^{\gamma_1}u_{n}^h\|_{L^{\infty}_t(L^{p_3})}\|t^{\beta_1}\na
u_{\la,n}^h\|_{L^{2r}_t(L^{p_2})}\\
&+\f1{(2\la_1 r)^{\f1{2r}}}\|t^{\gamma_1}
u_{\la,n}^h\|_{L^{\infty}_t(L^{p_3})}   +\f1{(2\la_2 r)^{\f1{2r}}}
\|t^{\beta_2}\na u_{\la,n}^h\|_{L^{\infty}_t(L^{p_2})}\Bigr\}.
\end{split}\eeno
Taking $\la_1=\f{(4C)^{2r}}{2r\mu^{\f{d}{p_2}r}}$,
$\la_2=\f{(4C)^{2r}}{2r\mu^{(1+\f{d}{p_3})r}}$ and
$\la_3=\f{C^{2r}}{2r\mu^{(\f{d}{p_1}-1)r}}$ in the above inequality
results in \beq\label{3.11} \begin{split}
\mu^{\f{d}{2p_2}}&\bigl(\f34\|t^{\gamma_1}u_{\la,n}^h\|_{L^{\infty}_t(L^{p_3})}
+\|t^{\gamma_2}u_{\la,n}^h\|_{L^{2r}_t(L^{p_3})}\bigr)
 + \mu^{\f12+\f{d}{2p_3}}\bigl(\|t^{\beta_1}\na u_{\la,n}^h\|_{L^{2r}_t(L^{p_2})}\\
 &+\f34\|t^{\beta_2}\na u_{\la,n}^h\|_{L^{\infty}_t(L^{p_2})}\bigr)+\mu\|t^{\alpha_1} \D u_{\la,n}^h\|_{L^{2r}_t(L^{p_1})}\\
\leq & C
\mu^{\f12(\f{d}{p_1}-1)}\bigl(\|u_0^h\|_{\dot{B}^{-1+\f{d}p}_{p,r}}+\mu\|a_0\|_{L^\infty}\bigr)+C\|t^{\gamma_1}u_{n}^h\|_{L^{\infty}_t(L^{p_3})}\|t^{\beta_1}\na
u_{\la,n}^h\|_{L^{2r}_t(L^{p_2})}.
\end{split}\eeq
Let  $c_1$ be a small enough positive constant, which will be
determined later on, we denote \beq \label{3.12}
\begin{split}\bar{T}_n\eqdefa \sup\Bigl\{\ T\leq T_n^\ast;\quad
&\mu^{\f12-\f{d}{2p_3}}\|t^{\gamma_1}
u_{n}^h\|_{L^{\infty}_T(L^{p_3})}
+\mu^{1-\f{d}{2p_2}}\|t^{\beta_1}\na
u_{n}^h\|_{L^{2r}_T(L^{p_2})}\\
&+\mu^{\f3{2}(1-\f{d}{p_1})}\|t^\alpha\D
u_{n}^h\|_{L^{2r}_T(L^{p_1})}\leq c_1\mu\ \Bigr\}. \end{split}\eeq
We shall prove that $\bar{T}_n=T_n^\ast=\infty$ as long as we take
$C_r$ sufficiently large and $c_0$ sufficiently small in
\eqref{small1}. In fact, if $\bar{T}_n<T_n^\ast,$ it follows from
\eqref{3.8} and \eqref{3.11} that \beq\label{3.13}
\begin{split} \mu^{\f12(1-\f{d}{p_3})}&\bigl(\|t^{\gamma_1}u_{n}^h\|_{L^{\infty}_t(L^{p_3})} +
\|t^{\gamma_2}u_{n}^h\|_{L^{2r}_t(L^{p_3})}\bigr)+
\mu^{1-\f{d}{2p_2}}\bigl(\|t^{\beta_1}\na
u_{n}^h\|_{L^{2r}_t(L^{p_2})}\\
&+
 \|t^{\beta_2}\na u_{n}^h\|_{L^{\infty}_t(L^{p_2})}\bigr)+\mu^{\f12(3-\f{d}{p_1})}\|t^{\alpha_1} \D u_{n}^h\|_{L^{2r}_t(L^{p_1})}\\
\leq & C
\bigl(\|u_0^h\|_{\dot{B}^{-1+\f{d}p}_{p,r}}+\mu\|a_0\|_{L^\infty}\bigr)\exp\Bigl\{-C_r\int_0^t\bigl(
\mu^{-\f{d}{p_2}r}\|s^{\beta_1}\na
u^d_n(s)\|_{L^{p_2}}^{2r}\\
&\qquad\qquad\qquad+\mu^{-(1+\f{d}{p_3})r}\|s^{\gamma_2}u^d_n(s)\|_{L^{p_3}}^{2r}+\mu^{(1-\f{d}{p_1})r}\|s^\alpha\D
u^d_n(s)\|_{L^{p_1}}^{2r}\bigr)\,ds\Bigr\}.
\end{split}
\eeq

On the other hand, it follows from a similar derivation of
\eqref{3.10} that \beq\label{3.16}
\begin{split}\mu^{\f{d}{2p_2}}&\bigl(\|t^{\gamma_1}u_n^{d}\|_{L^{\infty}_t(L^{p_3})} +\|t^{\gamma_2}u_n^{d}\|_{L^{2r}_t(L^{p_3})}\bigr)
 + \mu^{\f12+\f{d}{2p_3}}\bigl(\|t^{\beta_1}\na u_n^{d}\|_{L^{2r}_t(L^{p_2})}\\
 &+\|t^{\beta_2}\na u_n^{d}\|_{L^{\infty}_t(L^{p_2})}\bigr)+\mu\|t^{\alpha_1} \D u_n^{d}\|_{L^{2r}_t(L^{p_1})}\\
\leq & \mu^{-\f12(1-\f{d}{p_1})}\|
u_{0}^d\|_{\dot{B}^{-1+\f{d}p}_{p,r}}+C\Bigl\{
\mu\|a_0\|_{L^\infty}\bigl(\|t^{\alpha_1}\D
u_{n}^h\|_{L^{2r}_t(L^{p_1})}+\|t^{\alpha_1}\D
u_n^{d}\|_{L^{2r}_t(L^{p_1})}\bigr)\\
&+\bigl(\|t^{\gamma_1}u_{n}^h\|_{L^\infty_t(L^{p_3})}+\|t^{\gamma_1}u_n^{d}\|_{L^\infty_t(L^{p_3})}\bigr)\|t^{\beta_1}\na
u_{n}^h\|_{L^{2r}_t(L^{p_2})}\\
&+\|t^{\gamma_1}u_{n}^h\|_{L^\infty_t(L^{p_3})}\|t^{\beta_1}\na
u_n^{d}\|_{L^{2r}_t(L^{p_2})}\Bigr\}\quad\mbox{for}\quad t<T_n^\ast,
\end{split}\eeq
from which, \eqref{3.12} and taking $c_0$ small enough in
\eqref{small1} so that $C\|a_0\|_{L^\infty}\leq\f12,$ we deduce
\beno
\begin{split} \mu^{\f{d}{2p_2}}&\bigl(\|t^{\gamma_1}u_n^{d}\|_{L^{\infty}_t(L^{p_3})} +\|t^{\gamma_2}u_n^{d}\|_{L^{2r}_t(L^{p_3})}\bigr)
 + \mu^{\f12+\f{d}{2p_3}}\bigl(\|t^{\beta_1}\na u_n^{d}\|_{L^{2r}_t(L^{p_2})}\\
 &+\|t^{\beta_2}\na u_n^{d}\|_{L^{\infty}_t(L^{p_2})}\bigr)+\mu\|t^{\alpha_1} \D u_n^{d}\|_{L^{2r}_t(L^{p_1})}\\
&\leq 2C\mu^{-\f12(1-\f{d}{p_1})}\|
u_{0}^d\|_{\dot{B}^{-1+\f{d}p}_{p,r}}
+2Cc_1\mu^{\f12(1+\f{d}{p_1})}(1+\|a_0\|_{L^\infty})\quad\mbox{for}\quad
t\leq \bar{T}_n.
\end{split}\eeno
which implies \beq\label{3.14}
\begin{split}
 \mu^{\f12(1-\f{d}{p_3})}&\bigl(\|t^{\gamma_1}u_n^{d}\|_{L^{\infty}_t(L^{p_3})} +\|t^{\gamma_2}u_n^{d}\|_{L^{2r}_t(L^{p_3})}\bigr)
 + \mu^{1-\f{d}{2p_2}}\bigl(\|t^{\beta_1}\na u_n^{d}\|_{L^{2r}_t(L^{p_2})}\\
 &+\|t^{\beta_2}\na u_n^{d}\|_{L^{\infty}_t(L^{p_2})}\bigr)+\mu^{\f12(3-\f{d}{p_1})}\|t^{\alpha_1} \D u_n^{d}\|_{L^{2r}_t(L^{p_1})}\\
&\leq 2C\| u_{0}^d\|_{\dot{B}^{-1+\f{d}p}_{p,r}} +c\mu
\quad\mbox{for}\quad t\leq\bar{T}_n.
\end{split}\eeq
Substituting \eqref{3.14} into \eqref{3.13} leads to
\beq\label{3.15}
\begin{split} \mu^{\f12(1-\f{d}{p_3})}&\bigl(\|t^{\gamma_1}u_{n}^h\|_{L^{\infty}_t(L^{p_3})} +
\|t^{\gamma_2}u_{n}^h\|_{L^{2r}_t(L^{p_3})}\bigr)+
\mu^{1-\f{d}{2p_2}}\bigl(\|t^{\beta_1}\na
u_{n}^h\|_{L^{2r}_t(L^{p_2})}\\
&+
 \|t^{\beta_2}\na u_{n}^h\|_{L^{\infty}_t(L^{p_2})}\bigr)+\mu^{\f12(3-\f{d}{p_1})}\|t^{\alpha_1} \D u_{n}^h\|_{L^{2r}_t(L^{p_1})}\\
\leq&
C\bigl(\|u_0^h\|_{\dot{B}^{-1+\f{d}p}_{p,r}}+\mu\|a_0\|_{L^\infty}\bigr)\exp\Bigl\{C_r
\mu^{-2r}\|u_0^d\|_{\dot{B}^{-1+\f
dp}_{p,r}}^{2r}\Bigr\}\leq\f{c_1}2\mu\quad\mbox{for}\quad
t\leq\bar{T}_n,
\end{split}
\eeq as long as $C_r$ is sufficiently large and $c_0$  small enough
in \eqref{small1}. This contradicts with \eqref{3.12} and it in turn
shows that $\bar{T}_n=T_n^\ast.$ Then thanks to \eqref{3.14},
\eqref{3.15} and Theorem 1.3 in \cite{Kim}, we conclude that
$T_n^\ast=\infty$ and there holds \eqref{3.7}. It remains to prove
\eqref{3.17}. Indeed, similar to \eqref{3.16}, we get, by applying
Lemma \ref{lem3}, that \beno
\begin{split}
\mu\|t^{\alpha_2} \D u_{n}\|_{L^{r}(\R^+;L^{p_1})} \leq &
\mu^{-\f12(1-\f{d}{p_1})}\|
u_{0}\|_{\dot{B}^{-1+\f{d}p}_{p,r}}+C\Bigl\{
\mu\|a_0\|_{L^\infty}\|t^{\alpha_2}\D
u_{n}\|_{L^{r}(\R^+;L^{p_1})}\\
&+\bigl(\|t^{\gamma_2}u_{n}^h\|_{L^{2r}(\R^+;L^{p_3})}+\|t^{\gamma_2}u_n^{d}\|_{L^{2r}(\R^+;L^{p_3})}\bigr)\|t^{\beta_1}\na
u_{n}^h\|_{L^{2r}(\R^+; L^{p_2})}\\
&+\|t^{\gamma_2}u_{n}^h\|_{L^{2r}(\R^+; L^{p_3})}\|t^{\beta_1}\na
u_n^{d}\|_{L^{2r}_t(L^{p_2})}\Bigr\},
\end{split}
\eeno which along with \eqref{3.7} and  the fact:
$C\|a_0\|_{L^\infty}\leq \f12,$ gives rise to the first inequality
of \eqref{3.17}. Along the same line, one gets the estimate of
$\|t^{\alpha_2} \na \Pi_{n}\|_{L^{r}(\R^+;L^{p_1})}$ in
\eqref{3.17}.  Whereas it follows from  and \eqref{pest1} that \beno
\begin{split}
\|t^{\alpha_1}& \na \Pi_{n}\|_{L^{2r}(\R^+;L^{p_1})}\leq  C\Bigl\{
\mu\|a_0\|_{L^\infty}\|t^{\alpha_1}\D
u_{n}\|_{L^{2r}(\R^+;L^{p_1})}\\
&+\|t^{\gamma_1}u_{n}^h\|_{L^{\infty}(\R^+;L^{p_3})}\|t^{\beta_1}\na
u_{n}\|_{L^{2r}(\R^+;
L^{p_2})}+\|t^{\gamma_2}u_n^{d}\|_{L^{2r}(\R^+;L^{p_3})}\|t^{\beta_2}\na
u_{n}^h\|_{L^{\infty}(\R^+; L^{p_2})}\Bigr\},
\end{split}
\eeno from which and \eqref{3.7}, we obtain the estimate of
$\|t^{\alpha_1} \na \Pi_{n}\|_{L^{2r}(\R^+;L^{p_1})}$ in
\eqref{3.17}.
 This
completes the proof of the proposition.
\end{proof}

Now we are in a position to complete the proof  to the existence
part of Theorem
\ref{thm1.1} for the remaining case.\\

\no{\bf Proof to the existence part of Theorem \ref{thm1.1} for
$p\in (\f{dr}{3r-2}, d).$}\ Notice that $p_1<d$ ensures $\al_2
r'<1,$ so that for any $T>0,$ we deduce from \eqref{3.17} that $\D
u_n=t^{-\al_2}(t^{\al_2}\D u_n)$ is uniformly bounded in
$L^{\tau_1}((0,T); L^{p_1}(\R^d))$ for some $\tau_1\in (1,\infty).$
Similarly we infer from  \eqref{indexad} and \eqref{3.7} that $\{\na
u_n\}$ is uniformly bounded in $L^{\tau_2}((0,T); L^{p_2}(\R^d))$
for any $\tau_2<\f{2p_2}{2p_2-d},$ and $\{u_n\}$ is uniformly
bounded in $L^{\tau_3}((0,T); L^{p_3}(\R^d))$ for any
$\tau_3<\f{2p_3}{p_3-d}.$ Moreover as
$\f{2p_2-d}{2p_2}+\f{p_3-d}{2p_3}=\f32-\f{d}{2p_1}$ and $p_1<d,$ we
can choose $\tau_2$ and $\tau_3$ so that
$\f1{\tau_2}+\f1{\tau_3}=\f1{\tau_4}<1.$ This together with
\eqref{schema} implies that $\{\p_tu_n\}$ is uniformly bounded in
$L^{\tau_1}((0,T); L^{p_1}(\R^d))+L^{\tau_4}((0,T); L^{p_1}(\R^d))$
for any $T<\infty,$
 from which, Ascoli-Arzela
Theorem and $p_2<\f{dp_1}{d-p_1},$  we conclude that there exists a
subsequence of $\{a_n,u_n,\na\Pi_n\},$ which we still denote by
$\{a_n,u_n,\na\Pi_n\}$ and some $(a,u,\na\Pi)$ with $a\in
L^\infty(\R^+\times\R^d),$ $u\in L^{\tau_3}_{loc}(\R^+;
L^{p_3}(\R^d))$ with $\na u\in L^{\tau_2}_{loc}(\R^+;
L^{p_2}(\R^d))$ and $\D u,\na\Pi\in L^{\tau_1}_{loc}(\R^+;
L^{p_1}(\R^d)),$ such that \beq \label{3.25}
\begin{split}
&a_n \rightharpoonup  a\quad \mbox{weak}\ \ast \ \mbox{in}\
L^\infty_{loc}(\R^+\times\R^d),\\
&  u_n \rightarrow u \quad \mbox{strongly}\ \ \mbox{in}\
L^{\tau_3}_{loc}(\R^+; L^{p_3}_{loc}(\R^d)),\\
&  \na u_n \rightarrow \na u \quad \mbox{strongly}\  \mbox{in}\ \
L^{\tau_2}_{loc}(\R^+; L^{p_2}_{loc}(\R^d)),\\
& \D u_n \rightharpoonup \D u\quad\mbox{and}\quad \na\Pi_n
\rightharpoonup \na\Pi \quad \mbox{weakly}\ \ \mbox{in}\
L^{\tau_1}_{loc}(\R^+; L^{p_1}(\R^d)).
\end{split}
\eeq With \eqref{3.7}, \eqref{3.17} and \eqref{3.25}, we can repeat
the argument at the end of Section \ref{sect2} to complete the proof
to the existence part of Theorem \ref{thm1.1} for the case when
$\f{dr}{3r-2}<p<d.$ \ef

\setcounter{equation}{0}
\section{The uniqueness part of Theorem \ref{thm1.1}}


With a little bit more regularity on the initial velocity, namely,
$u_0\in \dot{B}^{-1+\f{d}{p}+\ep}_{p,r}(\R^d)$ for some small enough
$\ep>0$, we can prove the uniqueness of the solution constructed in
the last two sections. This result is strongly inspired by the
Lagrangian approach in \cite{dm, dm2}. Nevertheless, with an almost
critical regularity for the initial velocity field, the proof here
is more challenging. The main result is listed as follows:

\begin{thm}\label{thm1.1bis}
{\sl Let $r\in(1,\infty),$ $p\in (1,d)$ and $0<\ep<\min\{ \f1r,
1-\f1r, \f{d}p-1\}$. Let $a_0\in L^\infty(\R^d)$ and $ u_0\in
\dot{B}^{-1+\f{d}p}_{p,r}(\R^d)\cap
\dot{B}^{-1+\f{d}p+\ep}_{p,r}(\R^d),$ which satisfies the nonlinear
smallness condition \eqref{small1}. Then \eqref{INS} has a unique
global weak solution $(a,u)$ which satisfies
(\ref{thm1a}-\ref{thm1c}) and
 \beq\label{thm1cg}
\begin{split} \mu^{\f12(3-\f{d}{q_1}-\ep)}&\|t^{\alpha_1^{\ep}} \D u\|_{L^{2r}(\R^+; L^{q_1})}
+ \mu^{1-\f{d}{2q_2}-\f\ep2}\bigl(\|t^{\beta_1^{\ep}}\na
u\|_{L^{2r}(\R^+; L^{q_2})}+
 \|t^{\beta_2^{\ep}}\na u\|_{L^{\infty}(\R^+; L^{q_2})}\bigr)\\
&+\mu^{\f12(1-\f{d}{q_3}-\ep)}\bigl(\|t^{\gamma_1^{\ep}}u\|_{L^{\infty}(\R^+;L^{q_3})}
+
\|t^{\gamma_2^{\ep}}u\|_{L^{2r}(\R^+;L^{q_3})}\bigr)\\
 &\qquad\leq C\|u_0\|_{\dot{B}^{-1+\f{d}p+\ep}_{p,r}}\exp\Bigl(C_r
\mu^{-2r}\|u_0^d\|_{\dot{B}^{-1+\f{d}p}_{p,r}}^{2r}\Bigr),
\end{split}
\eeq and  \beq\label{thm1d}
\begin{split}
&\mu^{\f12(3-\f{d}{q_1}-\ep)}\|t^{\alpha_2^\ep} \D u\|_{L^{r}(\R^+;
L^{q_1})}\leq C\|u_0\|_{\dot{B}^{-1+\f{d}p+\ep}_{p,r}}\exp\Bigl(C_r
\mu^{-2r}\|u_0^d\|_{\dot{B}^{-1+\f{d}p}_{p,r}}^{2r}\Bigr),\\
&\mu^{\f12(3-\f{d}{q_1}-\ep)}\bigl(\|t^{\alpha_1^\ep}
\na\Pi\|_{L^{2r}(\R^+; L^{q_1})}+\|t^{\alpha_2^\ep}
\na\Pi\|_{L^{r}(\R^+; L^{q_1})}\bigr)\\
&\qquad\qquad\qquad\qquad\qquad\qquad\leq
C\eta\|u_0\|_{\dot{B}^{-1+\f{d}p+\ep}_{p,r}}\exp\Bigl(C_r
\mu^{-2r}\|u_0^d\|_{\dot{B}^{-1+\f{d}p}_{p,r}}^{2r}\Bigr),
\end{split} \eeq where $q_1, q_2, q_3$ satisfy $\max (p,
\f{dr}{2r-1})<q_1<\f{d}{1-\ep}$,
 and $\f{dr}{r-1}<q_3<\infty$ so that $\f1{q_2}+\f1{q_3}=\f1{q_1},$ the indices
  $\alpha_1^{\ep}$, $\alpha_2^{\ep},$ $\beta_1^{\ep}$,
$\beta_2^{\ep}$, $\gamma_1^{\ep}$, $\gamma_2^{\ep}$ are determined
by
  \beq\label{indexad1}\begin{split}
 & \alpha_1^{\ep}=\f12(3-\f{d}{q_1}-\ep)-\f1{2r},\quad
\beta_1^{\ep}=\f12(2-\f{d}{q_2}-\ep)-\f1{2r},\quad
\gamma_1^{\ep}=\f12(1-\f{d}{q_3}-\ep),\\
& \alpha_2^{\ep}=\f12(3-\f{d}{q_1}-\ep)-\f1{r},\quad \
\beta_2^{\ep}=\f12(2-\f{d}{q_2}-\ep),\qquad\qquad
\gamma_2^{\ep}=\f12(1-\f{d}{q_3}-\ep)-\f1{2r}. \end{split} \eeq  }
\end{thm}

\begin{rmk}\label{rmk3.1}
To prove the uniqueness part of Theorem \ref{thm1.1bis}, we shall
choose $p_1=\f{d}{1+\ep}$ in \eqref{indexad}. This choice of $p_1$
satisfies $\max (p, \f{dr}{2r-1})<p_1<d$ with $0<\e <\min (1-\f1r,
\f{d}p-1)$.  Then by virtue of \eqref{thm1b}, $\D
u=t^{-\al_2}(t^{\al_2}\D u)\in L^{\tau_1}((0,T); L^{p_1}(\R^d))$ for
any  $T<\infty$ and $\tau_1$ satisfying $\al_2 < \f1{\tau_1}-\f1r$,
which implies $\tau_1<\f2{2-\e}$, we  thus take $\tau_1=\f8{8-\ep}$.
While as $q_1<\f{d}{1-\ep}$ in Theorem \ref{thm1.1bis}, we can
choose $q_1>d$ in order to get $\na u \in L^1_{loc}(L^\infty)$ (see
Lemma \ref{lem4.1} below), and this is the reason we need an
additional regularity on $u_0$ for the uniqueness part of Theorem
\ref{thm1.1}. We shall choose $q_1=\f{2d}{2-\ep}<\f{d}{1-\ep}$ in
\eqref{indexad1} later on. Then in order that $\D u\in
L^{\tau_1^\e}((0,T); L^{p_1}(\R^d)),$ we have $\al_2^{\e} <
\f1{\tau_1^{\ep}}-\f1r$, and hence we  take
$\tau_1^{\e}=\f8{8-\e}<\f4{4-\e}.$
\end{rmk}

It follows from the existence proof of Theorem \ref{thm1.1} that: in
order to prove the solution constructed in the last two sections
satisfies  \eqref{thm1cg} and \eqref{thm1d}, we only need to prove
that the same inequalities hold for the approximate solutions $
(a_n, u_n, \na\Pi_n)$ determined by \eqref{schema}.

We now turn to the uniform estimate of $(u_n, \na\Pi_n)_{n\in\N}$
when the initial velocity $u_0\in
\dot{B}^{-3+\f{d}{p}+\ep}_{p,r}(\R^d)$ for some $\ep\in (0,
\min\{\f1r,1-\f1r,\f{d}p-1\}).$ Notice that for all $q_1\geq p$ and
$r_1\geq r$, $\D u_0\in
\dot{B}^{-3+\f{d}{q_1}+\ep}_{q_1,r_1}(\R^d),$ then we deduce from
Proposition \ref{prop1.1} that
$t^{\f12(3-\f{d}{q_1}-\ep)-\f1{r_1}}\tri e^{t\tri}u_0\in
L^{r_1}(\R^+; L^{q_1}(\R^d)).$ Similarly, we choose
$q_2>\f{d}{2-\ep}$ and $q_3>\f{d}{1-\ep}$ with
$\f1{q_2}+\f1{q_3}=\f1{q_1}$, there holds
$t^{\f12(2-\f{d}{q_2}-\ep)-\f1{r_1}}\na e^{t\tri}u_0\in
L^{r_1}(\R^+; L^{q_2}(\R^d))$ and $t^{\f12(1-\f{d}{q_3}-\ep)}
e^{t\tri}u_0\in L^\infty(\R^+; L^{q_3}(\R^d)).$ We thus take
$r_1=2r>2$ and $q_1, q_2$ and $q_3$ satisfying $\max (p,
\f{dr}{2r-1})<q_1<\f{d}{1-\ep}$,
 $\f{dr}{r-1}<q_3<\infty,$ and $\f1{q_2}+\f1{q_3}=\f1{q_1}.$
 We shall investigate the uniform estimate to the solutions $(a_n,u_n, \na\Pi_n)$
of \eqref{schema} in the following functional space: \beq
\label{3.6}\begin{split} X=\Bigl\{\ u: \, \,&a\in
L^\infty(\R^+\times\R^d),\quad t^{\gamma_1^{\ep}} u\in
L^\infty(\R^+;
L^{q_3}(\R^d)), \\
&t^{\gamma_2^{\ep}} u\in L^{2r}(\R^+; L^{q_3}(\R^d)),\quad
t^{\beta_1^{\ep}}\na
u \in L^{2r}(\R^+;L^{q_2}(\R^d)),\, \, \\
 &t^{\beta_2^{\ep}}\na u \in
L^{\infty}(\R^+;L^{q_2}(\R^d))£¬\quad t^{\alpha_1^{\ep}}\tri u \in
L^{2r}(\R^+; L^{q_1}(\R^d))\,\Bigr\},\end{split}\eeq with the
indices $\alpha_1^{\ep}$, $\beta_1^{\ep}$, $\beta_2^{\ep}$,
$\gamma_1^{\ep}$, $\gamma_2^{\ep}$ being determined by
\eqref{indexad1}. Note that $\al_1^\ep \neq
\beta_i^\ep+\gamma_i^\ep$, but $\al_1^\ep=
\bar{\beta}_1+\gamma_1^\ep=\beta_2^\ep+\bar{\gamma}_2$
 with $\bar{\beta}_1= \f12(1-\f{d}{q_3})-\f1{2r}$ and $\bar{\gamma}_2=\f12(2-\f{d}{q_2})-\f1{2r}$. Then
 we can apply Proposition \ref{uniformbound1} to prove the following proposition concerning
the uniform time-weighted bounds of $(a_n,u_n, \na\Pi_n)$.

\begin{prop}\label{uniformbound2}
{\sl Under the assumptions of Theorem \ref{thm1.1bis},
\eqref{schema} has a unique global smooth solution
$(a_n,u_n,\na\Pi_n)$ which satisfies
 \beq\label{3.1}
\begin{split} \mu^{\f12(3-\f{d}{q_1}-\ep)}&\|t^{\alpha_1^{\ep}} \D u_{n}\|_{L^{2r}(\R^+; L^{q_1})}
+ \mu^{1-\f{d}{2q_2}-\f\ep2}\bigl(\|t^{\beta_1^{\ep}}\na
u_{n}\|_{L^{2r}(\R^+; L^{q_2})}+
 \|t^{\beta_2^{\ep}}\na u_{n}\|_{L^{\infty}(\R^+; L^{q_2})}\bigr)\\
&+
\mu^{\f12(1-\f{d}{q_3}-\ep)}\bigl(\|t^{\gamma_1^{\ep}}u_{n}\|_{L^{\infty}(\R^+;L^{q_3})}
+
\|t^{\gamma_2^{\ep}}u_{n}\|_{L^{2r}(\R^+;L^{q_3})}\bigr)\\
 &\qquad\leq C\|u_0\|_{\dot{B}^{-1+\f{d}p+\ep}_{p,r}}\exp\Bigl(C_r
\mu^{-2r}\|u_0^d\|_{\dot{B}^{-1+\f{d}p}_{p,r}}^{2r}\Bigr),
\end{split}\eeq  and  \beq\label{3.1adf} \begin{split}
&\mu^{\f12(3-\f{d}{q_1}-\ep)}\|t^{\alpha_2^\ep} \D
u_n\|_{L^{r}(\R^+; L^{q_1})}\leq
C\|u_0\|_{\dot{B}^{-1+\f{d}p+\ep}_{p,r}}\exp\Bigl(C_r
\mu^{-2r}\|u_0^d\|_{\dot{B}^{-1+\f{d}p}_{p,r}}^{2r}\Bigr),\\
&\mu^{\f12(3-\f{d}{q_1}-\ep)}\bigl(\|t^{\alpha_1^\ep}
\na\Pi_n\|_{L^{2r}(\R^+; L^{q_1})}+\|t^{\alpha_2^\ep}
\na\Pi_n\|_{L^{r}(\R^+; L^{q_1})}\bigr)\\
&\qquad\qquad\qquad\qquad\qquad\qquad\leq
C\eta\|u_0\|_{\dot{B}^{-1+\f{d}p+\ep}_{p,r}}\exp\Bigl(C_r
\mu^{-2r}\|u_0^d\|_{\dot{B}^{-1+\f{d}p}_{p,r}}^{2r}\Bigr),
\end{split} \eeq for
some constant $C$ and  $\al_i^\ep, \beta_i^\ep, \gamma_i^\ep,$
$i=1,2$, given by \eqref{indexad1}. }
\end{prop}

\begin{proof} For $N$ large enough, we already proved in Proposition \ref{uniformbound1}  that
\eqref{schema} has a unique global smooth solution $(a_n, u_n,
\na\Pi_n).$ It remains to prove \eqref{3.1} and \eqref{3.1adf}. In
order to do so, for $\la_1, \la_2>0,$ we denote \beq
\label{3.2}\begin{split} &g_{1,n}(t)\eqdefa \|t^{\bar{\beta}_1}\na
u^d_n(t)\|_{L^{q_2}}^{2r} ,\quad
g_{2,n}(t)\eqdefa \|t^{\bar{\gamma}_2}u^d_{n}(t)\|_{L^{q_3}}^{2r},\\
&u_{\la,n}(t,x)\eqdefa u_n(t,x)\exp\Bigl\{-
\int_0^t\bigl(\la_1g_{1,n}(t')+\la_2g_{2,n}(t')\bigr)\,dt'\Bigr\},
\end{split} \eeq and similar notation for
$\Pi_{\la,n}(t,x).$ Then it follows from a similar derivation of
\eqref{pest} that \beq\label{pest2}
\begin{split}
\|\na \Pi_{\la,n}(t)\|_{L^{q_1}}\leq &
C\Bigl\{\mu\|a_0\|_{L^\infty}\|\tri
u_{\la,n}\|_{L^{q_1}}+\|u_{\la,n}^h\|_{L^{q_3}}\|\na
u_{n}^h\|_{L^{q_2}}\\
&+\|u_n^{d}\|_{L^{q_3}}\|\na
u_{\la,n}^h\|_{L^{q_2}}+\|u_{\la,n}^h\|_{L^{q_3}}\|\na
u_n^{d}\|_{L^{q_2}}\Bigr\}.
\end{split}
\eeq
 Whereas by virtue of
\eqref{pesta}  and \eqref{3.2}, we write \beq \label{3.3}
\begin{split}
u_{\la,n}=& u_{n,L}\exp\Bigl\{-
\int_0^t(\la_1g_{1,n}(t')+\la_2g_{2,n}(t'))\,dt'\Bigr\}\\
&+ \int_0^t e^{\mu(t-s)\Delta}\exp\Bigl\{-
\int_s^t(\la_1g_{1,n}(t')+\la_2g_{2,n}(t'))\,dt'\Bigr\}\\
&\qquad\times\bigl(-u_{n}\cdot\nabla u_{\la,n}+\mu a_{n}\Delta
u_{\la,n}-(1+a_{n})\na\Pi_{\la,n}\bigr)\,ds.
\end{split}\eeq
For $\gamma_1^\e, \beta_1^\e$ given by \eqref{indexad1}, we get, by
applying Lemma \ref{lem4} for $r_1=2r$, that \beno
\begin{split} \mu^{\f{d}{2q_2}}&\|t^{\gamma_1^\ep}u_{\la,n}\|_{L^{\infty}_t(L^{q_3})}
 + \mu^{\f12+\f{d}{2q_3}}\|t^{\beta_1^\ep}\na
 u_{\la,n}\|_{L^{2r}_t(L^{q_2})}\\
 \leq & \mu^{\f{d}{2q_2}}\|t^{\gamma_1^\ep}u_{n,L}\|_{L^{\infty}_t(L^{q_3})}
 + \mu^{\f12+\f{d}{2q_3}}\|t^{\beta_1^\ep}\na
 u_{n,L}\|_{L^{2r}_t(L^{q_2})}\\
 &+\bigl\|\exp\bigl\{-
\int_s^t(\la_1g_{1,n}(t')+\la_2g_{2,n}(t'))\,dt'\bigr\}\\
&\qquad\times s^{\alpha_1^\ep}\bigl(-u_{n}\cdot\nabla u_{\la,n}+\mu
a_{n}\Delta u_{\la,n}-(1+a_{n})\na
\Pi_{\la,n}\bigr)\bigr\|_{L^{2r}_t(L^{q_1})}.
 \end{split}
 \eeno
Similarly for $\gamma_2^\e,\beta_2^\e$ and $\al_1^\e,$ we deduce
from Lemma \ref{lem1} and Lemma \ref{lem5} for $r_1=2r$ that \beno
\begin{split}
 \mu^{\f{d}{2q_2}}&\|t^{\gamma_2^\ep}u_{\la,n}\|_{L^{2r}_t(L^{q_3})}
 +\mu^{\f12+\f{d}{2q_3}}\|t^{\beta_2^\ep}\na u_{\la,n}\|_{L^{\infty}_t(L^{q_2})}+\mu\|t^{\alpha_1^\ep} \D u_{\la,n}\|_{L^{2r}_t(L^{q_1})}\\
\leq
&\mu^{\f{d}{2q_2}}\|t^{\gamma_2^\ep}u_{n,L}\|_{L^{2r}_t(L^{q_3})}
 +\mu^{\f12+\f{d}{2q_3}}\|t^{\beta_2^\ep}\na u_{n,L}\|_{L^{\infty}_t(L^{q_2})}+\mu\|t^{\alpha_1^\ep} \D u_{n,L}\|_{L^{2r}_t(L^{q_1})}\\
 &+\bigl\|\exp\bigl\{-
\int_s^t(\la_1g_{1,n}(t')+\la_2g_{2,n}(t'))\,dt'\bigr\}\\
&\qquad\times s^{\alpha_1^\ep}\bigl(-u_{n}\cdot\nabla u_{\la,n}+\mu
a_{n}\Delta u_{\la,n}-(1+a_{n})\na
\Pi_{\la,n}\bigr)\bigr\|_{L^{2r}_t(L^{q_1})}.
 \end{split}
 \eeno
Hence, by using \eqref{2.2aq},  Proposition \ref{prop1.1} and
\eqref{pest2}, we obtain
 \beq\label{3.4}
\begin{split} \mu^{\f{d}{2q_2}}&\bigl(\|t^{\gamma_1^\ep}u_{\la,n}\|_{L^{\infty}_t(L^{q_3})} +\|t^{\gamma_2^\ep}u_{\la,n}\|_{L^{2r}_t(L^{q_3})}\bigr)
 + \mu^{\f12+\f{d}{2q_3}}\bigl(\|t^{\beta_1^\ep}\na u_{\la,n}\|_{L^{2r}_t(L^{q_2})}\\
 &+\|t^{\beta_2^\ep}\na u_{\la,n}\|_{L^{\infty}_t(L^{q_2})}\bigr)
 +\mu\|t^{\alpha_1^\ep} \D u_{\la,n}\|_{L^{2r}_t(L^{q_1})}\\
\leq &C\Bigl\{
\mu^{-\f12(1-\f{d}{q_1}-\ep)}\|u_0\|_{\dot{B}^{-1+\f{d}{q_1}+\ep}_{q_1,2r}}+
\mu\|a_0\|_{L^\infty}\|t^{\alpha_1^\ep} \D u_{\la,n}\|_{L^{2r}_t(L^{q_1})}\\
&+\|t^{\gamma_1^\ep}u_{\la,n}^h\|_{L^{\infty}_t(L^{q_3})}\|t^{\bar{\beta}_1}\na
u_{n}^h\|_{L^{2r}_t(L^{q_2})}\\
&+\Bigl(\int_{0}^te^{-
2r\int_s^t(\la_1g_{1,n}(t')+\la_2g_{2,n}(t'))\,dt'}\bigl(\|s^{\bar{\gamma}_2}u_{n}^d(s)\|_{L^{q_3}}^{2r}\|s^{\beta_2^\ep}\na
u_{\la,n}^h(s)\|_{L^{q_2}}^{2r}\\
&\qquad+\|s^{\gamma_1^\ep}u_{\la,n}^h(s)\|_{L^{q_3}}^{2r}\|s^{\bar{\beta}_1}\na
u_{n}^d(s)\|_{L^{q_2}}^{2r}\bigr)\,ds\Bigr)^{\f1{2r}}\Bigr\}.
\end{split}\eeq
Then as $C\|a_0\|_{L^\infty}\leq \f12,$ and \beno
\begin{split}
&\Bigl(\int_{0}^te^{-2\la_1
r\int_s^tg_{1,n}(t')\,dt'}\|s^{\gamma_1^\ep}u_{\la,n}^h\|_{L^{q_3}}^{2r}\|t^{\bar{\beta}_1}\na
u_{n}^d(s)\|_{L^{q_2}}^{2r} \,ds\Bigr)^{\f1{2r}}\leq \f1{(2\la_1
r)^{\f1{2r}}}\|s^{\gamma_1^\ep} u_{\la,n}^h\|_{L^{\infty}_t(L^{q_3})},\\
&\Bigl(\int_{0}^te^{-2\la_2 r\int_s^tg_{2,n}(t')\,dt'}
 \|s^{\bar{\gamma}_2}u_{n}^d(s)\|_{L^{q_3}}^{2r}\|t^{\beta_2^\ep}\na
u_{\la,n}^h(s)\|_{L^{q_2}}^{2r}\,ds\Bigr)^{\f1{2r}}\leq \f1{(2\la_2
r)^{\f1{2r}}} \|s^{\beta_2^\ep}\na
u_{\la,n}^h\|_{L^{\infty}_t(L^{q_2})},
\end{split}
\eeno  we infer from \eqref{3.4} that \beq\label{aplq} \begin{split}
\mu^{\f{d}{2q_2}}&\bigl(\|t^{\gamma_1^\ep}u_{\la,n}\|_{L^{\infty}_t(L^{q_3})}
+\|t^{\gamma_2^\ep}u_{\la,n}\|_{L^{2r}_t(L^{q_3})}\bigr)
+ \mu^{\f12+\f{d}{2q_3}}\bigl(\|t^{\beta_2^\ep}\na u_{\la,n}\|_{L^{\infty}_t(L^{q_2})} \\
 &+\|t^{\beta_1^\ep}\na u_{\la,n}\|_{L^{2r}_t(L^{q_2})}\bigr)
 +\mu\|t^{\alpha_1^\ep} \D u_{\la,n}\|_{L^{2r}_t(L^{q_1})}\\
\leq &C\Bigl\{
\mu^{-\f12(1-\f{d}{q_1}-\ep)}\|u_0\|_{\dot{B}^{-1+\f{d}{q_1}+\ep}_{q_1,r}}
+\|t^{\gamma_1^\ep}u_{\la,n}^h\|_{L^{\infty}_t(L^{q_3})}\|t^{\bar{\beta}_1}\na
u_{n}^h\|_{L^{2r}_t(L^{q_2})}\\
&+\f1{(2\la_1 r)^{\f1{2r}}}\|t^{\gamma_1^\ep}
u_{\la,n}^h\|_{L^{\infty}_t(L^{q_3})}   +\f1{(2\la_2 r)^{\f1{2r}}}
\|t^{\beta_2^\ep}\na u_{\la,n}^h\|_{L^{\infty}_t(L^{q_2})}\Bigr\}.
\end{split}\eeq
Recalling from \eqref{3.7} that
$$ \|t^{\bar{\beta}_1}\na
u_{n}^h\|_{L^{2r}_t(L^{q_2})}\leq Cc_0\mu^{\f{d}{2q_2}},$$ so that
as long as $c_0$ is small enough in \eqref{small1}, taking
$\la_1=\f{(4C)^{2r}}{2r\mu^{\f{d}{q_2}r}}$,
$\la_2=\f{(4C)^{2r}}{2r\mu^{(1+\f{d}{q_3})r}}$ in \eqref{aplq}
results in \beq\label{3.5} \begin{split}
\mu^{\f{d}{2q_2}}&\bigl(\f12\|t^{\gamma_1^\ep}u_{\la,n}^h\|_{L^{\infty}_t(L^{q_3})}
+\|t^{\gamma_2^\ep}u_{\la,n}^h\|_{L^{2r}_t(L^{q_3})}\bigr)
 + \mu^{\f12+\f{d}{2q_3}}\bigl(\|t^{\beta_1^\ep}\na u_{\la,n}^h\|_{L^{2r}_t(L^{q_2})}\\
 &+\f34\|t^{\beta_2^\ep}\na u_{\la,n}^h\|_{L^{\infty}_t(L^{q_2})}\bigr)+\mu\|t^{\alpha_1^\ep} \D u_{\la,n}^h\|_{L^{2r}_t(L^{q_1})}\\
\leq & C
\mu^{\f12(\f{d}{q_1}-1-\ep)}\|u_0\|_{\dot{B}^{-1+\f{d}p+\ep}_{p,r}}.
\end{split}\eeq
From \eqref{3.7}, \eqref{3.2},  and \eqref{3.5}, we infer \beno
\begin{split}
\mu^{\f{d}{2q_2}}&\bigl(\|t^{\gamma_1^\ep}u_{n}\|_{L^{\infty}_t(L^{q_3})}
+\|t^{\gamma_2^\ep}u_{n}\|_{L^{2r}_t(L^{q_3})}\bigr)
 + \mu^{\f12+\f{d}{2q_3}}\bigl(\|t^{\beta_1^\ep}\na u_{n}\|_{L^{2r}_t(L^{q_2})}\\
 &+\|t^{\beta_2^\ep}\na u_{n}\|_{L^{\infty}_t(L^{q_2})}\bigr)+\mu\|t^{\alpha_1^\ep} \D u_{n}\|_{L^{2r}_t(L^{q_1})}\\
\leq & C
\mu^{\f12(\f{d}{q_1}-1-\ep)}\|u_0\|_{\dot{B}^{-1+\f{d}p+\ep}_{p,r}}\exp\Bigl\{C_r\int_0^t
\mu^{-\f{d}{q_2}r} \|s^{\bar{\beta}_1}\na
u^d_n(s)\|_{L^{q_2}}^{2r}\\
&\qquad\qquad\qquad\qquad\qquad\qquad\qquad\qquad+\mu^{-(1+\f{d}{q_3})r}\|s^{\bar{\gamma}_2}u^d_{n}(s)\|_{L^{q_3}}^{2r}\,ds\Bigr\}\\
\leq &
C\mu^{\f12(\f{d}{q_1}-1-\ep)}\|u_0\|_{\dot{B}^{-1+\f{d}p+\ep}_{p,r}}\exp\Bigl\{C_r
\mu^{-2r}\|u_0^d\|_{\dot{B}^{-1+\f{d}p}_{p,r}}^{2r}\Bigr\},
\end{split}\eeno which implies \eqref{3.1}. It remains to prove
\eqref{3.1adf}. In fact, we get, by applying Lemma \ref{lem1} and
\eqref{pest2},  that \beno
\begin{split}
\mu\|t^{\alpha_2^\ep} \D u_{n}\|_{L^{r}(\R^+;L^{q_1})}
 \leq &
\mu^{-\f12(1-\f{d}{q_1}-\ep)}\|
u_{0}\|_{\dot{B}^{-1+\f{d}p+\ep}_{p,r}}+C\Bigl\{
\mu\|a_0\|_{L^\infty}\|t^{\alpha_2^\ep}\D
u_{n}\|_{L^{r}(\R^+;L^{q_1})}\\
&+\|t^{\gamma_2^\ep}u_{n}\|_{L^{2r}(\R^+;L^{q_3})}\|t^{\bar{\beta}_1}\na
u_{n}^h\|_{L^{2r}(\R^+;
L^{q_2})}\\
&+\|t^{\bar{\gamma}_2}u_{n}^h\|_{L^{2r}(\R^+;
L^{q_3})}\|t^{\beta_1^\ep}\na u_n^{d}\|_{L^{2r}(\R^+;
L^{q_2})}\Bigr\},
\end{split}
\eeno which along with \eqref{3.7}, \eqref{3.1} and  the fact:
$C\|a_0\|_{L^\infty}\leq \f12,$ gives rise to the first inequality
of \eqref{3.1adf}. The second inequality of \eqref{3.1adf} follows
along the same line. This completes the proof of Proposition
\ref{uniformbound1}.
\end{proof}

To prove the uniqueness part of Theorem \ref{thm1.1bis}, we need the
following lemma:

\begin{prop}\label{prop4.1}
{\sl Let  $\al_1, \beta_1,\beta_2,\gamma_1,\gamma_2,$ and
$p_1,p_2,p_3$ be given by Theorem \ref{thm1.1}, if $ t^{\al_1} f,
t^{\al_1} \na g,  t^{\al_1} R \in L^{2r}((0,T);L^{p_1}(\R^d)).$ Then
the system
\begin{equation}\label{Stokes}
 \quad\left\{\begin{array}{l}
\displaystyle \pa_t v -\tri v+ \grad P=f, \\
\displaystyle \dv\, v = g, \\
\displaystyle \pa_t g=\dive R,\\
\displaystyle v|_{t=0}=0,
\end{array}\right.
\end{equation}
has a unique solution $(v, \na P)$ so that \beq\label{4.2}
\begin{split}
\|t^{\gamma_1}v&\|_{L^\infty_T(L^{p_3})}+\|t^{\gamma_2}v\|_{L^{2r}_T(L^{p_3})}+\|t^{\beta_1}\na
v\|_{L^{2r}_T(L^{p_2})}+\|t^{\beta_2}\na v\|_{L^{\infty}_T(L^{p_2})}\\
&+ \|(t^{\al_1}\pa_t v, t^{\al_1}\na^2 v, t^{\al_1}\na
P)\|_{L^{2r}_T(L^{p_1})}\leq C\|(t^{\al_1} f, t^{\al_1}\na g,
t^{\al_1} R)\|_{L^{2r}_T(L^{p_1})}.\end{split} \eeq}
\end{prop}

\begin{proof}
We first get, by taking space divergence on \eqref{Stokes},  that
\beno \D P=\dive(f+\na g-R), \eeno which implies \beno \|\na
P\|_{L^q(\R^d)}\leq C\|(f,\na g, R)\|_{L^q(\R^d)}, \eeno for any
$q\in (1,\infty).$ On the other hand, we have \beq\label{4kernel}
v=\int_0^te^{(t-s)\D}(f-\na P)\,ds, \eeq from which, \eqref{Stokes},
Lemma \ref{lem3}, and Lemma \ref{lem4} and Lemma \ref{lem5} for
$\e=0$ and $r_1=2r,$ we deduce  \eqref{4.2}.
\end{proof}

\begin{lem}\label{lem4.1}
{\sl Let $(a,u)$ be a global weak solution of \eqref{INS}, which
satisfies (\ref{thm1b}-\ref{thm1c}) and (\ref{thm1cg}-\ref{thm1d}).
Then for any $\ep\in (0,\min\{1-\f1r,\f{d}p-1\}),$ one has
\beq\label{lem4.1q}
\begin{split}
&\|\na^2u\|_{L^{\f8{8-\ep}}_T(L^d)}+\|\na
u\|_{L^{\f8{8-\ep}}_T(L^\infty)}\leq C_{T,\ep},\\
&\|t^{\al_1}\na
u\|_{L^{2r}_T(L^\infty)}+\|t^{\al_1}\na^2u\|_{L^{2r}_T(L^d)}\leq
C_{T,\ep}\quad \mbox{for any}\quad T<\infty. \end{split} \eeq }
\end{lem}

\begin{proof} We first deduce from  Remark
\ref{rmk3.1} that $\na^2 u \in L^{\f8{8-\ep}}((0,T);
L^{\f{d}{1+\ep}}(\R^d))\cap L^{\f8{8-\ep}}((0,T);
L^{\f{2d}{2-\ep}}(\R^d))$ for any $T<\infty.$
 Then applying H\"older's inequality yields
  \beq\label{4.3}
  \|\na^2 u \|_{L^{\f8{8-\ep}}_T(L^d)} \leq
\|\na^2 u\|_{L^{\f8{8-\ep}}_T(L^{\f{d}{1+\ep}})}^{\f13} \|\na^2
u\|_{L^{\f8{8-\ep}}_T(L^{\f{2d}{2-\ep}})}^{\f23}\leq C_{T,\ep}.\eeq
Whereas by virtue of  Lemma \ref{lem2.1}, one has \beno
\begin{split} \|\na u\|_{L^{\f8{8-\ep}}_T(L^\infty)}&\lesssim \sum_{j\leq
0}\|\na \dot{\D}_ju\|_{L^{\f8{8-\ep}}_T(L^\infty)}+\sum_{j> 0}\|\na
\dot{\D}_ju\|_{L^{\f8{8-\ep}}_T(L^\infty)}\\
&\lesssim \sum_{j\leq 0}2^{j\ep}\|\na^2
\dot{\D}_ju\|_{L^{\f8{8-\ep}}_T(L^{\f{d}{1+\ep}})}
+\sum_{j>0}2^{-j\f{\ep}2}\|\na^2
\dot{\D}_ju\|_{L^{\f8{8-\ep}}_T(L^{\f{2d}{2-\ep}})}\\
&\lesssim \|\na^2 u\|_{L^{\f8{8-\ep}}_T(L^{\f{d}{1+\ep}})} +\|\na^2
u\|_{L^{\f8{8-\ep}}_T(L^{\f{2d}{2-\ep}})}\leq C_{T,\ep},
\end{split}
\eeno which together with \eqref{4.3} proves the first line of
\eqref{lem4.1q}. The second line of \eqref{4.3} follows along the
same lines.
\end{proof}

Thanks to Lemma \ref{lem4.1}, we can  taking $T$ small enough so
that \beq\label{lipassume} \int_0^T\|\na u(t)\|_{L^\infty}dt\leq
\f12. \eeq

As in \cite{dm,dm2}, we shall prove the uniqueness part of Theorem
\ref{thm1.1bis} by using the Lagrangian formulation of \eqref{INS}.
Toward this, we first recall some basic facts concerning Lagrangian
coordinates from \cite{dm, dm2}. By virtue of \eqref{lipassume}, for
any $y\in\R^d,$  the following ordinary differential equation has a
unique solution on $[0,T]$ \beq\label{ode} \f{dX(t,y)}{dt}= u(t,
X(t,y))\eqdefa v(t,y), \qquad X(t,y)|_{t=0}=y. \eeq This leads to
the following relation between the Eulerian coordinates $x$ and the
Lagrangian coordinates $y$: \beq\label{ode1} x=
X(t,y)=y+\int_0^tv(\tau,y)\,d\tau. \eeq Let $Y(t,\cdot)$ be the
inverse mapping of $X(t,\cdot)$. Then $D_xY(t,x)=(D_yX(t,y))^{-1}$
for $x=X(t,y)$. Providing $D_yX-Id$ is small enough, we have \beno
D_xY
=(Id+(D_yX-Id))^{-1}=\sum_{k=0}^{\infty}(-1)^k(\int_0^tD_yv(\tau,y)d\tau)^k.\eeno
We denote $A(t,y)\eqdefa (\na X(t,y))^{-1}=\na_x Y(t,x),$ then we
have \beq\label{h.12}
 \nabla_x u(t,x)=^TA(t,x)\nabla_y v(t,y)\quad\mbox{ and}\quad
\dive u(t,x)=\dive(A(t,y) v(t,y)). \eeq By the chain rule, we also
have \ben\label{eq:A-div} \dive_y\big(A\cdot\big)=^TA:\na_y.\een
Here and in what follows, we always denote $^TA$ the transpose
matrix of $A.$

As in \cite{dm,dm2}, we denote \beq\label{h.13}
\begin{split}
\na_u\eqdefa&  ^TA\cdot\na_y,\quad
\dive_u\eqdefa\dive(A\cdot)\quad\mbox{and}\quad
\D_u\eqdefa\dive_u\na_u,\\
b(t,y)&\eqdefa a(t, X(t,y)), \quad v(t,y)\eqdefa u(t,
X(t,y))\quad\mbox{and}\quad P(t,y)\eqdefa \Pi(t, X(t,y)).
\end{split}
\eeq Notice that for any $t>0,$  the  solution of \eqref{INS}
obtained in  Theorem \ref{thm1.1bis} satisfies the smoothness
assumption of Proposition 2 in \cite{dm2}, so that $(b, v, \na P)$
defined by \eqref{h.13} fulfils \begin{equation}\label{INSL}
 \quad\left\{\begin{array}{l}
\displaystyle b_t=0,\\
\displaystyle \pa_t v -(1+b)(\mu\tri_u v- \grad_u P)=0, \\
\displaystyle \dv_u\, v = 0, \\
\displaystyle (b,v)|_{t=0}=(a_0, u_0),
\end{array}\right.
\end{equation}
which is  the Lagrangian formulation of \eqref{INS}. For the sake of
simplicity, we shall take $\mu=1$ in what follows.

We now present the proof  of Theorem \ref{thm1.1bis}.

\begin{proof}[ Proof of Theorem
\ref{thm1.1bis}] We first deduce from the  proof to the existence
part  of Theorem \ref{thm1.1}, \eqref{3.1} and \eqref{3.1adf} that
the global weak solution $(a,u,\na\Pi)$ constructed in Theorem
\ref{thm1.1} satisfies \eqref{thm1cg} and \eqref{thm1d}.  It remains
to prove the uniqueness part of Theorem \ref{thm1.1bis}.

Let $(a_i,u_i,\Pi_i), i=1,2,$ be two solutions of \eqref{INS} which
satisfies (\ref{thm1b}-\ref{thm1c}) and (\ref{thm1cg}-\ref{thm1d}).
Let $X_i, (v_i, P_i), A_i, i=1,2$ be given by \eqref{ode1} and
\eqref{h.13}. We denote \beno \delta v\eqdefa v_2-v_1,\quad \d
P\eqdefa P_2-P_1,\eeno then $(\d v, \d P)$ solves
\begin{equation}\label{differ}
 \quad\left\{\begin{array}{l}
\displaystyle \pa_t \delta v -\tri\delta v+\grad \delta P=a_0(\tri\delta v- \grad \delta P)+\delta f_1 +\delta f_2, \\
\displaystyle \dv\, \delta v = \delta g, \\
\displaystyle  \pa_t \delta g =\dive \delta R,\\
\displaystyle \delta v|_{t=0}=0,
\end{array}\right.
\end{equation}
with \beno\begin{split} \delta f_1 &\eqdefa(1+a_0)[(Id- ^TA_2)\na
\delta P-\delta A \na
P_1],\\
\delta f_2 &\eqdefa\mu(1+a_0)\dive[(A_2 ^TA_2-Id)\na \delta v+ (A_2
^TA_2-
A_1 ^TA_1)\na v_1],\\
\delta g &\eqdefa (Id-A_2):D\delta v-\delta A:Dv_1,\\
\delta R &\eqdefa\pa_t[(Id-A_2)\delta v]-\pa_t[\delta A v_1].
\end{split} \eeno
In what follows, we will use repeatedly the following fact  (see
\cite{dm} for instance) that \beq\label{4.5} \delta A(t)=(\int_0^t
D\delta vd\tau)(\sum_{k\geq 1}\sum_{0\leq
j<k}C_i^jC_2^{k-1-j})\quad\mbox{with}\quad
C_i(t)\eqdefa\int_0^tDv_id\tau.\eeq Let the indices $\al_1,
\beta_1,\beta_2,\gamma_1$ and $\gamma_2$ be given by
\eqref{indexad}. As in Remark \ref{rmk3.1}, we take
$p_1=\f{d}{1+\e}$ and $p_2, p_3$ satisfying $\f{dr}{r-1}<p_3<\infty$
and $\f1{p_2}+\f1{p_3}=\f1{p_1}.$ We denote \beq\label{g1t}
\begin{split}
G(t) \eqdefa & \|t^{\gamma_1}\delta v\|_{L^{\infty}_t(L^{p_3})}+
\|t^{\gamma_2}\delta v\|_{L^{2r}_t(L^{p_3})}+\|t^{\beta_1}\na \delta
v\|_{L^{2r}_t(L^{p_2})} \\
&+ \|t^{\beta_2}\na \delta v\|_{L^{\infty}_t(L^{p_2})}
+\|(t^{\al_1}\pa_t\delta v, t^{\al_1}\na^2\delta v,
t^{\al_1}\na\delta P)\|_{L^{2r}_t(L^{\f{d}{1+\ep}})}. \end{split}
\eeq Then  we deduce from Proposition \ref{prop4.1} and
\eqref{differ} that \beq\label{4.6} G(t)\leq
C\bigl(\|t^{\al_1}\delta f_1\|_{L^{2r}_t(L^{\f{d}{1+\ep}})} +
\|t^{\al_1}\delta f_2\|_{L^{2r}_t(L^{\f{d}{1+\ep}})} +
\|t^{\al_1}\na\delta g\|_{L^{2r}_t(L^{\f{d}{1+\ep}})} +
\|t^{\al_1}\delta R\|_{L^{2r}_t(L^{\f{d}{1+\ep}})}\bigr) \eeq as
long as $C\|a_0\|_{L^\infty}\leq \f12.$

Let us now estimate term by term on the right-hand side of
\eqref{4.6}. We first get, by using \eqref{thm1b} and
\eqref{lipassume}, that \beno
\begin{split}
\|t^{\al_1}\delta f_1\|_{L^{2r}_t(L^{\f{d}{1+\ep}})} &\lesssim
\|Id-^TA_2\|_{L^\infty_t(L^\infty)}\|t^{\al_1}\na \delta
P\|_{L^{2r}_t(L^{\f{d}{1+\ep}})}+ \|\delta
A\|_{L^\infty_t(L^{\f{d}{\e}})}\|t^{\al_1}\na
P_1\|_{L^{2r}_t(L^{d})}
\\
&\lesssim\|\na v_2\|_{L^1_t(L^\infty)}G(t) +\|\na\delta
v\|_{L^1_t(L^{\f{d}{\ep}})}(\|t^{\al_1}\na
P_1\|_{L^{2r}_t(L^{p_1})})^\th (\|t^{\al_1}\na
P_1\|_{L^{2r}_t(L^{q_1})})^{1-\th},
\end{split}
\eeno for $\th$ determined by $\f1d=\f{\th}{p_1}+\f{1-\th}{q_1}.$
However by virtue of Sobolev  embedding theorem,
$W^{1,\f{d}{1+\ep}}(\R^d)\hookrightarrow L^{\f{d}{\ep}}(\R^d)$, one
has \beq \label{ess} \|\na\delta v\|_{L^1_t(L^{\f{d}{\ep}})}\lesssim
\|\na^2\delta v\|_{L^1_t(L^{\f{d}{1+\ep}})}\lesssim
t^{\f12(\f{d}{p_1}-1)}\|t^{\al_1}\na^2\d
v\|_{L^{2r}_t(L^{\f{d}{1+\e}})}, \eeq we obtain \beq\label{4.6af}
\|t^{\al_1}\delta f_1\|_{L^{2r}_t(L^{\f{d}{1+\ep}})} \lesssim
\eta(t)G(t),\eeq for some positive continuous function $\eta(t)$
which tends to $0$ as  $t\to 0$. Along the same line, we deduce
\beno
\begin{split}
\|t^{\al_1}&\na ((Id-A_2):D\delta v)\|_{L^{2r}_t(L^{\f{d}{1+\ep}})}\\
&\lesssim \|DA_2\otimes t^{\al_1}D\delta
v\|_{L^{2r}_t(L^{\f{d}{1+\ep}})}+\|(Id-A_2)\otimes
t^{\al_1}D^2\delta
v\|_{L^{2r}_t(L^{\f{d}{1+\ep}})}\\
&\lesssim \|\na^2 v_2\|_{L^1_t(L^d)}\|t^{\al_1}\na\delta
v\|_{L^{2r}_t(L^{\f{d}{\ep}})} + \|\na
v_2\|_{L^1(L^\infty)}\|t^{\al_1}\na^2\delta
v\|_{L^{2r}_t(L^{\f{d}{1+\ep}})}\\
&\lesssim \|t^{\al_1}\na^2\delta v\|_{L^{2r}_t(L^{\f{d}{1+\ep}})}
(\|\na^2 v_2\|_{L^1_t(L^d)}+\|\na v_2\|_{L^1(L^\infty)}),
\end{split}
\eeno and \beno
\begin{split}
\|t^{\al_1}&\na (\delta A:Dv_1)\|_{L^{2r}_t(L^{\f{d}{1+\ep}})}\\
&\lesssim \| t^{\al_1}|\na v_1|\int_0^\tau |\na^2 \delta
v|d\tau'\|_{L^{2r}_t(L^{\f{d}{1+\ep}})}+ \|
t^{\al_1}|\na^2v_1|\int_0^\tau |\na \delta
v|d\tau'\|_{L^{2r}_t(L^{\f{d}{1+\ep}})}\\
&\lesssim \|t^{\al_1}\na v_1\|_{L^{2r}_t(L^{\infty})}\|\na^2 \delta
v\|_{L^1_t(L^{\f{d}{1+\e}})} + \|t^{\al_1}\na^2
v_1\|_{L^{2r}_t(L^{d})}\|\na\delta
v\|_{L^1_t(L^{\f{d}\e})}\\
&\lesssim t^{\f12(\f{d}{p_1}-1)}\|t^{\al_1}\na^2 \d
v\|_{L^{2r}_t(L^{\f{d}{1+\ep}})} (\|t^{\al_1}\na
v_1\|_{L^{2r}_t(L^{\infty})}+ \|t^{\al_1}\na^2
v_1\|_{L^{2r}_t(L^{d})}).
\end{split}
\eeno So that for all $t\in [0,T]$, we get, by using
\eqref{lem4.1q}, that \beq\label{4.6ah} \|t^{\al_1}\na \delta
g\|_{L^{2r}_t(L^{\f{d}{1+\ep}})} \lesssim \eta(t)G(t). \eeq The
estimate to the term $\delta f_2$ can be handled along  the same
line.

To deal with  $\delta R,$ we denote $Dv_{1,2}$ to be the components
of $Dv_1$ and $Dv_2$. Then  we get , by using \eqref{4.5} once
again, that \beno\begin{split} \|t^{\al_1}\pa_t((Id-A_2)\delta
v)\|_{L^{2r}_t(L^{\f{d}{1+\ep}})}
 &\lesssim \|t^{\beta_1}Dv_2
t^{\gamma_1}\delta v\|_{L^{2r}_t(L^{\f{d}{1+\ep}})}+
\|(Id-A_2)t^{\al_1}\pa_t \delta
v\|_{L^{2r}_t(L^{\f{d}{1+\ep}})}\\
&\lesssim \|t^{\beta_1}\na
v_2\|_{L^{2r}_t(L^{p_2})}\|t^{\gamma_1}\delta
v\|_{L^{\infty}_t(L^{p_3})}\\
&\qquad+ \|\na v_2\|_{L^1_t(L^\infty)}\|t^{\al_1}\pa_t \delta
v\|_{L^{2r}_t(L^{\f{d}{1+\ep}})},
\end{split}\eeno
and \beno\begin{split} \|t^{\al_1}\pa_t(\delta A
v_1)\|_{L^{2r}_t(L^{\f{d}{1+\ep}})} \lesssim& \|t^{\gamma_2}v_1
t^{\beta_2}D\delta v \|_{L^{2r}_t(L^{\f{d}{1+\ep}})}+ \|\delta A
t^{\al_1}\pa_t v_1\|_{L^{2r}_t(L^{\f{d}{1+\ep}})}\\
&+ \|\int_0^\tau |D\delta v|d\tau't^{\beta_2}|Dv_{1,2}|t^{\gamma_2}
|v_1|\|_{L^{2r}_t(L^{\f{d}{1+\ep}})}\\
 \lesssim&
\|t^{\beta_2}\na\delta
v\|_{L^{\infty}_t(L^{p_2})}\|t^{\gamma_2}v_1\|_{L^{2r}_t(L^{p_3})}
+\|\na\delta v\|_{L^1_t(L^{\f{d}\e})}\|t^{\al_1}\pa_t
v_1\|_{L^{2r}_t(L^{d})}\\
&+ \|\na\delta v\|_{L^1_t(L^{\f{d}\e})}\bigl(\|t^{\beta_2}\na
v_{1,2}\|_{L^{\infty}_t(L^{p_2})}\|t^{\gamma_2}v_1\|_{L^{2r}_t(L^{p_3})}\bigr)^\th\\
&\qquad\times \bigl(\|t^{\beta_2}\na
v_{1,2}\|_{L^{\infty}_t(L^{q_2})}\|t^{\gamma_2}v_1\|_{L^{2r}_t(L^{q_3})}\bigr)^{1-\th},
\end{split}
\eeno for $\th$ given by $\f1d=\f\th{p_1}+\f{1-\th}{q_1}.$ Hence it
follows from \eqref{thm1b} and \eqref{thm1cg} that \beq\label{4.6al}
\|t^{\al_1}\delta R\|_{L^{2r}_t(L^{\f{d}{1+\ep}})} \lesssim
\eta(t)G(t).\eeq

Substituting (\ref{4.6af}-\ref{4.6al}) into \eqref{4.6}  results in
\beq\label{4.7}\begin{split} G(t) &\lesssim \eta(t) G(t),
\end{split}\eeq which implies the uniqueness of the solution to \eqref{INS} on a sufficiently small
time interval. Then  uniqueness part of Theorem \ref{thm1.1bis} can
be completed by a bootstrap method. This concludes the proof of
Theorem \ref{thm1.1bis}.
 \end{proof}


\setcounter{equation}{0}
\section{Proof of Theorem \ref{thm1.2}}

The goal of this section is to present the proof of Theorem
\ref{thm1.2}. Indeed given  $a_0\in L^\infty(\R^d)\cap
B^{\f{d}q+\ep}_{q,\infty}(\R^d),$ $u_0\in\dot
B^{-1+\f{d}p-\ep}_{p,r}(\R^d)\cap \dot B^{-1+\f{d}p}_{p,r}(\R^d)$
with $ \|a_0\|_{L^\infty\cap B^{\f{d}q+\ep}_{q,\infty}}$ being
sufficiently small and $p,q,\ep$ satisfying the conditions listed in
Theorem \ref{thm1.2}, we deduce from \cite{Haspot} that there exists
a positive time $T$ so that \eqref{INS} has a unique solution $(a,u,
\na\Pi)$ with \beq
\begin{split}
&a\in C([0,T]; L^\infty(\R^d)\cap
B^{\f{d}q+\f\e2}_{q,\infty}(\R^d)),\quad u\in
\wt{L}^\infty_T(\dot{B}^{-1+\f{d}p-s}_{p,r})\cap \wt{L}^1_T(
\dot{B}^{1+\f{d}p-s}_{p,r})\quad\\
& \mbox{and}\quad \na \Pi\in \wt{L}^1_T(
\dot{B}^{-1+\f{d}p-s}_{p,r})\quad \mbox{for}\quad s=0\ \ \ \and\ \
\ep.
\end{split}\label{5.14}
\eeq We denote $T^\ast$ to be the lifespan of this solution. Then
the proof of Theorem \ref{thm1.2} is reduced to prove that
$T^\ast=\infty.$

\subsection{The estimate of the density}  Notice from \eqref{5.14} that
$u\in \wt{L}^1_T(\dot{B}^{-1+\f{d}{p}-\e}_{p,r})\cap
\wt{L}^1_T(\dot{B}^{-1+\f{d}{p}}_{p,r}),$  which is not Lipschitz in
the space variables,
 the regularity of the solution
$a$ to \eqref{INS} may be coarsen for positive time. In order to
applying the losing derivative estimate in \cite{BCD}, we first need
to prove that $u\in L^1_T(C_\mu)$ for $\mu(r)=r(-\log r)^\al$ and
some $\al\in (0,1).$ Indeed, let $u\in \wt{L}^1_T(
\dot{B}^{1+\f{d}p-\e}_{p,r})\cap
\wt{L}^1_T(\dot{B}^{1+\f{d}p}_{p,r})$ for some $\ep>0,$ we denote
$\theta(t,x,y)\eqdefa|u(t,x)-u(t,y)|.$ Then similar to the proof of
Proposition 2.1 in  \cite{CL}, for any positive integer $N$ and
$\ep_1>0,$ one has \beno
\begin{split}
\theta(t,x,y)\leq& |x-y|(2+N)^{1-\f1r+\ep_1}\sum_{-1\leq j\leq
N}\f{\|\na \D_j
u(t)\|_{L^\infty}}{(2+j)^{1-\f1r+\ep_1}}\\
&+2\sum_{j>N} 2^{-(2+j)}(2+j)^{1-\f1r+\ep_1}\f{2^{2+j}\|\D_j
u(t)\|_{L^\infty}}{(2+j)^{1-\f1r+\ep_1}}. \end{split}\eeno  Note
that $2^xx^\alpha$ with $0<\alpha<1$ is a decreasing function, we
get, by taking  $N=[1-\log|x-y|]-2$ in the above inequality, that
\beno \theta(t,x,y) \leq C|x-y|(1-\log|x-y|)^{1-\f1r+\ep_1}
\sum_{j\geq -1}\f{\|\na\D_j
u(t)\|_{L^\infty}}{(2+j)^{1-\f1r+\ep_1}},\eeno from which, we infer
\beq\label{8.0}\begin{split} \int_0^T&\sup_{0<|x-y|<1}
\f{\theta(t,x,y)}{|x-y|(1-\log|x-y|)^{1-\f1r+\ep_1}}\,dt\\
&\leq C (\sum_{j\geq -1} \|\na\D_j
u\|_{L^1_T(L^\infty)}^r)^{\f1r}(\sum_{j\geq -1}
\f1{(2+j)^{(1-\f1r+\ep_1)r'}})^{\f1{r'}} \leq C_{\ep_1}\|\na
u\|_{\wt{L}^1_T(B^0_{\infty,r})}, \end{split}\eeq where $r'$ denotes
the conjugate number of $r.$

On the other hand,  it follows from Lemma \ref{lem2.1} that \beno
\|\na\D_{-1}u\|_{L^1_T(L^\infty)}\leq
C_{\ep}\|u\|_{\wt{L}^1_T(\dot{B}^{1+\f{d}p-\ep}_{p,r})}, \eeno so
that \beno \|\na u\|_{\wt{L}^1_T(B^0_{\infty,r})} \leq C_{\ep}\bigl(
\|u\|_{\wt{L}^1_T(\dot{B}^{1+\f{d}p-\ep}_{p,r})}
+\|u\|_{\wt{L}^1_T(\dot{B}^{1+\f{d}p}_{p,r})}\bigr),\eeno which
together with \eqref{8.0} and Theorem 3.33 of \cite{BCD} (see also
\cite{danchin3}) implies that \beq\label{aest}
\|a\|_{\wt{L}^\infty_t(B^{\f{d}q+\f{\e}2}_{q,\infty})}\leq
C_{\ep}\|a_0\|_{B^{\f{d}q+\ep}_{q,\infty}}\exp\Bigl\{C_\ep(\|u\|_{\wt{L}^1_t(\dot{B}^{1+\f{d}p-\ep}_{p,r})}+\|u\|_{\wt{L}^1_t(\dot{B}^{1+\f{d}p}_{p,r})})\Bigr\},
\eeq for any  $t<T,$ $\f{d}q<1+\f{d}p$ and $0<\ep<1+\f{d}p-\f{d}q$.

\subsection{The estimate of the pressure}
We first get, by taking space divergence to the momentum equation of
\eqref{INS}, that \beno
 -\D\Pi=\dive(a\na\Pi)-\mu\dive (a\D
u)+\dive_h\dive_h(u^h\otimes
u^h)+\dive_h(u^d\p_du^h)+\p_d(u^h\cdot\na_hu^d)+\p_d^2(u^d)^2. \eeno
Thanks to $\dive u=0,$ one has \beno
\begin{split}\p_d(u^h\cdot\na_hu^d)=& \pa_du^h\cdot\na_h
u^d+u^h\cdot\na_h\pa_du^d\\
=&\dive_h(u^d\pa_du^h)-u^d\pa_d\dive_hu^h +
\dive_h(u^h\pa_du^d)-\dive_hu^h\pa_du^d\\
=&\dive_h(u^d\pa_du^h)-\pa_d(u^d\dive_hu^h)-\dive_h(u^h\dive_hu^h),
\end{split} \eeno which gives rise to \beq\label{5.1}\begin{split} -\D\Pi=&\dive(a\na\Pi)-\mu\dive (a\D
u)+\dive_h\dive_h(u^h\otimes
u^h)+2\dive_h(u^d\p_du^h)\\
&-3\pa_d(u^d\dive_hu^h)-\dive_h(u^h\dive_hu^h). \end{split}\eeq

The following proposition concerning the estimate of the pressure
will be the key ingredient used in the estimate of the horizontal
component of the velocity.

\begin{prop}\label{prop5.1}
{\sl Let $1<q\leq p<2d$, $r>1$ and $0<\ep<\f{2d}p-1.$ let $a\in
L^\infty_T(L^\infty)\cap \wt{L}^\infty_T(B^{\f{d}q}_{q,\infty}),$
$u\in \wt{L}^\infty_T(\dot{B}^{-1+\f{d}p-s}_{p,r})\cap
\wt{L}^1_T(\dot{B}^{1+\f{d}p-s}_{p,r}),$ for $s=0$ and $\ep,$ be the
unique local solution of \eqref{INS} given by \eqref{5.14}. We
denote \beq\label{5.2}
 f(t)\eqdefa \|u^d(t)\|^{2r}_{\dot{B}^{-1+\f{d}p+\f1r}_{p,r}} \quad
\mbox{ and} \quad\Pi_{\la}\eqdefa \Pi
\exp\Bigl\{-\la\int_0^tf(t')\,dt'\Bigr\}\quad\mbox{for}\quad \la>0,
\eeq and similar notation for $u_{\la}.$ Then \eqref{5.1} has a
unique solution $\na\Pi\in
\wt{L}^1_T(\dot{B}^{-1+\frac{d}p-\e}_{p,r})\cap
\wt{L}^1_T(\dot{B}^{-1+\frac{d}p}_{p,r})$ which decays to zero when
$|x|\to\infty$ so that for all $t\in [0,T],$ there holds \beq
\label{5.3}
\begin{split}
\|\na\Pi_{\la}\|_{\wt{L}^1_t(\dot{B}^{-1+\frac{d}p-s}_{p,r})}\leq &
\frac{C}{1-C\|a\|_{L^\infty_t(L^\infty)\cap\wt{L}^\infty_t(
B^{\f{d}q}_{q,\infty})}}\Bigl\{\|u^h_{\la}\|_{\wt{L}^1_t(\dot{B}^{1+\f{d}p-s}_{p,r})}^{1-\f1{2r}}\|u^h_{\la}\|_{\wt{L}^1_{t,f}(\dot{B}^{-1+\f{d}p-s}_{p,r})}^{\f1{2r}}
\\
&\bigl(\mu\|a\|_{L^\infty_t(L^\infty)\cap\wt{L}^\infty_t(
B^{\f{d}q}_{q,\infty})}
+\|u^h\|_{\wt{L}^\infty_t(\dot{B}^{-1+\f{d}p}_{p,r})}\bigr)\|u^h_{\la}\|_{\wt{L}^1_t(\dot{B}^{1+\f{d}p-s}_{p,r})}\\
&+\mu \|a\|_{L^\infty_t(L^\infty)\cap\wt{L}^\infty_t(
B^{\f{d}q}_{q,\infty})}\|u^d\|_{\wt{L}^1_t(\dot{B}^{1+\f{d}p-s}_{p,r})}\Bigr\}\quad\mbox{for}\
\ s=0,\e,
\end{split}
\eeq provided that $C\|a\|_{L^\infty_T(L^\infty)\cap\wt{L}^\infty_T(
B^{\f{d}q}_{q,\infty})}\leq \f12,$ and where the  norms of
$\wt{L}^1_t(\dot{B}^{1+\f{d}p-s}_{p,r})$ and
$\wt{L}^1_{t,f}(\dot{B}^{1+\f{d}p-s}_{p,r})$
 are given by Definitions \ref{chaleur+} and
\ref{defa.3} respectively.}
\end{prop}

The proof of this proposition will mainly be based on the following
lemmas:

\begin{lem}\label{lem5.1}
{\sl For $s<\f{d}p,$ one has \beno
\begin{split}
&\|g h\|_{\wt{L}^1_t(\dot{B}^{\f{d}p-s}_{p,r})}\lesssim
\|g\|_{\wt{L}^\infty_t(\dot{B}^{-1+\f{d}p}_{p,r})}
\|h\|_{\wt{L}^1_t(\dot{B}^{1+\f{d}p-s}_{p,r})}
+\|g\|_{\wt{L}^1_t(\dot{B}^{1+\f{d}p}_{p,r})}\|h\|_{\wt{L}^\infty_t(\dot{B}^{-1+\f{d}p-s}_{p,r})}.
\end{split}
\eeno}
\end{lem}

\begin{proof} Applying Bony's decomposition (\cite{Bo}) gives
\beno \dot{\D}_j(g h)=\sum_{j'\geq
j-N_0}\dot{\D}_j\bigl(\dot{S}_{j'}g\dot{\D}_{j'}
h+\dot{\D}_{j'}g\dot{S}_{j'+1}h\bigr), \eeno from which and Lemma
\ref{lem2.1}, we infer \beno
\begin{split}
\|\dot{\D}_j(gh)\|_{L^1_t(L^p)} \lesssim &\sum_{j'\geq
j-N_0}\bigl(\|\dot{S}_{j'}g\|_{L^\infty_t(L^\infty)}\|\dot{\D}_{j'}h\|_{L^1_t(L^p)}
+\|\dot{\D}_{j'}g\|_{L^1_t(L^p)}\|\dot{S}_{j'+1}h\|_{L^\infty_t(L^\infty)}\bigr)\\
\lesssim &\sum_{j'\geq j-N_0}
c_{j',r}2^{-j'(\frac{d}p-s)}\bigl(\|g\|_{\wt{L}^\infty_t(\dot{B}^{-1+\f{d}p}_{p,r})}\|h\|_{\wt{L}^1_{t}(\dot{B}^{1+\f{d}p-s}_{p,r})}\\
&\qquad\qquad\qquad\qquad\qquad\qquad+\|g\|_{\wt{L}^1_t(\dot{B}^{1+\f{d}p}_{p,r})}\|h\|_{\wt{L}^\infty_t(\dot{B}^{-1+\f{d}p-s}_{p,r})}
\bigr)\\
\lesssim&
c_{j,r}2^{-j(\frac{d}p-s)}\bigl(\|g\|_{\wt{L}^\infty_t(\dot{B}^{-1+\f{d}p}_{p,r})}\|h\|_{\wt{L}^1_{t}(\dot{B}^{1+\f{d}p-s}_{p,r})}
+\|g\|_{\wt{L}^1_t(\dot{B}^{1+\f{d}p}_{p,r})}\|h\|_{\wt{L}^\infty_t(\dot{B}^{-1+\f{d}p-s}_{p,r})}
\bigr),
\end{split}
\eeno where and in what follows, we always denote
$(c_{j,r})_{j\in\Z}$  as a generic element of $\ell^r(\Z)$ so that
$\sum_{j\in\Z}c_{j,r}^r=1.$ Then by virtue of Definition
\ref{chaleur+}, we complete the proof of the lemma.\end{proof}

\begin{lem}\label{lem5.2}
{\sl Let $1\leq p<2d,$ $s\in \bigl(-\f1r,\f{2d}p-1)$ and $f$ be
given by \eqref{5.2}. Then under the assumptions of Proposition
\ref{prop5.1}, one has \beq\label{5.4} \|u^d\na
u^h\|_{\wt{L}^1_t(\dot{B}^{-1+\f{d}p-s}_{p,r})}\lesssim
\|u^h\|_{\wt{L}^1_t(\dot{B}^{1+\f{d}p-s}_{p,r})}^{1-\f1{2r}}\|u^h\|_{\wt{L}^1_{t,f}(\dot{B}^{-1+\f{d}p-s}_{p,r})}^{\f1{2r}},
\eeq and \beq\label{5.4a}\begin{split}
 \|u^d\na u^h\|_{\wt{L}^1_t(\dot{B}^{-1+\f{d}p-s}_{p,r})}\lesssim\|u^h\|_{\wt{L}^\infty_t(\dot{B}^{-1+\f{d}p}_{p,r})}
 \|u^d\|_{\wt{L}^1_t(\dot{B}^{1+\f{d}p-s}_{p,r})}+
 \|u^d\|_{\wt{L}^\infty_t(\dot{B}^{-1+\f{d}p-s}_{p,r})}\|u^h\|_{\wt{L}^1_t(\dot{B}^{1+\f{d}p}_{p,r})}.\end{split}\eeq}
\end{lem}

\begin{proof} We first get, by applying Bony's decomposition \eqref{bony}, that
\beq\label{5.5} u^d\na u^h=\dot{T}_{u^d}\na u^h+\dot{T}_{\na
u^h}u^d+\dot{R}(u^d,\na u^h). \eeq Applying Lemma \ref{lem2.1} gives
\beno
\begin{split}
\|\dot{\D}_j\bigl(T_{u^d}\na u^h\bigr)(t')\|_{L^p}\lesssim
&\sum_{|j'-j|\leq
5}2^{j'}\|\dot{S}_{j'-1}u^d(t')\|_{L^\infty}\|\dot{\D}_{j'}u^h(t')\|_{L^p}\\
\lesssim&\sum_{|j'-j|\leq
5}2^{j'(2-\f1r)}\|u^d(t')\|_{\dot{B}^{-1+\f{d}p+\f1r}_{p,r}}\|\dot{\D}_{j'}u^h(t')\|_{L^p},
\end{split}
\eeno integrating the above inequality over $[0,t]$ and using
Definition \ref{defa.3}, one has \beno
\begin{split}
\|\dot{\D}_j&\bigl(\dot{T}_{u^d}\na
u^h\bigr)\|_{L^1_t(L^p)}\\
\lesssim &\sum_{|j'-j|\leq
5}2^{j'(2-\f1r)}\Bigl\{\int_0^t\|u^d(t')\|_{\dot{B}^{-1+\f{d}p+\f1r}_{p,r}}^{2r}
\|\dot{\D}_{j'}u^h(t')\|_{L^p}\,dt'\Bigr\}^{\f1{2r}}\|\dot{\D}_{j'}u^h(t')\|_{L^1_t(L^p)}^{1-\f1{2r}}\\
\lesssim
&c_{j,r}2^{-j(-1+\f{d}p-s)}\|u^h\|_{\wt{L}^1_t(\dot{B}^{1+\f{d}p-s}_{p,r})}^{1-\f1{2r}}\|u^h\|_{\wt{L}^1_{t,f}(\dot{B}^{-1+\f{d}p-s}_{p,r})}^{\f1{2r}}.
\end{split}
\eeno It follows from the same line that  \beno
\begin{split}
\|\dot{\D}_j\bigl(\dot{T}_{\na u^h}u^d\bigr)(t')\|_{L^p}\lesssim&
\sum_{|j'-j|\leq
5}\|\dot{S}_{j'-1}\na u^h(t')\|_{L^\infty}\|\dot{\D}_{j'}u^d(t')\|_{L^p}\\
\lesssim &2^{-j(-1+\f{d}p+\f1r)}\sum_{\ell\leq
j+4}2^{\ell(1+\f{d}p)}\|u^d(t')\|_{\dot{B}^{-1+\f{d}p+\f1r}_{p,r}}\|\dot{\D}_\ell
u^h(t')\|_{L^p},
\end{split}
\eeno from which and $s>-\f1r,$ we infer \beno
\begin{split}
\|\dot{\D}_j&\bigl(\dot{T}_{\na
u^h}u^d\bigr)\|_{L^1_t(L^p)}\\
\lesssim& 2^{-j(-1+\f{d}p+\f1r)}\sum_{\ell\leq
j+4}2^{\ell(1+\f{d}p)}\Bigl\{\int_0^t\|u^d(t')\|_{\dot{B}^{-1+\f{d}p+\f1r}_{p,r}}^{2r}\|\dot{\D}_{\ell}u^h(t')\|_{L^p}\,dt'\Bigr\}^{\f1{2r}}\|\dot{\D}_{\ell}u^h(t')\|_{L^1_t(L^p)}^{1-\f1{2r}}\\
\lesssim &
c_{j,r}2^{-j(-1+\f{d}p-s)}\|u^h\|_{\wt{L}^1_t(\dot{B}^{1+\f{d}p-s}_{p,r})}^{1-\f1{2r}}\|u^h\|_{\wt{L}^1_{t,f}(\dot{B}^{-1+\f{d}p-s}_{p,r})}^{\f1{2r}}.
\end{split}
\eeno Finally, we deal with the remaining term in \eqref{5.5}.
Firstly for $2\leq p<2d,$ as $s<\f{2d}p-1,$  we get, by applying
Lemma \ref{lem2.1}, that \beno
\begin{split}
\|\dot{\D}_j\bigl(\dot{R}(u^d,\na u^h)\bigr)\|_{L^1_t(L^p)} \lesssim
&2^{j\f{d}p}\sum_{j'\geq
j-N_0}2^{j'}\int_0^t\|\wt{\dot{\D}}_{j'}u^d(t')\|_{L^p}\|\dot{\D}_{j'}u^h(t')\|_{L^p}\,dt'\\
\lesssim &2^{j\f{d}p}\sum_{j'\geq
j-N_0}2^{-j'(-2+\frac{d}p+\f1r)}\int_0^t\|u^d(t')\|_{\dot{B}^{-1+\f{d}p+\f1r}_{p,r}}\|\dot{\D}_{j'}u^h(t')\|_{L^p}\,dt'\\
\lesssim &
c_{j,r}2^{-j(-1+\f{d}p-s)}\|u^h\|_{\wt{L}^1_t(\dot{B}^{1+\f{d}p-s}_{p,r})}^{1-\f1{2r}}\|u^h\|_{\wt{L}^1_{t,f}(\dot{B}^{-1+\f{d}p-s}_{p,r})}^{\f1{2r}}.
\end{split}
\eeno
 For $1\leq p< 2,$ we get, by applying Lemma \ref{lem2.1} once again, that
\beq\label{5.5a}
\begin{split}
\|\dot{\D}_j\bigl(\dot{R}&(u^d,\na u^h)\bigr)\|_{L^1_t(L^p)}\\
\lesssim &2^{jd(1-\f1p)}\sum_{j'\geq
j-N_0}2^{j'}\int_0^t\|\wt{\dot{\D}}_{j'}u^d(t')\|_{L^{\f{p}{p-1}}}\|\dot{\D}_{j'}u^h(t')\|_{L^p}\,dt'\\
\lesssim &2^{jd(1-\f1p)}\sum_{j'\geq
j-N_0}2^{j'(2+\f{d}p-d-\f1r)}\int_0^t\|u^d(t')\|_{\dot{B}^{-1+\f{d}p+\f1r}_{p,r}}\|\dot{\D}_{j'}u^h(t')\|_{L^p}\,dt'\\
\lesssim &
c_{j,r}2^{-j(-1+\f{d}p-s)}\|u^h\|_{\wt{L}^1_t(\dot{B}^{1+\f{d}p-s}_{p,r})}^{1-\f1{2r}}\|u^h\|_{\wt{L}^1_{t,f}(\dot{B}^{-1+\f{d}p-s}_{p,r})}^{\f1{2r}}.
\end{split}
\eeq Whence thanks to \eqref{5.5}, we finish the proof of
\eqref{5.4}.

On the other hand, it is easy to observe that \beno\begin{split}
\|&\dot{\D}_j(\dot{T}_{u^d}\na u^h+\dot{T}_{\na
u^h}u^d)\|_{L^1_t(L^p)}\\
&\lesssim \sum_{|j'-j|\leq
5}(\|\dot{S}_{j'-1}u^d\|_{L^\infty_t(L^\infty)}\|\dot{\D}_{j'}\na
u^h\|_{L^1_t(L^p)}+ \|\dot{S}_{j'-1}\na
u^h\|_{L^\infty_t(L^\infty)}\|\dot{\D}_{j'}u^d\|_{L^1_t(L^p)})\\
&\lesssim
c_{j,r}2^{-j(-1+\f{d}p-s)}\bigl(\|u^h\|_{\wt{L}^\infty_t(\dot{B}^{-1+\f{d}p}_{p,r})}\|u^d\|_{\wt{L}^1_t(\dot{B}^{1+\f{d}p-s}_{p,r})}+\|u^d\|_{\wt{L}^\infty_t(\dot{B}^{-1+\f{d}p-s}_{p,r})}\|u^h\|_{\wt{L}^1_t(\dot{B}^{1+\f{d}p}_{p,r})}\bigr).
\end{split}\eeno
Whereas  as $s<\f{2d}p-1,$ for $2\leq p<2d$, we get, by applying
Lemma \ref{lem2.1}, that\beno
\begin{split}
\|\dot{\D}_j\bigl(\dot{R}(u^d,\na u^h)\bigr)\|_{L^1_t(L^p)} \lesssim
&2^{j\f{d}p}\sum_{j'\geq
j-N_0}2^{j'}\|\dot{\D}_{j'}u^d\|_{L^1_t(L^p)}\|\wt{\dot{\D}}_{j'}u^h\|_{L^\infty_t(L^p)}\\
\lesssim
&c_{j,r}2^{-j(-1+\f{d}p-s)}\|u^h\|_{\wt{L}^\infty_t(\dot{B}^{-1+\f{d}p}_{p,r})}\|u^d\|_{\wt{L}^1_t(\dot{B}^{1+\f{d}p-s}_{p,r})}.
\end{split}
\eeno Along the same line to the proof of \eqref{5.5a}, we can prove
the same estimate holds  for $1\leq p<2$. This proves \eqref{5.4a}
and Lemma \ref{lem5.2}.
\end{proof}

\begin{lem}\label{lem5.3}
{\sl Let $1<q\leq p<2d$, $s<\f{d}p+\f{d}q-1$ and $g\in
\wt{L}^1_T(\dot{B}^{-1+\f{d}p-s}_{p,r}).$ Then under the assumptions
of Proposition \ref{prop5.1}, one has \beno
\|ag\|_{\wt{L}^1_t(\dot{B}^{-1+\f{d}p-s}_{p,r})}\lesssim
\|a\|_{L^\infty_t(L^\infty)\cap\wt{L}^\infty_t(
B^{\f{d}q}_{q,\infty})}\|g\|_{\wt{L}^1_t(\dot{B}^{-1+\f{d}p-s}_{p,r})}.
\eeno }
\end{lem}

\begin{proof}
Again thanks to Bony's decomposition \eqref{bony}, we have
\beq\label{5.6} ag=\dot{T}_a g +\dot{T}_g a + \dot{R}(a,g). \eeq
Applying Lemma \ref{lem2.1} gives \beno
\begin{split}\|\dot{\D}_j(\dot{T}_a g)\|_{L^1_t(L^p)}\lesssim&
\sum_{|j'-j|\leq
5}\|\dot{S}_{j'-1}a\|_{L^\infty_t(L^\infty)}\|\dot{\D}_{j'}g\|_{L^1_t(L^p)}\\
\lesssim&
c_{j,r}2^{-j(-1+\f{d}p-s)}\|a\|_{L^\infty_t(L^\infty)}\|g\|_{\wt{L}^1_t(\dot{B}^{-1+\f{d}p-s}_{p,r})}.
\end{split}
\eeno  While as $p\geq q$, applying Lemma \ref{lem2.1} once again
gives rise to \beno
\begin{split}
\|\dot{\D}_j(\dot{T}_g a)\|_{L^1_t(L^p)} \lesssim& \sum_{|j'-j|\leq
5}\|\dot{S}_{j'-1}g\|_{L^1_t(L^\infty)}\|\dot{\D}_{j'}a\|_{L^\infty_t(L^p)}\\
\lesssim&c_{j,r}2^{-j(-1+\f{d}p-s)}\|a\|_{\wt{L}^\infty_t(B^{\f{d}q}_{q,\infty})}\|g\|_{\wt{L}^1_t(\dot{B}^{-1+\f{d}p-s}_{p,r})}.
\end{split}
\eeno Finally as $s<\f{d}p+\f{d}q-1,$ for $\f1p+\f1q \leq 1$, we
get, by applying Lemma \ref{lem2.1}, that \beno
\begin{split}
\|\dot{\D}_j(\dot{R}(a,g))\|_{L^1_t(L^p)} \lesssim&
2^{j\f{d}{q}}\sum_{j'\geq
j-N_0}\|\wt{\dot{\D}}_{j'}a\|_{L^\infty_t(L^q)}\|\dot{\D}_{j'}g\|_{L^1_t(L^p)}\\
\lesssim& 2^{j\f{d}{q}}\sum_{j'\geq j-N_0}
c_{j',r}2^{j'(1-\f{d}p-\f{d}q+s)}\|a\|_{\wt{L}^\infty_t(B^{\f{d}q}_{q,\infty})}\|g\|_{\wt{L}^1_t(\dot{B}^{-1+\f{d}p-s}_{p,r})}\\
\lesssim&c_{j,r}2^{-j(-1+\f{d}p-s)}\|a\|_{\wt{L}^\infty_t(B^{\f{d}q}_{q,\infty})}\|g\|_{\wt{L}^1_t(\dot{B}^{-1+\f{d}p-s}_{p,r})}.
\end{split}
\eeno For the case when $\f1p+\f1q >1$, we have  \beno
\begin{split}
\|\dot{\D}_j(\dot{R}(a,g))\|_{L^1_t(L^p)} \lesssim&
2^{jd(1-\f1p)}\sum_{j'\geq
j-N_0}\|\wt{\dot{\D}}_{j'}a\|_{L^\infty_t(L^{\f{p}{p-1}})}\|\dot{\D}_{j'}g\|_{L^1_t(L^p)}\\
\lesssim& 2^{jd(1-\f1p)}\sum_{j'\geq j-N_0}c_{j',r}2^{j'(1-d+s)}\|a\|_{\wt{L}^\infty_t(B^{\f{d}q}_{q,\infty})}\|g\|_{\wt{L}^1_t(\dot{B}^{-1+\f{d}p-s}_{p,r})}\\
\lesssim&c_{j,r}2^{-j(-1+\f{d}p-s)}\|a\|_{\wt{L}^\infty_t(B^{\f{d}q}_{q,\infty})}\|g\|_{\wt{L}^1_t(\dot{B}^{-1+\f{d}p-s}_{p,r})}.
\end{split}
\eeno Whence thanks to \eqref{5.6}, we prove Lemma \ref{lem5.3}.
\end{proof}

We now turn to the proof of Proposition
\ref{prop5.1}.\\

\no{\bf Proof of Proposition \ref{prop5.1}.}\ Again as both the
proof of the existence and uniqueness of solutions to \eqref{5.1}
essentially follows from the estimates \eqref{5.3} for some
appropriate approximate solutions. For the sake of simplicity, we
just prove \eqref{5.3} for smooth enough solutions of \eqref{5.1}.
Indeed thanks to \eqref{5.1} and \eqref{5.2}, we write \beno
\begin{split}
\na\Pi_{\la}=\na(-\D)^{-1}\bigl[&\dive(a\na\Pi_{\la})+\dive_h\dive_h(u^h\otimes
u^h_{\la})+2\dive_h(u^d\p_du^h_{\la})-3\pa_d(u^d\dive_hu^h_{\la})\\
&-\dive_h(u^h\dive_hu^h_{\la})-\mu\dive_h(a\D u^h_{\la})-\mu\p_d(a\D
u^d_{\la})\bigr].
\end{split} \eeno Acting
 $\dot{\D}_j$ to the above equation and using Lemma \ref{lem2.1} leads
to \beq\label{5.7}
\begin{split}
\|\dot{\D}_j(\na\Pi_{\la})\|_{L^1_t(L^p)}\lesssim&
\|\dot{\D}_j(a\na\Pi_{\la})\|_{L^1_t(L^p)}+2^j\|\dot{\D}_j(u^h\otimes
u^h_{\la})\|_{L^1_t(L^p)}+\|\dot{\D}_j(u^h\dive_h
u^h_{\la})\|_{L^1_t(L^p)}\\
&+\|\dot{\D}_j(u^d\na u^h_{\la})\|_{L^1_t(L^p)}+\mu\|\dot{\D}_j(a\D
u^h_{\la})\|_{L^1_t(L^p)}+\mu\|\dot{\D}_j(a\D
u^d_{\la})\|_{L^1_t(L^p)},
\end{split}
\eeq  from which, and Lemma \ref{lem5.1} to Lemma \ref{lem5.3}, we
deduce that for $s=0$ and $\e$ \beno
\begin{split}
\|&\dot{\D}_j(\na\Pi_{\la})\|_{L^1_t(L^p)}\lesssim
c_{j,r}2^{-j(-1+\f{d}p-s)}\Bigl\{\|a\|_{L^\infty_t(L^\infty)\cap\wt{L}^\infty_t(
B^{\f{d}q}_{q,\infty})}\bigl(\mu\|u^h_{\la}\|_{\wt{L}^1_t(\dot{B}^{1+\f{d}p-s}_{p,r})}+\mu
\|u^d\|_{\wt{L}^1_t(\dot{B}^{1+\f{d}p-s}_{p,r})}\\
&+\|\na
\Pi_{\la}\|_{\wt{L}^1_t(\dot{B}^{-1+\f{d}p-s}_{p,r})}\bigr)+\|u^h\|_{\wt{L}^\infty_t(\dot{B}^{-1+\f{d}p}_{p,r})}\|u^h_{\la}\|_{\wt{L}^1_t(\dot{B}^{1+\f{d}p-s}_{p,r})}
+\|u^h_{\la}\|_{\wt{L}^1_t(\dot{B}^{1+\f{d}p-s}_{p,r})}^{1-\f1{2r}}\|u^h_{\la}\|_{\wt{L}^1_{t,f}(\dot{B}^{-1+\f{d}p-s}_{p,r})}^{\f1{2r}}\Bigr\}.
\end{split}
\eeno Therefore  \eqref{5.3} follows as long as
$C\|a\|_{L^\infty_T(L^\infty)\cap\wt{L}^\infty_T(
B^{\f{d}q}_{q,\infty})}\leq \f12.$ This finishes the proof of
Proposition \ref{prop5.1}. \ef

To deal with the estimate of $u^d,$ we also need the following
proposition:

\begin{prop}\label{prop5.2}
{\sl Under the assumptions of Proposition \ref{prop5.1}, one has for
$s=0$ and $\e$ \beq \label{5.8}
\begin{split}
\|\na \Pi&\|_{\wt{L}^1_t(\dot{B}^{-1+\f{d}p-s}_{p,r})}\leq
\frac{C}{1-C\|a\|_{L^\infty_t(L^\infty)\cap\wt{L}^\infty_t(
B^{\f{d}q}_{q,\infty})}}\Bigl\{\|u^h\|_{\wt{L}^1_t(\dot{B}^{1+\f{d}p}_{p,r})}
\|u^d\|_{\wt{L}^\infty_t(\dot{B}^{-1+\f{d}p-s}_{p,r})}\\
&+\bigl(\mu \|a\|_{L^\infty_t(L^\infty)\cap\wt{L}^\infty_t(
B^{\f{d}q}_{q,\infty})}
+\|u^h\|_{\wt{L}^\infty_t(\dot{B}^{-1+\f{d}p}_{p,r})}\bigr)
\bigl(\|u^h\|_{\wt{L}^1_t(\dot{B}^{1+\f{d}p-s}_{p,r})}+\|u^d\|_{\wt{L}^1_t(\dot{B}^{1+\f{d}p-s}_{p,r})}\bigr)\Bigr\}
\end{split}
\eeq for $t\leq T$ provided that
$C\|a\|_{L^\infty_T(L^\infty)\cap\wt{L}^\infty_T(
B^{\f{d}q}_{q,\infty})}\leq \f12.$}
\end{prop}

\begin{proof} The proof of this proposition exactly follows from that of
Proposition \ref{prop5.1}. In fact, taking $\la=0$ in \eqref{5.7},
and then applying Lemma \ref{lem5.1}, Lemma \ref{lem5.3}, we arrive
at \beno
\begin{split} \|\na\Pi\|_{\wt{L}^1_t(\dot{B}^{-1+\f{d}p-s}_{p,r})}\leq
&C\Bigl\{\|a\|_{L^\infty_t(L^\infty)\cap\wt{L}^\infty_t(
B^{\f{d}q}_{q,\infty})} \|\na
\Pi\|_{\wt{L}^1_t(\dot{B}^{-1+\f{d}p-s}_{p,r})}+\|u^h\|_{\wt{L}^1_t(\dot{B}^{1+\f{d}p}_{p,r})}
\|u^d\|_{\wt{L}^\infty_t(\dot{B}^{-1+\f{d}p-s}_{p,r})}\\
&+\bigl(\mu \|a\|_{L^\infty_t(L^\infty)\cap\wt{L}^\infty_t(
B^{\f{d}q}_{q,\infty})}
+\|u^h\|_{\wt{L}^\infty_t(\dot{B}^{-1+\f{d}p}_{p,r})}\bigr)\|u\|_{\wt{L}^1_t(\dot{B}^{1+\f{d}p-s}_{p,r})}
 \Bigr\}
\end{split}
\eeno  for $t\leq T,$ from which and the fact that
$C\|a\|_{L^\infty_T(L^\infty)\cap\wt{L}^\infty_T(
B^{\f{d}q}_{q,\infty})}\leq \f12,$ we conclude the proof of
\eqref{5.8}.
\end{proof}

\subsection{ The estimate of $u^h$} We first deduce from the
transport equation of \eqref{INS} that \beq\label{tse}
\|a\|_{L^\infty_t(L^\infty)}\leq
\|a_0\|_{L^\infty}\quad\mbox{for}\quad t<T^\ast. \eeq Let $f(t),$
$u_{\la},$ $
 \Pi_{\la}$ be given by \eqref{5.2}.
  Then thanks to \eqref{INS}, we write \beno
\p_tu^h_{\la}+\la f(t)u_{\la}^h-\mu\D u^h_{\la}=-u\cdot\na
u^h_{\la}-(1+a)\na_h\Pi_{\la}+\mu a\D u^h_{\la}. \eeno Applying the
operator $\dot{\D}_j$ to the above equation and then taking the
$L^2$ inner product of the resulting equation with
$|\dot{\D}_ju^h_{\la}|^{p-2}\dot{\D}_j u^h_{\la}$ (in the case when
$p\in (1,2),$ we need to make some modification as that in
\cite{Danchin1}), we obtain \beno
\begin{split}
\f1p&\f{d}{dt}\|\dot{\D}_ju^h_{\la}(t)\|_{L^p}^p+\la
f(t)\|\dot{\D}_ju^h_{\la}(t)\|_{L^p}^p
-\mu\int_{\R^d}\D\dot{\D}_ju^h_{\la}\ \bigl|\ |\dot{\D}_ju^h_{\la}|^{p-2}\dot{\D}_j u^h_{\la}\,dx\\
=&-\int_{\R^d}\bigl(\dot{\D}_j(u\cdot\na
u^h_{\la})+\dot{\D}_j((1+a)\na_h\Pi_{\la})-\mu \dot{\D}_j(a\D
u^h_{\la})\bigr)\ \bigl|\ |\dot{\D}_ju^h_{\la}|^{p-2}\dot{\D}_j
u^h_{\la}\,dx.
\end{split}
\eeno However thanks to \cite{Danchin1} (see also \cite{Plan2}),
there exists a positive constant $\bar{c}$ so that \beno
-\int_{\R^d}\D\dot{\D}_ju^h_{\la}\ \bigl|\
|\dot{\D}_ju^h_{\la}|^{p-2}\dot{\D}_j u^h_{\la}\,dx\geq
\bar{c}2^{2j}\|\dot{\D}_ju^h_{\la}\|_{L^p}^p, \eeno whence a similar
argument as that in \cite{Danchin1} gives rise to \beq\label{5.9}
\begin{split}
\|\dot{\D}_ju^h_{\la}\|_{L^\infty_t(L^p)}&+\la\int_0^t
f(t')\|\dot{\D}_ju^h_{\la}(t')\|_{L^p}\,dt'
+\bar{c}\mu2^{2j}\|\dot{\D}_ju^h_{\la}\|_{L^1_t(L^p)}\\
\leq &\|\dot{\D}_j(u\cdot\na
u^h_{\la})\|_{L^1_t(L^p)}+\|\dot{\D}_j((1+a)\na_h\Pi_{\la})\|_{L^1_t(L^p)}+\mu
\|\dot{\D}_j(a\D u^h_{\la})\|_{L^1_t(L^p)}.
\end{split}
\eeq For $s=0$ and $\e,$ applying  Lemma \ref{lem5.2} gives
\beno\begin{split}
 \|\dot{\D}_j(u\cdot\na u^h_{\la})\|_{L^1_t(L^p)}\leq&
 \|\dot{\D}_j(u^h\cdot\na_h
u^h_{\la})\|_{L^1_t(L^p)}+\|\dot{\D}_j(u^d \pa_du^h_{\la})\|_{L^1_t(L^p)}\\
\leq&
Cc_{j,r}2^{-j(-1+\f{d}p-s)}\bigl(\|u^h\|_{\wt{L}^\infty_t(\dot{B}^{-1+\f{d}p}_{p,r})}\|u^h_{\la}\|_{\wt{L}^1_t(\dot{B}^{1+\f{d}p-s}_{p,r})}\\
&+ \|
u^h_{\la}\|_{\wt{L}^1_{t}(\dot{B}^{1+\f{d}p-s}_{p,r})}^{1-\f1{2r}}\|
u^h_{\la}\|_{\wt{L}^1_{t,f}(\dot{B}^{-1+\f{d}p-s}_{p,r})}^{\f1{2r}}\bigr).
\end{split}
\eeno And applying Lemma \ref{lem5.3} and \eqref{tse} yields \beno
\|\dot{\D}_j(a\D u^h_{\la})\|_{L^1_t(L^p)}\leq
Cc_{j,r}2^{-j(-1+\f{d}p-s)}\bigl(\|a_0\|_{L^\infty}+\|a\|_{\wt{L}^\infty_t(
B^{\f{d}q}_{q,\infty})}\bigr)\|u^h_{\la}\|_{\wt{L}^1_{t}(\dot{B}^{1+\f{d}p-s}_{p,r})},
\eeno and \beno \|\dot{\D}_j((1+a)\na_h\Pi_{\la})\|_{L^1_t(L^p)}\leq
Cc_{j,r}2^{-j(-1+\f{d}p-s)}\bigl(1+\|a_0\|_{L^\infty}+\|a\|_{\wt{L}^\infty_t(
B^{\f{d}q}_{q,\infty})}\bigr)\|\na_h\Pi_{\la}\|_{\wt{L}^1_{t}(B^{-1+\f{d}p-s}_{p,r})}.
\eeno Now let $c_1$ be a small enough positive constant, which will
be determined later on, we define $\frak{T}$ by \beq
\begin{split}
\frak{T}\eqdefa \max\Bigl\{ t\in [0,T^\ast):\
&\|u^h\|_{\wt{L}^\infty_t(\dot{B}^{-1+\f{d}p-\e}_{p,r})}+\|u^h\|_{\wt{L}^\infty_t(\dot{B}^{-1+\f{d}p}_{p,r})}
+\mu\bigl(\|a_0\|_{L^\infty}\\
&+\|a\|_{\wt{L}^\infty_t( B^{\f{d}q}_{q,\infty})}
+\|u^h\|_{\wt{L}^1_t(\dot{B}^{1+\f{d}p-\e}_{p,r})}+\|u^h\|_{\wt{L}^1_t(\dot{B}^{1+\f{d}p}_{p,r})}\bigr)\leq
c_1\mu\Bigr\}. \end{split} \label{5.16} \eeq In particular,
\eqref{5.16}
 implies that $\|a_0\|_{L^\infty}+\|a\|_{\wt{L}^\infty_t(
B^{\f{d}q}_{q,\infty})}\leq c_1$ for $t\leq\frak{T}.$ Taking $c_1$
so small that $Cc_1\leq \f12,$  \eqref{5.3} and \eqref{tse} ensures
that \beno
\begin{split}
\|\dot{\D}_j((1+a)\na_h\Pi_{\la})\|_{L^1_t(L^p)}\leq&
Cc_{j,r}2^{-j(-1+\f{d}p-s)}\Bigl\{\|u^h_{\la}\|_{\wt{L}^1_t(\dot{B}^{1+\f{d}p-s}_{p,r})}^{1-\f1{2r}}\|u^h_{\la}\|_{\wt{L}^1_{t,f}(\dot{B}^{-1+\f{d}p-s}_{p,1})}^{\f1{2r}}
\\
&+\bigl[\mu(\|a_0\|_{L^\infty}+\|a\|_{\wt{L}^\infty_t(
B^{\f{d}q}_{q,\infty})})
+\|u^h\|_{\wt{L}^\infty_t(\dot{B}^{-1+\f{d}p}_{p,r})}\bigr]\|u^h_{\la}\|_{\wt{L}^1_t(\dot{B}^{1+\f{d}p-s}_{p,r})}\\
&+\mu(\|a_0\|_{L^\infty}+\|a\|_{\wt{L}^\infty_t(
B^{\f{d}q}_{q,\infty})})
\|u^d\|_{\wt{L}^1_t(\dot{B}^{1+\f{d}p-s}_{p,r})}\Bigr\}\quad\mbox{for}\quad
t\leq \frak{T}.
\end{split}
\eeno Substituting the above estimates into \eqref{5.9}, we obtain
for $s=0$ and $\e$ that \beno
\begin{split}
\|u^h_{\la}&\|_{\wt{L}^\infty_{t}(\dot{B}^{-1+\f{d}p-s}_{p,r})}+\la\|u^h_{\la}\|_{\wt{L}^1_{t,f}(\dot{B}^{-1+\f{d}p-s}_{p,r})}
+\bar{c}\mu\|u^h_{\la}\|_{\wt{L}^1_t(\dot{B}^{1+\f{d}p-s}_{p,r})}\\
\leq&
\|u_0^h\|_{\dot{B}^{-1+\f{d}p-s}_{p,r}}+\f{\bar{c}\mu}4\|u^h_{\la}\|_{\wt{L}^1_t(\dot{B}^{1+\f{d}p-s}_{p,r})}
+C\Bigl\{\f1{\mu^{2r-1}}\|u^h_{\la}\|_{\wt{L}^1_{t,f}(\dot{B}^{-1+\f{d}p-s}_{p,r})}\\
&+\bigl[\mu(\|a_0\|_{L^\infty}+\|a\|_{\wt{L}^\infty_t(
B^{\f{d}q}_{q,\infty})})+\|u^h\|_{\wt{L}^\infty_t(\dot{B}^{-1+\f{d}p}_{p,r})}\bigr]\|u^h_{\la}\|_{\wt{L}^1_t(\dot{B}^{1+\f{d}p-s}_{p,r})}\\
&+ \mu(\|a_0\|_{L^\infty}+\|a\|_{\wt{L}^\infty_t(
B^{\f{d}q}_{q,\infty})})\|u^d\|_{\wt{L}^1_t(\dot{B}^{1+\f{d}p-s}_{p,r})}\Bigr\}\quad\mbox{for}\quad
t\leq \frak{T}.
\end{split}
\eeno
 Taking $\la=
\f{2C}{\mu^{2r-1}}$ in the above inequality and thanks to
\eqref{5.16}, we get \beq\label{5.11a}
\begin{split}
\|u^h_{\la}&\|_{\wt{L}^\infty_{t}(\dot{B}^{-1+\f{d}p-s}_{p,r})}
+\f{\bar{c}}2\mu\|u^h_{\la}\|_{\wt{L}^1_t(\dot{B}^{1+\f{d}p-s}_{p,r})}\\
\leq&
\|u_0^h\|_{\dot{B}^{-1+\f{d}p-s}_{p,r}}+C\mu\bigl(\|a_0\|_{L^\infty}+\|a\|_{\wt{L}^\infty_t(
B^{\f{d}q}_{q,\infty})}\bigr)\|u^d\|_{\wt{L}^1_t(\dot{B}^{1+\f{d}p-s}_{p,r})}\quad\mbox{for}\quad
t\leq \frak{T},
\end{split}
\eeq provided that $c_1$ in \eqref{5.16} is so small that
$Cc_1\leq\f{\bar{c}}4.$

On the other hand, it is easy to observe from \eqref{5.2} and
Definition \ref{defa.3} that \beno
\begin{split}
\bigl(\|u^h\|_{\wt{L}^\infty_{t}(\dot{B}^{-1+\f{d}p-s}_{p,r})}
+\f{\bar{c}}2\mu\|u^h&\|_{\wt{L}^1_t(\dot{B}^{1+\f{d}p-s}_{p,r})}\bigr)\exp\Bigl\{-\int_0^t\la f(t')\,dt'\Bigr\}\\
&\leq
\|u^h_{\la}\|_{\wt{L}^\infty_{t}(\dot{B}^{-1+\f{d}p-s}_{p,r})}+\f{\bar{c}}2\mu\|u^h_{\la}\|_{\wt{L}^1_t(\dot{B}^{1+\f{d}p-s}_{p,r})},
\end{split}
\eeno from which, and \eqref{aest}, \eqref{5.16}, \eqref{5.11a}, we
infer that for $s=0, \e,$ and $t\leq \frak{T}, $ \beq \label{5.17}
\begin{split}
\|u^h&\|_{\wt{L}^\infty_{t}(\dot{B}^{-1+\f{d}p-s}_{p,r})}+\f{\mu}2\bigl(\|a\|_{\wt{L}^\infty_t(
B^{\f{d}q+\f\e2}_{q,\infty})}
+\|u^h\|_{\wt{L}^1_t(\dot{B}^{1+\f{d}p-s}_{p,r})}\bigr)\\
&\leq C\bigl[\mu(\|a_0\|_{L^\infty}+\|a_0\|_{
B^{\f{d}q+\ep}_{q,\infty}})+\|u_0^h\|_{\dot{B}^{-1+\f{d}p-s}_{p,r}}\bigr]\\
&\quad\times\exp\Bigl\{C_\e\bigl(\|u^d\|_{\wt{L}^1_t(\dot{B}^{1+\f{d}p-\e}_{p,r})}+\|u^d\|_{\wt{L}^1_t(\dot{B}^{1+\f{d}p}_{p,r})}
+\int_0^t\f1{\mu^{2r-1}}\|u^d(t')\|_{\dot{B}^{-1+\f{d}p+\f1r}_{p,r}}^{2r}\,dt'\bigr)\Bigr\}.
\end{split}
\eeq

\subsection{ The estimate of $u^d$}
 By virtue of \eqref{INS}, we
get, by a similar derivation of \eqref{5.9}, that \beq \label{5.12}
\begin{split}
\|\dot{\D}_ju^d&\|_{L^\infty_t(L^p)}+\bar{c}\mu 2^{2j}\|\dot{\D}_j
u^d\|_{L^1_t(L^p)}\leq \|\dot{\D}_ju_0^d\|_{L^p}\\
&+C\Bigl(\|\dot{\D}_j(u\cdot\na
u^d)\|_{L^1_t(L^p)}+\|\dot{\D}_j((1+a)\p_d\Pi)\|_{L^1_t(L^p)}+\mu\|\dot{\D}_j(a\D
u^d)\|_{L^1_t(L^p)}\Bigr). \end{split} \eeq Applying Lemma
\ref{lem5.1} and Lemma \ref{lem5.2} gives for $s=0$ and $\e$ that
\beno
\begin{split}
\|\D_j&(u\cdot\na u^d)\|_{L^1_t(L^p)}\lesssim
2^j\|\D_j(u^hu^d)\|_{L^1_t(L^p)}+\|\D_j(u^d\dive_hu^h)\|_{L^1_t(L^p)}\\
&\lesssim
c_{j,r}2^{-j(-1+\f{d}p-s)}\bigl(\|u^h\|_{\wt{L}^\infty_t(\dot{B}^{-1+\f{d}p}_{p,r})}
\|u^d\|_{\wt{L}^1_t(\dot{B}^{1+\f{d}p-s}_{p,r})}
+\|u^h\|_{\wt{L}^1_{t}(\dot{B}^{1+\f{d}p}_{p,r})}\|u^d\|_{\wt{L}^\infty_t(\dot{B}^{-1+\f{d}p-s}_{p,r})}\bigr).
\end{split}
\eeno Whereas  thanks to \eqref{5.16} and \eqref{tse}, we get, by
applying Lemma \ref{lem5.3} and Proposition \ref{prop5.2}, that
\beno
\begin{split}
\|\dot{\D}_j&((1+a)\p_d\Pi)\|_{L^1_t(L^p)}\leq
Cc_{j,r}2^{-j(-1+\f{d}p-s)}(1+\|a_0\|_{L^\infty}+\|a\|_{\wt{L}^\infty_t( B^{\f{d}q}_{q,\infty})})\|\p_d\Pi\|_{\wt{L}^1_t(\dot{B}^{-1+\f{d}p-s}_{p,r})}\\
&\leq
Cc_{j,r}2^{-j(-1+\f{d}p-s)}\Bigl\{\|u^h\|_{\wt{L}^1_t(\dot{B}^{1+\f{d}p}_{p,r})}
\|u^d\|_{\wt{L}^\infty_t(\dot{B}^{-1+\f{d}p-s}_{p,r})}\\
&\quad+\bigl[\mu(\|a_0\|_{L^\infty}+ \|a\|_{\wt{L}^\infty_t(
B^{\f{d}q}_{q,\infty})})+\|u^h\|_{\wt{L}^\infty_t(\dot{B}^{-1+\f{d}p}_{p,r})}\bigr]\bigl(\|u^h\|_{\wt{L}^1_t(\dot{B}^{1+\f{d}p-s}_{p,r})}+\|u^d\|_{\wt{L}^1_t(\dot{B}^{1+\f{d}p-s}_{p,r})}\bigr)
 \Bigr\},
 \end{split}
 \eeno
 for $t\leq \frak{T}.$
Substituting the above estimates into \eqref{5.12} leads to \beq
\label{5.13} \begin{split}
&\|u^d\|_{\wt{L}^\infty_t(\dot{B}^{-1+\f{d}p-s}_{p,r})}+\bar{c}\mu
\|u^d\|_{\wt{L}^1_t(\dot{B}^{1+\f{d}p-s}_{p,r})} \leq
\|u_0^d\|_{\dot{B}^{-1+\f{d}p-s}_{p,r}}+C\Bigl\{\|u^h\|_{\wt{L}^1_t(\dot{B}^{1+\f{d}p}_{p,r})}
\|u^d\|_{\wt{L}^\infty_t(\dot{B}^{-1+\f{d}p-s}_{p,r})}\\
&\quad+\bigl[\mu(\|a_0\|_{L^\infty}+ \|a\|_{\wt{L}^\infty_t(
B^{\f{d}q}_{q,\infty})})+\|u^h\|_{\wt{L}^\infty_t(\dot{B}^{-1+\f{d}p}_{p,r})}\bigr]\bigl(\|u^h\|_{\wt{L}^1_t(\dot{B}^{1+\f{d}p-s}_{p,r})}+\|u^d\|_{\wt{L}^1_t(\dot{B}^{1+\f{d}p-s}_{p,r})}\bigr)
 \Bigr\},
 \end{split}
 \eeq
 for $t\leq \frak{T}$ and $s=0,\e.$

In particular, if  we take $c_1\leq\min\bigl\{\f1{2C},
\f{\bar{c}}{4C}\bigr\}$ in \eqref{5.16}, we deduce from \eqref{5.13}
that \beq\label{5.13a}
 \|u^d\|_{\wt{L}^\infty_{t}(\dot{B}^{-1+\f{d}p-s}_{p,r})}+\mu\bar{c}
\|u^d\|_{\wt{L}^1_t(\dot{B}^{1+\f{d}p-s}_{p,r})}\leq
2\|u_0^d\|_{\dot{B}^{-1+\f{d}p-s}_{p,r}}+c_2\mu \eeq for $t\leq
\frak{T}$ and $s=0,\e.$

\subsection{ The proof of Theorem \ref{thm1.2}}

According to the arguments at the beginning of this section,  we
only need to prove that $\frak{T}=\infty$ under the assumption of
\eqref{small2}. Otherwise, if $\frak{T}<T^\ast<\infty,$ we first
deduce from \eqref{5.13a} that \beq\label{5.18}\begin{split}
\|u^d\|_{L^{2r}_t(\dot{B}^{-1+\f{d}p+\f1r}_{p,r})} \leq& C
\|u^d\|_{\wt{L}^{2r}_t(\dot{B}^{-1+\f{d}p+\f1r}_{p,r})}\leq C
\|u^d\|_{\wt{L}^{1}_t(\dot{B}^{1+\f{d}p}_{p,r})}^{\f1{2r}}\|u^d\|_{\wt{L}^\infty_t(\dot{B}^{-1+\f{d}p}_{p,r})}^{1-\f1{2r}}\\
\leq&C
\mu^{-\f1{2r}}(\|u_0^d\|_{\dot{B}^{-1+\f{d}p}_{p,r}}+c_2\mu)\quad\mbox{for}\quad
t\leq\frak{T}.
\end{split}\eeq
Substituting \eqref{5.13a} and \eqref{5.18} into \eqref{5.17} gives
rise to
 \beno
\begin{split}
&\|u^h\|_{\wt{L}^\infty_{t}(\dot{B}^{-1+\f{d}p-\e}_{p,r})}+\|u^h\|_{\wt{L}^\infty_{t}(\dot{B}^{-1+\f{d}p}_{p,r})}
+\mu\bigl(\|a\|_{\wt{L}^\infty_t(B^{\f{d}q+\f\e2}_{q,\infty})}
+\|u^h\|_{\wt{L}^1_t(\dot{B}^{1+\f{d}p-\e}_{p,r})}+\|u^h\|_{\wt{L}^1_t(\dot{B}^{1+\f{d}p}_{p,r})}\bigr)\\
&\leq \frak{C}_1\bigl(\mu\|a_0\|_{L^\infty\cap
B^{\f{d}q+\e}_{q,\infty}}+\|u_0^h\|_{\dot{B}^{-1+\f{d}p-\e}_{p,r}\cap\dot{B}^{-1+\f{d}p}_{p,r}}\bigr)\exp\Bigl\{\f{\frak{C}_2}{\mu^{2r}
}\bigl(\|u_0^d\|_{\dot{B}^{-1+\f{d}p-\e}_{p,r}\cap
\dot{B}^{-1+\f{d}p}_{p,r} }^{2r}\Bigr\}
\end{split}
\eeno for $t\leq \frak{T}$ and some positive constants $\frak{C}_1,
\frak{C}_2$ which depends on $\bar{c},$  $c_1$ and $\e.$ In
particular, if we take $C_{r,\ep}$ large enough and $c_0$
sufficiently small in \eqref{small2}, the above inequality implies
that for $\d$ given by \eqref{small2} \beno
\begin{split}
\|u^h&\|_{\wt{L}^\infty_{t}(\dot{B}^{-1+\f{d}p-\e}_{p,r})}+\|u^h\|_{\wt{L}^\infty_{t}(\dot{B}^{-1+\f{d}p}_{p,r})}
+\mu\bigl(\|a_0\|_{L^\infty}+\|a\|_{\wt{L}^\infty_t(
B^{\f{d}q+\f\e2}_{q,\infty})}\\
&+\|u^h\|_{\wt{L}^1_t(\dot{B}^{1+\f{d}p-\e}_{p,r})}+\|u^h\|_{\wt{L}^1_t(\dot{B}^{1+\f{d}p}_{p,r})}\bigr)\leq
C\delta\leq \f{c_1}2\mu\quad\mbox{ for}\quad t\leq \frak{T},
\end{split} \eeno which contradicts with \eqref{5.16}. Whence we conclude
that $\frak{T}=T^\ast=\infty,$ and there holds \eqref{5.20}. This
completes the proof of Theorem \ref{thm1.2}.\ef

\renewcommand{\theequation}{\thesection.\arabic{equation}}
\setcounter{equation}{0}
\appendix

\setcounter{equation}{0}
\section{Littlewood-Paley analysis}\label{abbendixb}

The proof  of Theorem \ref{thm1.2} requires Littlewood-Paley
decomposition. Let us briefly explain how it may be built in the
case $x\in\R^d$ (see e.g. \cite{BCD}). Let $\varphi$ be a smooth
function  supported in the ring $\mathcal{C}\eqdefa \{
\xi\in\R^d,\frac{3}{4}\leq|\xi|\leq\frac{8}{3}\}$ and  $\chi$ be a
smooth function supported in the ball $\cB\eqdefa \{ \xi\in\R^d,\,
|\xi|\leq\frac{4}{3}\}$ such that
\begin{equation*}
 \chi(\xi)+\sum_{q\geq 0}\varphi(2^{-q}\xi)=1\quad\hbox{for all}\quad \xi\in\R^d\quad
  \mbox{and}\quad \sum_{j\in\Z}\varphi(2^{-j}\xi)=1 \quad\hbox{for}\quad \xi\neq 0.
\end{equation*}
Then for $u\in{\mathcal S}'(\R^d),$ we set
\begin{equation}\label{dydic}
\begin{split}
&\forall\ j\in\Z,\quad \dot\Delta_j
u\eqdefa\varphi(2^{-j}\textnormal{D}) u\hspace{1cm}\mbox{and}
\hspace{1cm} \dot S_j u\eqdefa \chi(2^{-j}D)u,\\
&\forall\ q\in\N,\quad \Delta_q u=\varphi(2^{-q}\textnormal{D})
u\hspace{1cm}\mbox{and} \hspace{1cm} S_q u=\sum_{\ell\leq
q-1}\Delta_{\ell}u,\quad \Delta_{-1} u=\chi(\textnormal{D}) u.
\end{split}
\end{equation}
We have the formal decomposition
\begin{equation}\label{decom}
u=\sum_{j\in\Z}\dot\Delta_j \,u,\quad\forall\,u\in {\mathcal
{S}}'(\R^d)/{\mathcal{P}}[\R^d]\ \ \ \mbox{and}\ \ \
u=\sum_{q\geq-1}\Delta_q \,u,\quad\forall\,u\in {\mathcal
{S}}'(\R^d),
\end{equation}
where ${\mathcal{P}}[\R^d]$ is the set of polynomials. Moreover, the
Littlewood-Paley decomposition satisfies the property of almost
orthogonality:
\begin{equation}\label{Pres_orth}
\dot\Delta_j\dot\Delta_\ell u\equiv 0 \quad\mbox{if}\quad|
j-\ell|\geq 2 \quad\mbox{and}\quad\dot\Delta_j(\dot
S_{\ell-1}u\dot\Delta_\ell u) \equiv 0\quad\mbox{if}\quad| j-q|\geq
5.
\end{equation}

We recall now the definition of homogeneous Besov spaces and
Bernstein type inequalities from \cite{BCD}.
\begin{defi}\label{def1.1} {\sl  Let $(p,r)\in[1,+\infty]^2,$ $s\in\R$ and $u\in{\mathcal
S}_h'(\R^3),$ which means that $u\in\cS'(\R^3)$ and
$\lim_{j\to-\infty}\|\dot{S}_ju\|_{L^\infty}=0,$ we define
\beno\begin{split} & \|u\|_{B^s_{p,r}}\eqdefa\Big(2^{qs}\|\Delta_q
u\|_{L^{p}}\Big)_{\ell ^{r}}\quad\mbox{and}\quad \dot
B^s_{p,r}(\R^d)\eqdefa \big\{u\in{\mathcal S}_h'(\R^d)\;\big|\; \|
u\|_{\dot B^s_{p,r}}<\infty\big\}.\end{split}\eeno}
\end{defi}

Inhomogeneous Besov spaces $B^s_{p,r}(\R^d)$ can be defined in a
similar way so that
$$
\|u\|_{B^s_{p,r}}\eqdefa\Bigl(2^{qs}\|\Delta_q
u\|_{L^{p}}\Bigr)_{\ell ^{r}(\N)}.
$$

\begin{lem}\label{lem2.1} {\sl Let $\cB$ be a ball   and $\cC$ a ring of $\R^d.$
 A constant $C$ exists so
that for any positive real number $\d,$ any non negative integer
$k,$ any smooth homogeneous function $\sigma$ of degree $m,$ and any
couple of real numbers $(a, \; b)$ with $ b \geq a \geq 1,$ there
hold
\begin{equation}
\begin{split}
&\Supp \hat{u} \subset \d \mathcal{B} \Rightarrow \sup_{|\alpha|=k}
\|\pa^{\alpha} u\|_{L^{b}} \leq  C^{k+1}
\d^{k+ 3(\frac{1}{a}-\frac{1}{b} )}\|u\|_{L^{a}},\\
& \Supp \hat{u} \subset \d \mathcal{C} \Rightarrow C^{-1-k}\d^{
k}\|u\|_{L^{a}}\leq \sup_{|\alpha|=k}\|\partial^{\alpha} u\|_{L^{a}}
\leq
C^{1+k}\d^{ k}\|u\|_{L^{a}},\\
& \Supp \hat{u} \subset \d\mathcal{C} \Rightarrow \|\sigma(D)
u\|_{L^{b}}\leq C_{\sigma, m} \d^{ m+3(\frac{1}{a}-\frac{1}{b}
)}\|u\|_{L^{a}}.
\end{split}\label{2.1}
\end{equation}}
\end{lem}

We shall frequently use  Bony's decomposition  from \cite{Bo} in the
homogeneous context:
\begin{equation}
\begin{split}
&uv=\dot{T}_u v+\dot{\cR}(u,v)=\dot{T}_u v+\dot{T}_v u+\dot{R}(u,v),
\end{split}\label{bony}
\end{equation}
where
\begin{equation*}
\begin{split}
&\dot{T}_u v\eqdefa\sum_{j\in \mathbb{Z}}\dot S_{j-1}u\dot\Delta_j
v,\qquad
\dot{\cR}(u,v)\eqdefa\sum_{j\in\Z}\dot\Delta_j u \dot S_{j+2}v,\\
&\dot{R}(u,v)\eqdefa\sum_{j\in\Z}\dot\Delta_j u
\widetilde{\dot\Delta}_{j}v\quad  \mbox{with}\quad
\widetilde{\dot\Delta}_{j}v\eqdefa \sum_{|j'-j|\leq
1}\dot\Delta_{j'}v.
\end{split}
\end{equation*}

In  order to obtain a better description of the regularizing effect
of the transport-diffusion equation, we need to use Chemin-Lerner
type spaces $\widetilde{L}^{\lambda}_T(\dot B^s_{p,r}(\R^d)).$
\begin{defi}\label{chaleur+}
Let $(r,\lambda,p)\in[1,\,+\infty]^3$ and $T\in (0,\,+\infty]$. We
define $\widetilde{L}^{\lambda}_T(\dot B^s_{p\,r}(\R^d))$ as the
completion of $C([0,T]; \,\cS(\R^d))$ by the norm
$$
\| u\|_{\widetilde{L}^{\lambda}_T(\dot B^s_{p,r})} \eqdefa
\Big(\sum_{j\in\Z}2^{jrs} \Big(\int_0^T\|\dot\Delta_j\,u(t)
\|_{L^p}^{\lambda}\, dt\Big)^{\frac{r}{\lambda}}\Big)^{\frac{1}{r}}
<\infty.
$$
with the usual change if $r=\infty.$  For short, we just denote this
space by $\widetilde{L}^{\lambda}_T(\dot B^s_{p,r}).$
\end{defi}

As one can not use Gronwall's inequality in the Chemin-Lerner type
spaces,  we \cite{PZ1, PZ2} introduced  weighted Chemin-Lerner norm
in the context of anisotropic Besov spaces. To prove \ref{thm1.2},
we need the following version of weighted Chemin-Lerner in the
context of isentropic Besov spaces:

\begin{defi}\label{defa.3}
Let $(r, p)\in[1,\,+\infty]^2$ and $T\in(0,\,+\infty]$. Let $0\leq
f(t)\in L^1(0,T).$ We define the norm
$\wt{L}_{T,f}^1(\dot{B}^s_{p,r})$ as \beno
\|u\|_{\wt{L}_{T,f}^1(\dot{B}^s_{p,r})}\eqdefa
\Bigl(\sum_{j\in\Z}2^{jrs}\|f(t)\dot\Delta_j\,u(t)
\|_{L^1_T(L^p)}^{r}\Bigr)^{\frac{1}{r}}. \eeno
\end{defi}

\noindent {\bf Acknowledgments.} Part of this work was done when M.
Paicu was visiting Morningside Center of the Chinese Academy of
Sciences in the Spring of 2012. We appreciate the hospitality of MCM
and the financial support from the Chinese Academy of Sciences. P.
Zhang is partially supported by NSF of China under Grant 10421101
and 10931007, the one hundred talents' plan from Chinese Academy of
Sciences under Grant GJHZ200829 and innovation grant from National
Center for Mathematics and Interdisciplinary Sciences.

\end{document}